\documentclass[oneside]{amsart}

\pdfoutput=1

\usepackage[hmarginratio={1:1}]{geometry}
\usepackage[shortlabels]{enumitem}
\usepackage{mathrsfs}
\usepackage[colorlinks,allcolors=blue!50!black,bookmarksdepth=2]{hyperref}

\usepackage{tikz-cd}
\usepackage{soul}
\usepackage{amssymb}
\usepackage{dutchcal}
\usepackage[style=alphabetic,
urldate=iso,
doi=true,
isbn=false,
maxnames=99,
maxalphanames=99,
backref=true]{biblatex}
\usepackage{datetime2}
\usepackage{mathtools}
\usepackage{amstext}
\usepackage{graphicx}
\usepackage{calc}

\AtEveryBibitem{%
  \clearfield{eprintclass}%
  \iffieldundef{eprint}
    {}
    {\clearfield{url}}%
  \iffieldundef{doi}
    {}
    {\clearfield{url}}%
}

\newcommand\abs[2][]{#1\lvert{#2}#1\rvert}
\newcommand\absbig[1]{\big\lvert{#1}\big\rvert}
\newcommand\alttext{\texorpdfstring}
\newcommand\AND{\text{ and }}
\newcommand\Ar{\operatorname{Ar}}
\newcommand\BiCone{\mathrm{BiCone}}
\newcommand\bo{\mathrm{bo}}
\newcommand\br[1]{\langle{}#1\rangle{}}
\newcommand\brbig[1]{\big<{#1}\big>}
\newcommand\cat[1]{\mathbf{#1}}
\newcommand\Cat{\mathcal{Cat}}
\newcommand\cA{\cat{A}}
\newcommand\cC{\cat{C}}
\newcommand\CC{\mathcal{C}}
\newcommand\cD{\cat{D}}
\newcommand\CD{\mathcal{D}}

\newcommand\cE{\cat{E}}
\newcommand\cF{\mathcal{F}}

\newcommand\Cocone{\mathrm{Cocone}}
\newcommand\cod{\mathrm{cod}}
\newcommand\Cone{\mathrm{Cone}}
\newcommand\Cong{\operatorname{Cong}}
\newcommand\const{\mathrm{const}}
\newcommand\coop{\mathrm{coop}}
\newcommand\co{\mathrm{co}}
\newcommand\cU{\mathcal{U}}
\newcommand\defeq{\mathrel{\vcenter{\baselineskip0.5ex \lineskiplimit0pt \hbox{.}\hbox{.}}} =}
\newcommand\defword[1]{\textbf{#1}}
\newcommand\DFib{\operatorname{DFib}}
\newcommand\DOF{\mathbf{DOF}}
\newcommand\dom{\mathrm{dom}}
\newcommand\dto{\dashrightarrow}
\newcommand\El{\mathrm{El}}
\newcommand\el{\mathrm{el}}
\newcommand\ev{\mathrm{ev}}
\newcommand\FinCat{\mathcal{FinCat}}
\newcommand\Fun{\operatorname{Fun}}
\newcommand\gpd{\mathrm{gpd}}
\newcommand\Gr{\mathrm{El}}

\newcommand\htox[2][]{\xhookrightarrow[#1]{#2}}
\newcommand\hto{\hookrightarrow}
\newcommand\id{\mathrm{id}}
\newcommand\inv{\mathrm{inv}}
\newcommand\iso{\mathrm{iso}}
\newcommand\I{^{-1}}
\newcommand\loccitnodot{\textit{loc.\@ cit}}
\newcommand\loccit{\textit{loc.\@ cit.\@}}
\newcommand\Mon{\mathrm{Mon}}
\newcommand\Nv{\mathrm{Nv}}
\newcommand\Ob{\operatorname{Ob}}
\newcommand\OFs{\mathrm{OF}_{\mathrm{s}}}
\newcommand\ol[1]{\overline{#1}}
\newcommand\op{\mathrm{op}}
\newcommand\ot{\leftarrow}
\newcommand\oT{\Leftarrow}
\newcommand\pbig[1]{\big({#1}\big)}
\newcommand\pBig[1]{\Big({#1}\Big)}
\newcommand\pbigg[1]{\bigg({#1}\bigg)}
\newcommand\pbsquare{{\text{\tikz[anchor=base, baseline]{\draw (0em,0em) rectangle (0.7em,0.7em) (0em,0.7em) rectangle (0.35em,0.35em)}}}} 
\newcommand\pb{\ar[rd, phantom, "\lrcorner", pos=0]}
\renewcommand\phi{\varphi}
\newcommand\plbig[1]{\big({#1}\big)}
\newcommand\pow{\mathcal{P}}
\newcommand\ps{\mathrm{ps}}
\newcommand\rstr[2]{{\left.#1\right|_{#2}}}
\newcommand\set[2][]{#1\{{#2}#1\}}
\newcommand\Set{\mathbf{Set}}
\newcommand\Sp{\mathrm{Sp}}
\newcommand\Sub{\mathrm{Sub}}
\newcommand\s{\mathrm{S}}
\newcommand\tm{\mathbf{1}}
\newcommand\tocellud[2]{\hspace{-5pt}\begin{tikzcd}[column sep=15pt, ampersand replacement=\&]{}\ar[r, shift left=3pt, "#1"{name=top}]\ar[r, shift right=3pt, "#2"'{name=bot}, start anchor=east, end anchor=west]\&{}\ar[Rightarrow, from=top, to=bot, shorten=0.5pt]\end{tikzcd}\hspace{-5pt}}
\newcommand\tocell{\hspace{-5pt}\begin{tikzcd}[column sep=15pt, ampersand replacement=\&]{}\ar[r, shift left=3pt, ""{name=top}]\ar[r, shift right=3pt, ""'{name=bot}, start anchor=east, end anchor=west]\&{}\ar[Rightarrow, from=top, to=bot, shorten=0.5pt]\end{tikzcd}\hspace{-5pt}}
\newcommand\toix[1]{\let\mytoixlength\relax\newlength{\mytoixlength}\setlength{\mytoixlength}{\maxof{\widthof{\ensuremath{#1}}}{18pt}}\!\!\!\begin{tikzcd}[ampersand replacement=\&, column sep=\the\mytoixlength]{}\ar[r, "#1", "\sim"']\&{}\end{tikzcd}\!\!\!}
\newcommand\toi{\xrightarrow{_\sim}}
\newcommand\ToT{\Leftrightarrow}
\newcommand\tox[2][]{\xrightarrow[#1]{#2}}
\newcommand\toy{{\overset{{}_{\sim}}{\to}}}
\newcommand\To{\Rightarrow}
\newcommand\tto{\rightarrowtail}
\newcommand\twocell[6][0pt]{\begin{tikzcd}[row sep=0pt, column sep=10pt, ampersand replacement=\&]\&\ar[dd, shorten >=2pt, shorten <=-2pt, Rightarrow, "\ #6", pos=0.4]\&[#1]\\[-5pt]#2\ar[rr, "#4", bend left=30pt, start anchor=east, end anchor=west, shift left=7pt]\ar[rr, "#5"', bend right=30pt, start anchor=east, end anchor=west, shift right=7pt]\&\&#3\\\&{}\&\end{tikzcd}}
\newcommand\und[1]{\underline{#1}}
\newcommand\wc[1]{\widecheck{#1}}
\newcommand\wh[1]{\widehat{#1}}
\newcommand\wt[1]{\widetilde{#1}}
\newcommand\wwc[1]{\wc{\wc{#1}}}
\newcommand\wwh[1]{\widehat{\widehat{{#1}\,}}\!\!}
\newcommand\xot[2][]{\xleftarrow[#1]{#2}}
\newcommand\Z{\mathbb{Z}}

\makeatletter
\newcommand\rabel[1]{#1\phantomsection\protected@edef\@currentlabel{#1}}
\makeatother

\DeclareFontFamily{U}{mathx}{\hyphenchar\font45}
\DeclareFontShape{U}{mathx}{m}{n}{
      <5> <6> <7> <8> <9> <10>
      <10.95> <12> <14.4> <17.28> <20.74> <24.88>
      mathx10
      }{}
\DeclareSymbolFont{mathx}{U}{mathx}{m}{n}
\DeclareFontSubstitution{U}{mathx}{m}{n}
\DeclareMathAccent{\widecheck}{0}{mathx}{"71} 

\theoremstyle{definition}
\newtheorem{defn}{Definition}[subsection]
\newtheorem{rmk}[defn]{Remark}

\theoremstyle{remark}

\theoremstyle{plain}
\newtheorem{thm}[defn]{Theorem}
\newtheorem{propn}[defn]{Proposition}
\newtheorem{cor}[defn]{Corollary}
\newtheorem{lem}[defn]{Lemma}

\let\oldproofname=\proofname
\renewcommand{\proofname}{\rm\bf{\oldproofname}}

\title{Internal 1-topoi in 2-topoi}
\author{Joseph Helfer}
\date{}

\dedicatory{To M.\ Makkai}

\begin{document}
\begin{abstract}
  We further develop the notion of elementary 2-topos, introduced by Weber, by proposing certain new axioms.
  We show that in a 2-category $\CC$ satisfying these axioms, the ``discrete opfibration (DOF) classifier'' $\s$ is always an internal elementary 1-topos, in an appropriate sense.
  The axioms introduced for this purpose are closure conditions on the DOFs having ``$\s$-small fibres''.
  Among these closure conditions, the most interesting one asserts that a certain DOF, analogous to the ``subset fibration'' over $\Set$, has small fibres.

  The remaining new axioms concern ``groupoidal'' objects in a 2-category, which are seen to play a significant role in the general theory.
  We prove two results to the effect that a 2-category $\CC$ satisfying these axioms is ``determined by'' its groupoidal objects: the first shows that $\CC$ is equivalent to a 2-category of internal categories built out of groupoidal objects, and the second shows that the groupoidal objects are \emph{dense} in $\CC$.
\end{abstract}

\maketitle

\setcounter{tocdepth}{1}
\tableofcontents

\section{Introduction}
The notion of \emph{elementary 2-topos} stands in relation to the 2-category $\Cat$ of (large) categories as that of \emph{elementary topos} does to the category of sets.

That is, an elementary 2-topos should be a 2-category $\CC$ satisfying a list of first-order (i.e., ``elementary'') axioms satisfied by $\Cat$ which, ideally, suffice to allow one to carry out category theory inside of $\CC$.
We take as our starting point Weber's paper \cite{weber-2-toposes} which, to our knowledge, marks the first use of the term elementary 2-topos (though Grothendieck 2-topoi were introduced earlier in \cite{street-two-sheaf}); see \S\ref{subsec:historical} below for further historical remarks.
We should note right away, however, that there is no settled definition of elementary 2-topos, and in fact one of our aims here is to propose certain new axioms that we believe should be part of the definition.

Hereafter, if we write 2-topos or topos, we will always mean elementary, unless stated otherwise.

Three of Weber's axioms parallel those of an elementary 1-topos: a 2-topos $\CC$ is required to have (i) finite limits, (ii) exponentials, and (iii) a so-called \emph{discrete opfibration (or DOF) classifier} $\s$.
There is then one additional axiom which is (iv) the existence of a \emph{duality involution}, which we shall return to below.

The most interesting of these axioms is the existence of the DOF classifier $\s$.
This plays an analogous role in a 2-topos $\CC$ to that played by the subobject classifier in a 1-topos, and in the archetypical example $\CC = \Cat$ it is given by the object $\Set \in \Cat$.

Here, the ``universal monomorphism'' $\tm \to \Omega$ carried by a subobject classifier is replaced by a ``universal DOF'' $\s_* \to \s$, which in $\Cat$ is the forgetful functor $p \colon \Set_* \to \Set$ from the category of pointed sets.

\subsection{The question of size}
Let us recall the universal property of the functor $p \colon \Set_* \to \Set$, which is thus taken as the defining property of a DOF classifier.

Given any category $\cC$ and functor $F \colon \cC \to \Set$, we can form the (strict) pullback
\[
  \begin{tikzcd}
    \El(F)\ar[r, ""]\ar[d, "p'"']\pb&\Set_*\ar[d, "p"]\\
    \cC\ar[r, "F"]&\Set,
  \end{tikzcd}
\]
which is given by the \emph{category of elements} $\El(F)$ with object set
\[
  \Ob\El(F) = \set{(A,x) \mid A \in \cC, x \in FA},
\]
and $p'$ is defined on objects by $p(A,x) = A$.
The functor $p'$ is a DOF (see \S\ref{subsec:dofs}) and moreover it has \emph{small fibres}, meaning that each fibre $p\I(A) \cong F(A)$ is a small set, as opposed to a proper class.

In this way, we obtain a functor from the functor category $\Fun(\cC,\Set)$ to the category of DOFs over $\cC$ with small fibres, and the universal property in question says that this is an equivalence for all $\cC$.

Now, when we seek to reproduce this universal property for a morphism $p \colon \s_* \to \s$ in a general 2-category $\CC$, there is no problem internalizing the concept of DOF.
However, there is no natural notion of ``DOF with small fibres''; on the contrary, the notion of ``small'' depends on a chosen universe, i.e., category of sets, and so it would be circular to try to have the universal property of the DOF classifier $\s$ depend on the notion of smallness.

Instead, by definition, for $p$ to be a DOF classifier, we (following \cite{weber-2-toposes}) only demand that for any $A \in \CC$, the functor $\CC(A,\s) \to \DOF(A)$ to the category of internal DOFs over $A$ be \emph{fully faithful}.
We then say that a DOF over $A$ is ``$\s$-small'' if it is in the essential image of this functor.

Note however, that with this condition alone, it could well happen that there simply \emph{are no $\s$-small DOFs}---for example, the empty category is a DOF classifier in $\Cat$.
We remedy this by imposing certain \emph{closure conditions} on the class of $\s$-small DOFs (see \S\ref{subsec:dofs}); we call a DOF classifier satisfying these conditions \emph{plentiful}.

There is a helpful analogy between this situation and that of class-set theories such as those of Von Neumann-Gödel-Bernays and Morse-Kelley.
There, one has a primitive notion of a given class being \emph{small} (i.e., a set), and one then postulates that the small classes are \emph{closed} under certain operations (union, power set, and so on).

The closure conditions for small DOFs are also very much in the spirit of the ``axioms for small maps'' from \cite{joyal-moerdijk-ast}.

We also note that the assumption that \emph{every} DOF is $\s$-small would correspond to the assumption in class-set theory that every class is a set, which of course leads to Russell's paradox.
The original motivation for this project was to prove that this assumption is contradictory in the present context as well, or more precisely that it implies that the ambient 2-topos $\CC$ is the trivial 2-category; this will be done in a subsequent paper \cite{helfer-paradoxes}.

\subsection{\alttext{$\s$}{S} as an internal topos; groupoids}\label{subsec:intro-internal-topos-groupoids}
Our main goal is to prove that in a 2-topos, any plentiful DOF classifier $\s$ is an \emph{internal topos} (see \S\ref{subsec:internal-power}), in much the same way as one shows that the subobject classifier in a 1-topos is an internal Heyting algebra.

There is an immediate obstacle to formulating this notion.
One would naturally like to say that for an internal topos $\s \in \CC$, each hom-category $\CC(A,\s)$ is a topos.
However, this already fails in $\Cat$: it is not true that for a topos $\cC$, each functor category $\Fun(\cA,\cC)$ is again a topos.
This \emph{does} hold, however, when $\cA$ is a groupoid.

Thus, for this and other reasons, we are led to consider \emph{groupoids} (or \emph{groupoidal objects}) in $\CC$ (that is, objects $G \in \CC$ for which each hom category $\CC(A,G)$ is a groupoid), and in particular to the natural axiom that each object $A \in \CC$ has an associated \emph{core}, or maximal subgroupoid $A^\iso$.

We also introduce an additional axiom concerning groupoids---akin to the effectiveness of internal equivalence relations in a 1-topos and of internal $\infty$-groupoids in an $\infty$-topos \cite[\S6.1]{lurie-htt}---which says that each object in $\CC$ can be recovered from a certain diagram of groupoids, and about which we will say more below in \S\ref{subsec:groupoids}.

Though this axiom is not needed in the proof that $\s$ is an internal topos, it nonetheless ends up playing an important role, and a significant part of this paper is devoted to studying its consequences and the attendant notions (\emph{viz.}, \S\S3~and~6, and the appendix).

\subsection{The power set functor}\label{subsec:intro-power-set}
The proof that a plentiful DOF-classifier $\s \in \CC$ is an internal topos primarily involves showing that $\CC(A,\s)$ is a topos for each groupoid $A \in \CC$, which is the same as saying that the subcategory $\DOF_\s(A) \subset \DOF(A)$ consisting of $\s$-small DOFs is a topos.
The existence of finite limits in $\DOF_\s(A)$ follows readily from the assumed closure conditions for $\s$-small DOFs.

The most interesting part of the proof is showing that $\DOF_\s(A)$ has power objects, and this correspondingly relies on the most interesting condition in the definition of a plentiful DOF classifier.

The crucial observation is that the DOF corresponding to the power set functor\footnote{%
  In general, there are two power set functors, a covariant and contravariant one, but they agree on isomorphisms, and we are here using their common restriction, as we do throughout the paper.
  We address the question of the full covariant and contravariant power functors in a 2-topos in Appendix~\ref{sec:groupoid-stuff}.
}
$\pow \colon \Set^\iso \to \Set$ is the subset fibration $\cat{Sub}^\iso\tox{\cod}\Set^\iso$, where $\cat{Sub}\subset\cat{Set^\to}$ is the full subcategory consisting of subset inclusions.
Hence, the corresponding condition on the plentiful DOF classifier $\s$ is that the analogous DOF $\Mon(\s)^\iso \to \s^\iso$ be $\s$-small.

We should mention at this juncture that throughout this paper, we work in a \emph{strict} context: that is, with strict 2-categories, strict 2-categorical limits, and so on.
However, this is merely a matter of convenience, and we take the position that the ``final'' notion of 2-topos should be in a completely weak context.
This is in keeping with the prevailing overall attitude toward higher category theory (see e.g. \cite{makkai-categorical-foundation}), and is also relevant to our reproduction of set-theoretic paradoxes in a 2-topos in \cite{helfer-paradoxes}, as we will want to know that this can be done under the weakest possible logical assumptions.

The reason we mention this now is that the fibration $\Mon(\s)^\iso \to \s^\iso$ we can construct in an arbitrary 2-topos is not directly analogous to the subset fibration, but rather, as the notation indicates, to the \emph{monomorphism fibration} $\Mon(\Set)^\iso \to \Set^\iso$, where $\Mon(\Set) \subset \Set^\to$ is the full subcategory consisting of monomorphisms.
The latter fibration, however, is not a DOF, but rather a (semi-)weak version of a DOF, which we call a \emph{setoidal opfibration}, or \emph{SOF}, since the fibres are not discrete categories, but rather \emph{setoids}, i.e., categories equivalent to discrete ones.
Thus, in this particular case, we must demand the $\s$-smallness not just of a DOF, but of a certain SOF (see \S\ref{subsec:sofs}).

In the ``fully weak'' context, such a distinction wouldn't exist, since we would only have the fully weak notion of discrete fibration available (namely a ``discrete'' variant of the \emph{Street (op)fibrations} of \cite[Definition~3.1]{johnstone-fibrations}).
Hence we see that even in the present ``fully strict'' context, we are nonetheless forced into considering this semi-weak notion.
This also brings out the interesting relationship between the issue of size and that of strictness, since the semi-weak fibration $\Mon(\Set)^\iso \to \Set^\iso$ has fibres that are only \emph{essentially} small.

\subsection{Groupoids}\label{subsec:intro-groupoids}
We now return to Weber's last axiom: the existence of a \emph{duality involution} on $\CC$, which is a 2-functor $\CC^\co \to \CC$ which plays in a general 2-topos the role of the 2-functor $\op \colon \Cat^\co \to \Cat$ in $\Cat$.
This is essential to the purposes of \cite{weber-2-toposes}, since the main emphasis there is the Yoneda embedding $A \to \s^{A^\op}$.

We, however, omit this axiom from our list, and the reason is that we can \emph{prove} the existence of a duality involution from our assumptions concerning groupoids.

In the same way that in a finitely complete category, one can form the kernel pair of any morphism and obtain an in internal equivalence relation, and in a finitely complete quasicategory, one can form the ``Čech nerve'' of a morphism and obtain an internal $\infty$-groupoid (Kan complex) \cite[\S6.1.2]{lurie-htt}---so in a (suitably complete) 2-category $\CC$, one can perform a similar construction on a given morphism and obtain an internal category in $\CC$.%
\footnote{%
  Below, we use the phrase ``internal double category'' rather than ``internal category'' for reasons we explain in Definition~\ref{defn:int-doub-cat}.
}
This is the basis for the notion of regular and exact 2-categories studied in \cite{street-two-sheaf,makkai-duality-and-definability,bourke-garner-2-reg-ex}.

What's more, given any object $A \in \CC$, we have a canonically associated morphism, namely the inclusion of $A^\iso \to A$ of the core, and performing the above construction, we obtain an internal category associated to $A$, which moreover consists level-wise of groupoids, and which we (following \cite[\S6.1.2]{lurie-htt}) call the \emph{nerve} of $A$.
Our axiom now says that this internal category is \emph{effective}, meaning that $A$ can be recovered from it as a certain weighted colimit; and furthermore, that any internal category satisfying a certain extra condition arises as a nerve.

The result is that $\CC$ turns out to be equivalent to a 2-category of internal categories in $\CC$ which are level-wise groupoidal (Theorem~\ref{thm:internal-cats-equivalence}).
In particular, we obtain the duality involution on this equivalent 2-category simply by exchanging the source and target maps of a given groupoid (which is just the way that the duality involutions in the examples in \cite{weber-2-toposes} are obtained).

As an aside, we note that a duality involution is not determined uniquely in any sense, and therefore had to be assumed as an extra \emph{structure} on a 2-topos in \cite{weber-2-toposes}.
(There is an interesting related circumstance with respect to the DOF classifier axiom; we return to this in \S\ref{sec:axioms}.)
Since, in the present axiomatization, we \emph{derive} the existence of the duality involution, this is no longer necessary.
This is in keeping with the idea that, as is the case with 1-topoi, a 2-topos should be a 2-category with certain \emph{properties}, but no extra \emph{structure}.

\subsection{2-Yoneda, and sketches}
In addition to the just mentioned Theorem~\ref{thm:internal-cats-equivalence}, we prove a second result in \S\ref{sec:suff-of-groupoids} to the effect that $\CC$ is ``determined by its groupoids'', having to do with the 2-dimensional Yoneda embedding.

Though we do not be appeal to it directly, the 2-dimensional Yoneda embedding motivates much of what we do in this paper.
Namely, there are several definitions of 2-categorical notions that we give ``representably''---for example, an object $X \in \CC$ is defined to have a certain property if each category $\CC(A,X)$ with $A \in \CC$ has that property.
Often, this amounts to saying that $X$ represents a certain 2-functor $\CC^\op \to \Cat$---which thus determines $X$ up to isomorphism by the 2-Yoneda embedding.

However, as we alluded to in \S\ref{subsec:intro-internal-topos-groupoids}, there are certain universal properties, such as that of $\s \in \CC$ being an internal topos, which only make reference to the categories $\CC(A,X)$ with $A$ \emph{groupoidal}.
Now the result in question (Theorem~\ref{thm:gpd-yoneda-rstr}) says that this is a ``legitimate'' procedure in the sense that the groupoids in $\CC$ form a \emph{dense} sub-2-category of $\CC$, meaning that the 2-Yoneda embedding remains an embedding if instead of 2-functors $\CC^\op \to \Cat$, we consider 2-functors $(\CC_\gpd)^\op \to \Cat$ defined on the full sub-2-category of $\CC$ consisting of groupoids.

This result is quite useful, and seems somehow more fundamental than Theorem~\ref{thm:internal-cats-equivalence}.
In \S\ref{subsec:opposites}, we apply it to give a ``representable'' definition of the opposite of an object, and in Appendix~\ref{sec:groupoid-stuff}, we apply it to construct the internal power object and exponential functors of an internal elementary topos.

These constructions necessitate a certain ancillary result that we prove in Appendix~\ref{sec:pbt-skectches}, and which is of independent interest.
Namely, we show that for a \emph{finite-limit sketch} $J$ and an object $X$ in a 2-category, we can (under certain assumptions) define the object $X^J$ of ``models of $J$ in $X$''.
In \cite{helfer-paradoxes}, further use will be made of these objects $X^J$, as well as generalizations of them.

\subsection{On the axioms}
We give our list of 2-topos axioms in \S\ref{sec:axioms}.

We do not necessarily want to suggest that this list be a final definition of 2-topos.
In fact, our general orientation toward selection of axioms is a liberal one: any reasonable property of $\Cat$ that is expressible in the general language of 2-categories should either be included as an axiom or at least be derivable from them.
Or perhaps better said: one should freely make such assumptions as one needs them, which is precisely the basis on which the present list was selected.

We note that in the case of 1-topoi, it so happens that one can get away with an extremely small set of axioms (finite limits, cartesian closed, and subobject classifier), from which one can derive everything else that one wants, for example, the existence of finite colimits.
However, had it so turned out that one could not \emph{derive} these properties from the topos axioms as they stand, we hold that it would have been perfectly reasonable to include them as axioms.
Indeed, \cite{lawvere-etcs} assumes the existence of finite colimits (as it was not yet known that it was derivable from the other axioms), as well as several other things.

On the other hand, it is important to note that \cite{lawvere-etcs} \emph{also} includes axioms that do \emph{not} hold in a general topos, most prominently that the terminal object is a generator, which axiom excludes most Grothendieck topoi.
Thus, we see that some care must be taken with the above principle ``freely assume any properties that hold in the motivating example'': it should be qualified by saying that we should make the \emph{weakest} assumptions possible that allow us to develop the general theory.

In any case, the axioms presented below seem to us to be reasonable to assume, and it remains an interesting question whether any of them is redundant, and also how much can be derived from them---and especially whether one can show the existence of finite (bi)colimits.

We note the crucial fact that, as with elementary 1-topoi, these axioms are all \emph{elementary}: they can be formulated in the first order language of 2-categories.

\subsection{Related work}\label{subsec:historical}
We mention some important related works, though this list is not exhaustive---especially as the idea of doing formal category theory has been central to the notion of 2-categories from the very beginning (see, e.g., \cite{gray-formal-category-theory}).

The idea of using the totality---in fact, the \emph{1-category}---of categories as a foundational system appeared already in Lawvere's paper \cite{lawvere-cat-of-cats}, not long after his paper \cite{lawvere-etcs} on the elementary theory of the category of sets, which eventually led to elementary topos theory, and at the end of which paper he in fact already alludes to the category of categories as a foundation.
Since \cite{lawvere-cat-of-cats} takes place in the 1-categorical setting---which is, so to speak, even more strict than the strict setting of this paper---certain things are different, and for example, the notion of DOF classifier is not available (which notion, however, according to \cite{weber-2-toposes}, is also due to Lawvere).
Two significant ideas that are already present, however, are (i) the special role played by the object $\Set \in \Cat$, and (ii) that of recovering each object as a certain colimit of an internal category (consisting of discrete, rather than groupoidal, objects), and using this to construct the opposite of an object.

The notion of a 2-category with a DOF classifier---or rather, with a presheaf-category functor---was introduced under the name \emph{elementary cosmos} in \cite{street-elementary-cosmoi-i,street-cosmoi-of-internal-cats}, which serves as a basis for \cite{weber-2-toposes}.
As mentioned, the idea of the DOF classifier is originally due to Lawvere, and is also discussed in \cite{gray-the-categorical-comprehension-scheme}.

It is observed in \cite{street-cosmoi-of-internal-cats} that any elementary cosmos has an associated \emph{Yoneda structure} in the sense of \cite{street-walters-yoneda-structure}, and this circumstance is also the main focus of \cite{weber-2-toposes}.
In a similar vein to the notion of a plentiful DOF classifier and to the small maps of \cite{joyal-moerdijk-ast}, a Yoneda structure also axiomatizes a class of ``small'' functors, namely those functors $F \colon \cC \to \cD$ with small hom sets $\cD(Fx,y)$.
By taking $F = \id_\cC$, this includes the notion of local smallness.
Crucially, one can recover from this the notion of (essential) smallness by demanding that both $\cC$ and $\Set^{\cC^\op}$ be locally small, as shown in \cite{freyd-street-size-of-categories} (see also Remark~\ref{rmk:cocompleteness}).

Though Weber does not consider closure conditions for $\s$-small DOFs in \cite{weber-2-toposes}, he has stated (in \cite{weber-ncafe-comment}) that the definition of 2-topos given there should be considered provisional, and that in particular further conditions should be imposed on $\s$, for example that it be \emph{cocomplete} in the sense studied in \cite{weber-2-toposes}.

We also note that Weber does show that the DOF classifier $\s$ has finite products (under suitable assumptions), and also that it is cartesian closed in a weak sense (under the aforementioned cocompleteness assumption); see Remark~\ref{rmk:cocompleteness}.

While this paper was nearing completion, it came to my attention that Mike Shulman has developed many of the same ideas as are presented here, as well as some of those we will pursue in \cite{helfer-paradoxes}, and written them up on his web page \cite{shulman-nlab-page}.
Among other things, he \emph{does} consider closure conditions for $\s$-small DOFs, as well as the idea of reconstructing a general object out of groupoids, the question of whether $\s$ is an internal topos, and the question of whether it is consistent for all DOFs to be $\s$-small.

\subsection{Acknowledgements}
This paper began as a project of B.\,Boshuk and M.\,Makkai some time around the year 2000, with the goal of reproducing the set-theoretic paradoxes in a 2-topos, as will be done in \cite{helfer-paradoxes}.

Makkai first told me about this idea in 2016, but I only began working on it seriously in 2022.
He has continued to be involved as I have been working on it, and several of the most important ideas in the paper are due to him.

I would like to thank: David Ayala and Arpon Raksit for helpful discussions; Richard Blute for inviting me to speak on this project at an early stage at the University of Ottawa Logic Seminar; Simon Henry, who attended that talk, for bringing my attention to related ongoing work of Cisinski, Nguyen and Walde in an $(\infty,2)$-categorical context; and John Bourke, who gave helpful feedback on the first version of the paper, and alerted me to the interesting related work \cite{hughes-miranda-et2cc}.

I would also like to thank the anonymous referee for their careful reading and for several suggestions which improved the paper.

\section{Preliminaries}\label{sec:prelim}
In this section and the next section, we recall some basic 2-categorical notions, and then provide the various definitions needed to state our 2-topos axioms in \S\ref{sec:axioms}.

We take as known the basic notions of 2-category theory, for which we refer to, e.g., \cite[\S4.1]{makkai-pare-accessible}.

We use the ``geometric'' order for composition of morphisms: the composite of $A \tox{f} B \tox{g} C$ is denoted $A \tox{fg} C$ (or $f \cdot g$ or $f \circ g$).
To be consistent with this, we also use this order for composition of functions (i.e., morphisms in $\Set$), and of functors and natural transformations.
However, we use the usual order for \emph{application}---thus, for example, for functors $\cC \tox{F} \cD \tox{G} \cE$ and an object $X \in \cC$, we have $(F \circ G)(X) = G(F(X))$.

Given morphisms $f,g \colon X \to Y$ in a 2-category, to indicate that $\alpha$ is a 2-cell $f \to g$, we may sometimes use the notation $\alpha \colon X \tocellud{f}{g} Y$, or even just $\alpha \colon X \tocell Y$ if $f$ and $g$ are clear from context or not relevant.

Given a morphism $f \colon X \to Y$ in a 2-category (or 1-category) $\CC$ and an object $A$, we write $f_*$ and $f^*$ for the functors (or functions) $f_* \colon \CC(A,X) \to \CC(A,Y)$ and $f^* \colon \CC(A,Y) \to \CC(A,X)$.
We may also denote these by $(f_*)_A$ and $(f^*)_A$ for emphasis.

Similarly, given a 2-cell $\alpha \colon X \tocellud{f}{g} Y$, we write $\alpha_*$ (or $(\alpha_*)_A$) for the natural transformation $\alpha_* \colon \CC(A,X) \tocellud{f_*}{g_*} \CC(A,Y)$ given by whiskering with $\alpha$; and $\alpha^*$ is defined similarly.

A morphism $f \colon X \to Y$ in a 2-category $\CC$ is \defword{faithful} or \defword{fully-faithful}, respectively, if $f_* \colon \CC(A,X) \to \CC(A,Y)$ is for each $A \in \CC$.

Almost everything we do in this paper is elementary, and our set-theoretic assumptions are accordingly minimal---with the exception \S\ref{subsec:examples-dof-classifiers}, where we are explicit about these assumptions.
In the few other places in the paper that depend on non-trivial set-theoretic assumptions such as Grothendieck universes, we trust the reader to supply the details.

For a 2-category $\CC$, we write $\abs{\CC}$ for its underlying 1-category.

When we say that a 2-category $\CC$ \emph{locally} has some property $P$, we mean that each hom category $\CC(X,Y)$ satisfies $P$; similarly, a 2-functor $F \colon \CC \to \CD$ locally satisfies $P$ if the functor $F \colon \CC(X,Y) \to \CC(FX,FY)$ satisfies $P$ for all $X,Y \in \CC$.

A 2-functor $F \colon \CC \to \CD$ is \defword{fully faithful} (resp.\ \defword{strictly fully faithful}) if it is locally an equivalence (resp.\ locally an isomorphism).
It is \defword{essentially surjective} if every $Z \in \CD$ is equivalent to some $FX$.
It is an \emph{equivalence} (resp. \defword{strict equivalence}) if it is essentially surjective and fully faithful (resp.\ strictly fully faithful).

We will be making use of \emph{anafunctors} (for which see \cite{makkai-avoiding-choice}) and taking for granted their basic properties, which are analogous those of ordinary functors.
Many of the functors which one comes across in category theory are really first of all anafunctors, which only become functors after applying the axiom of choice.
In such cases, we find it is usually preferable to work with the original anafunctor.

We will also be making use of \emph{ana-2-functors}, which are the obvious 2-dimensional generalization: given 2-categories $\CC,\CD$, an ana-2-functor $F \colon \CC \to \CD$ has for each $X \in \CC$ a set of specifications $\abs{F}(X)$, and for each $X,Y \in \CC$ and specifications $s \in \abs{F}(X)$ and $t \in \abs{F}(Y)$, an ana-functor $F_{s,t} \colon \CC(X,Y) \to \CC(F_sX,F_tY)$.
However, in the cases we consider, $F_{s,t}$ will in fact always be a \emph{functor}, and we will always use the term ana-2-functor in this more restrictive sense.

By a \defword{partially defined functor} or \defword{partial functor} between categories $\cC$ and $\cD$, we mean a functor $\cC' \to \cD$ defined on a \emph{full} subcategory $\cC' \subset \cC$.
Given a functor $G \colon \cD \to \cC$, by a \defword{partially defined left (resp. right) adjoint} to $G$, we mean a partial functor $F \colon \cC' \to \cD$ determined by the choice of a universal arrow from $X$ to $G$ (resp.\ from $G$ to $X$) for each $X \in \cC'$, in the sense of \cite[Iv.1~Theorem~2]{maclane-categories}.
Partially defined (adjoint) anafunctors are defined similarly.

We say that a functor \( F \colon \cC \to \cD \) is an \defword{isomorphism onto} a subcategory \( \cD' \subset \cD \) if it induces an isomorphism \( \cC \to \cD' \), or in other words, if it is injective on objects and on morphisms and has image \( \cD' \).

\subsection{Limits}\label{subsec:limits}
We recall some basic notions of limits in a 2-category $\CC$, and we also fix our terminology, as there are several competing conventions.
The corresponding colimit notions are obtained by dualizing.

We will mainly discuss the strict variants of these notions, and will comment briefly on the weaker notions afterwards.

The general notion of strict limit in a 2-category is that of a strict \emph{weighted} limit of a diagram (that is, a 2-functor) $D \colon J \to \CC$ ($J$ a 2-category) with respect to a \emph{weight} (that is, another 2-functor) $W \colon J \to \Cat$; it is called a \emph{finite} weighted limit if $D$ is a finite 2-category and each $W(d)$ is a finite category.
We will not spell out the general definition of weighted limit, as we will not need it, but see, e.g., \cite[\S5.1.1]{makkai-pare-accessible}.

It is a fact \cite[\S3]{kelly-2-cat-limits} that a category has all finite strict weighted limits (i.e., that it is \emph{finitely strictly complete}) as soon as it has strict pullbacks, a strict terminal object, and strict cotensors with $[1]$, all of which we now discuss.

Given a cospan
\begin{equation}\label{eq:cospan}
  D =
  \left(
    \begin{tikzcd}
      {}&X\ar[d, "f"]\\
      Y\ar[r, "g"]&Z
    \end{tikzcd}
  \right)
\end{equation}
in $\CC$ and an object $A \in \CC$, the category $2\Cone(A,D)$ of \emph{strict cones with vertex $A$} over $D$ has objects commuting squares
\begin{equation}\label{eq:square}
  \begin{tikzcd}
    A\ar[r, "h"]\ar[d, "k"']&X\ar[d, "f"]\\
    Y\ar[r, "g"]&Z,
  \end{tikzcd}
\end{equation}
which we write as a pair $(h,k)$, and morphisms $(\alpha,\beta) \colon (h,k) \to (h',k')$, consisting of 2-cells $\alpha \colon h \to h'$ and $\beta \colon k \to k'$ with $f\alpha = g\beta$, and composition defined in the obvious way.

Given any square $(h,k)$ as above and any $B \in \CC$, there is an evident functor $\CC(B,A) \to 2\Cone(B,D)$, and we say that $(h,k)$ is a \defword{strict pullback of $D$} if this functor is an \emph{isomorphism} for each $B$.

Note that if a given cospan (\ref{eq:cospan}) is known to have a strict pullback, then in order to check that (\ref{eq:square}) is a strict pullback, it suffices to check this in the underlying 1-category $\abs{\CC}$ of $\CC$, i.e., that the functor $\CC(B,A) \to 2\Cone(B,D)$ is a \emph{bijection on objects} for each $B \in \CC$.

\defword{Strict products} are defined in the same way by retaining $X$ and $Y$ in the above but omitting any reference to $(Z,f,g)$.
We say that $X \in \CC$ is a \defword{strict terminal object} if the category $\CC(A,X)$ has a single morphism for all $A \in \CC$.

Next, given a category $J$ and an object $X \in \CC$, a \defword{strict cotensor of $X$ with $J$} is an object $X^J \in \CC$ together with a functor $\ev = \ev_{J,X} \colon J \to \CC(X^J,X)$ satisfying the following universal property:
for each $A \in \CC$, the functor $\CC(A,X^J) \to \CC(A,X)^J$ defined on objects (and morphisms) by
\begin{equation}\label{eq:cotensor-morphism}
  f\mapsto
  \plbig{J \tox{\ev_{J,X}} \CC(X^J,X) \tox{f^*} \CC(A,X)}
\end{equation}
is an \emph{isomorphism}.

In light of this isomorphism, we will often simply identify an object in $\CC(A,X^J)$ with its image in $\CC(A,X)^J$, and refer to a morphism $A \to X^J$ as a ``diagram of shape $J$ in $\CC(A,X)$''.

It is called a \defword{finite} cotensor if $J$ is finite.

An important particular case is when $J$ is the free-standing arrow $J = [1]$, i.e., the category with two objects $0,1$ and one non-identity morphism $0 \to 1$.
In this case, we write $X^\to$ for the cotensor $X^{[1]}$, also called a (strict) \defword{arrow object}; it is characterized by having a universal 2-cell, which we will denote $\partial \colon X^\to \tocellud{\partial_0}{\partial_1} X$.

As mentioned above, if a category has strict pullbacks, a strict terminal object, and strict arrow objects, then it has all finite strict weighted limits, and in particular all strict finite cotensors.

Given any functor $F \colon J \to J'$ between categories and an object $X \in \CC$, if there exists strict cotensors $X^J$ and $X^{J'}$, then there is an induced morphism $X^F \colon X^{J'} \to X^J$ classifying the functor $J \tox{F} J' \tox{\ev} \CC(X^{J'},X)$.
Similarly, a morphism $f \colon X \to Y$ induces a morphism $f^J \colon X^J \to Y^J$ classifying $J \tox{\ev} \CC(X^J,X) \tox{f_*} \CC(X^J,Y)$.

\subsubsection{Pseudo-limits and bilimits}
We now briefly discuss the weaker notions.
A weak cotensor is defined in just the same way as a strict one, except that the comparison map $\CC(A,X^J) \to \CC(A,X)^J$ is only demanded to be an \emph{equivalence}.
Clearly, any strict cotensor is a weak one.

A weak pullback, or \defword{bipullback}, is defined in the same way as a strict pullback, but with two changes.

The first is that the category $2\Cone(A,D)$ of strict cones is replaced by the category $\BiCone(A,D)$ of weak cones, with objects $(h,k,\gamma)$, where $\gamma \colon h f \to k g$ is an invertible 2-cell, and morphisms $(\alpha,\beta) \colon (h,k,\gamma) \to (h',k',\gamma')$, where $\alpha \colon h \to h'$ and $\beta \colon k \to k'$ satisfy $(\alpha f)\gamma' = \gamma(\beta g)$; again, composition is defined in the obvious way.

The second change is, again, that the comparison functors $\CC(B,A) \to \BiCone(B,D)$ are demanded to be equivalences and not isomorphisms.
There is an intermediate notion, called a \defword{pseudo-pullback}, in which $\BiCone$ is used instead of $2\Cone$, but where the comparison functors are still required to be isomorphisms.

There is a general notion of \emph{weak weighted limit} (or \emph{weighted bilimit}), and again it is a fact that a 2-category (or more generally bicategory) has all finite weighted bilimits as soon as it has weak pullbacks, weak arrow objects, and a weak terminal object (this is implied, though not quite stated, in \cite[(1.27)]{street-fibrations-in-bicategories}).

It is clear that pseudo-pullbacks are always bipullbacks.
However, a strict pullback need \emph{not} be a bipullback in general.
Since we always have the weak setting in view, we will want to restrict ourselves to the consideration of those strict pullbacks which \emph{are} bipullbacks, a nice class of which is given as follows.

\subsubsection{Pita limits}\label{subsubsec:pita-limits}
We recall that a functor $F \colon \cC \to \cD$ is an \defword{isofibration} if for every $x \in \cC$ and isomorphism $f \colon Fx \to b$ in $\cD$, there is an isomorphism $p \colon x \to y$ in $\cC$ with $Fp = f$.
A morphism $f \colon X \to Y$ in a 2-category $\CC$ is an \defword{isofibration} if $f_* \colon \CC(A,X) \to \CC(A,Y)$ is an isofibration for all $A \in \CC$.

Now it is easy to see that any strict pullback square (\ref{eq:square}) in which $f$ is an isofibration is also a bipullback \cite{joyal-street-pullbacks} (and $k$ is then also an isofibration), and all the strict pullbacks we will consider will be of this form.

Thus, if a 2-category has strict terminal objects, strict arrow objects, and strict pullbacks of isofibrations, it has all finite bilimits.
We say that 2-category is \defword{pita} (for ``Pullbacks of Isofibrations, Terminal object, and Arrow objects'') or \defword{has pita limits} if it has these strict limits, and we call a 2-functor preserving strict limits of this kind a \defword{pita 2-functor}.

We now make some remarks on the relationship between pita limits and two well-studied classes of strict limits in 2-categories.

First, we have the finite \emph{PIE} limits (``products, inserters, and equifiers'').
These were introduced in \cite{power-robinson-pie-limits}, where an elegant intrinsic characterization of these limits (and those generated by them) is given.
If a 2-category has finite PIE limits it has all finite bilimits (and also all finite pseudo-limits).
It is easy to see that a pita category has finite PIE limits, hence having pita limits is a stronger condition.

The second class is the finite \emph{flexible} limits, which were introduced in \cite{bird-kelly-power-street-flexible}, where again a nice intrinsic definition is given, and where it is shown that a 2-category has flexible limits precisely if it has PIE limits as well as the splitting of idempotent equivalences (or all of all idempotents).
Moreover, in \cite{bourke-accessible-aspects}, it is shown that having flexible limits is equivalent to having finite products, cotensors, splitting of idempotents, and strict pullbacks of DOFs (or of so-called \emph{normal} isofibrations).
However, it does not seem that having flexible limits implies that \emph{all} isofibrations have strict pullbacks, nor that having pita limits implies the splitting of idempotents equivalences, hence having pita limits is neither strictly stronger or weaker than having flexible limits.

We do not know if there is an intrinsic description of pita limits as there is with PIE and flexible limits.

Finally, we point out an important class of isofibrations: the morphism $X^F \colon X^{J'} \to X^J$ between cotensors induced by a functor $F \colon J \to J'$ is an isofibration whenever $F$ is injective on objects.
(This is related to the well-known fact that there is a model structure on $\Cat$ in which the cofibrations and fibrations are the injective-on-objects functors and the isofibrations.)

In fact, the isofibrations $X^F$ are all \emph{normal} in the sense of \emph{op. cit.}, and since all the isofibrations that arise in this paper are either discrete or of the form $X^F$, it follows that everything we do would still hold with ``isofibration'' everywhere replaced by ``normal isofibration''.
In particular, the resulting notion of pita (or rather ``pnita'') 2-categories would include all 2-categories with finite flexible limits as examples.

\subsection{Groupoids}\label{subsec:groupoids}
Let $\CC$ be a 2-category.

An object $X \in \CC$ is \defword{groupoidal} \cite[(1.7)]{street-fibrations-in-bicategories} or simply a \defword{groupoid} if $\CC(A,X)$ is a groupoid for each $A \in \CC$.

We next want to define the \emph{core} $X^\iso$ of an arbitrary object $X \in \CC$, which when $\CC = \Cat$ should give the maximal subgroupoid of $X$.
To begin with, this should be a universal groupoid mapping to $X$, i.e., we have a morphism $i \colon X^\iso \to X$ through which any other morphism from a groupoid factors uniquely, or in other words, \( i_* \colon \cC(A,X^\iso) \to \cC(A,X) \) is bijective on objects for all groupoids \( A \).
The natural corresponding ``2-dimensional'' property is that \( i_* \) induces an isomorphism \(\cC(A,X^\iso) \to \cC(A,X)^\iso \), where the codomain is the maximal subgroupoid of \( \cC(A,X) \).
Either of these properties does determine $X^\iso$ up to isomorphism, but they are not strong enough for our purposes (though see Remark~\ref{rmk:cocores} below); in particular, we will also want a characterization of the morphisms into $X^\iso$ from non-groupoidal $A$.

\begin{defn}
  A morphism $f \colon A \to X$ in $\CC$ is an \defword{arrow-wise iso(morphism)} if $\alpha f$ is invertible for each 2-cell $\alpha \colon A' \tocell A$.
  Note that the composition of an arrow-wise iso with any morphism is again an arrow-wise iso.

  A (strict) \defword{core} of $X$ is an object $X^\iso$ together with an arrow-wise iso $i \colon X^\iso \to X$ with the following universal property:
  for each $A \in \CC$, the functor $i_* \colon \CC(A,X^\iso) \to \CC(A,X)$ is an isomorphism onto the subcategory of $\CC(A,X)$ consisting of arrow-wise isos, and invertible 2-cells between these.
  (It follows that $X^\iso$ is groupoidal.)

  We say that a 2-category is \defword{corepita} if it is pita and has cores.
\end{defn}

We will usually use the above notation $X^\iso \tox{i} X$ for a core of $X$.
From the universal property of the core, any morphism $f \colon X \to Y$ induces a morphism $f^\iso \colon X^\iso \to Y^\iso$.
When $X$ is already a groupoid, we have a canonical choice of core $\id_X \colon X \to X$, and we will always assume that $X^\iso$ has been so chosen unless indicated otherwise.
In particular, a morphism $f \colon X \to Y$ then induces $f^\iso \colon X \to Y^\iso$.

\begin{rmk}\label{rmk:cocores}
  There is a notion which is roughly dual to that of core:
  a (strict) \defword{cocore} of \( X \in \CC \) is an object \( \ol X \) together with a morphism \( \gamma \colon X \to \ol X \) such that for each \( A \in \CC \), the functor \( \gamma^* \colon \CC(\ol X,A) \to \CC(X,A) \) is an isomorphism onto the full subcategory of \( \CC(X,A) \) on the arrow-wise isos.
  (Assuming \( \CC \) has arrow objects, this is equivalent to demanding that \( \gamma \) be a bijection on objects for each \( A \).)
  In \( \Cat \), a cocore \( \ol X \) is given by the category obtained by formally inverting all of the morphisms in \( X \), or in other words the localization of \( X \) at all of its morphisms.

  Assuming that \( \CC \) \emph{has cocores}---i.e., that each object in \( X \) has a cocore---it follows that a morphism \( i \colon X^\iso \to X \) in \( \CC \) is a core of \( X \) as soon as it has the appropriate universal property with respect to \emph{groupoids}, i.e., as soon as \( i^* \colon \CC(A,X^\iso) \to \CC(A,X)^\iso \) is an isomorphism for every groupoid \( A \in \CC \).
\end{rmk}

\subsection{DOFs}\label{subsec:dofs}
We recall that a functor $F \colon \cC \to \cD$ is a \defword{discrete opfibration (or DOF)} if for each $x \in \cC$ and morphism $f \colon F(x) \to b$ in $\cD$, there is a unique morphism $p \colon x \to y$ with $F(p) = f$.

A morphism $f \colon X \to Y$ in a 2-category $\CC$ is a \defword{DOF} if $f_* \colon \CC(A,X) \to \CC(A,Y)$ is a DOF for all $A \in \CC$.
Clearly, every DOF is an isofibration.

\begin{rmk}\label{rmk:dofs-internally}
  If $X,Y \in \CC$ admit arrow objects $X^\to$ and $Y^\to$, then it is easy to see that $f \colon X \to Y$ is a DOF if and only if
  \[
    \begin{tikzcd}
      Y^\to\ar[r, "\partial_0"]\ar[d, "f^\to"']&Y\ar[d, "f"]\\
      X^\to\ar[r, "\partial_0"]&X
    \end{tikzcd}
  \]
  is a strict pullback square.
  In particular, this shows that pita 2-functors between pita 2-categories also preserve DOFs.

  In Appendix~\ref{sec:groupoid-stuff}, we give further ``non-representable'' reformulations of certain ``representably'' defined concepts.
\end{rmk}

We now discuss the properties of DOFs with respect to pullbacks.
First of all, as with isofibrations, they are stable under pullbacks: if $f$ is a DOF in a strict pullback square (\ref{eq:square}) (on p.~\pageref{eq:square}), then so is $k$.
Next, if (\ref{eq:square}) is only a \emph{bipullback} and $f$ is a DOF, then, by replacing $h$ with an isomorphic 1-cell, it can be made into a strictly commuting bipullback.
Finally, once this is done---i.e., assuming (\ref{eq:square}) is a strictly commuting bipullback square and $f$ is a DOF---if $k$ is \emph{also} a DOF, then the square is in fact a strict pullback.

\begin{defn}\label{defn:generic-dof}
  We say that a DOF $p \colon \s_* \to \s$ in a 2-category $\CC$ is \defword{generic}, and that $\s$ is a \defword{DOF classifier}, if, given any two strict pullback squares
  \begin{equation}\label{eq:pb-squares}
    \begin{tikzcd}
      F\ar[r, "\ddot{f}"]\ar[d, "\dot{f}"']\pb&\s_*\ar[d, "p"]\\
      A\ar[r, "f"]&\s
    \end{tikzcd}
    \quad\quad
    \begin{tikzcd}
      H\ar[r, "\ddot{h}"]\ar[d, "\dot{h}"']\pb&\s_*\ar[d, "p"]\\
      A\ar[r, "h"]&\s,
    \end{tikzcd}
  \end{equation}
  as well as a morphism $g \colon F \to H$ satisfying $g \dot{h} = \dot{f}$, there is a unique 2-cell $\gamma \colon f \to h$ for which there exists a (since \( p \) is a DOF necessarily unique) 2-cell $\ddot{\gamma} \colon \ddot{f} \to g \ddot{h}$ with $\ddot{\gamma} p = \dot{f}\gamma$.
  \begin{equation}\label{eq:generic-diags}
    \begin{tikzcd}
      F\ar[rrd, bend left, "\ddot{f}"{name=df}]\ar[ddr, "\dot{f}"', bend right]\ar[dr, "g"]
      \ar[from=df, to=2-2, Rightarrow, dashed, shorten >=0pt, shorten <=4pt, "\ddot{\gamma}"']
      &&\\
      &H\ar[r, "\ddot{h}"]\ar[d, "\dot{h}"']
      &\s_*\ar[d, "p"]\\
      &A
      \ar[r, "f"{name=f}, bend left]
      \ar[r, "h"'{name=h}, bend right]
      &\s.
      \ar[from=f, to=h, Rightarrow, shorten=3pt, "\gamma", dashed]
    \end{tikzcd}
  \end{equation}
  (We will reformulate this condition in \S\ref{subsec:sof-reformulation}.)

  Given a generic DOF $p \colon \s_* \to \s$, we say that a morphism is \defword{$p$-small} (or \defword{$\s$-small}) if it is a strict pullback of $p \colon \s_* \to \s$, as are $\dot{f}$ and $\dot{h}$ above.
  Note that any $p$-small morphism is a DOF and that, by the above-mentioned properties of DOFs, for a DOF to be $p$-small, it suffices for it to be a \emph{bipullback} of $p$.

  We say that a generic DOF $p \colon \s_* \to \s$ is \defword{pre-plentiful} if the $p$-small morphisms satisfy the following closure conditions:
  \begin{enumerate}[(i)]
  \item\label{item:pre-plentiful-monos} Every DOF which is a monomorphism (in $\abs{\CC}$) is $p$-small (and in particular, every isomorphism is $p$-small).
  \item\label{item:pre-plentiful-composites} The composite of $p$-small morphisms is $p$-small.
  \end{enumerate}
\end{defn}
We note that, in general, DOFs satisfy both of these closure conditions, as well as those in Proposition~\ref{propn:more-closure-props} below.

We will come to the full definition of \emph{plentiful} in \S\ref{subsec:mon-fib}.

\begin{rmk}\label{rmk:pseudo-monic}
  While the condition of being a monomorphism is unnatural in the 2-categorical setting, for a DOF $f \colon X \to Y$, it is equivalent to the more natural condition of being \emph{pseudo-monic} \cite[(2.7)]{carboni-et-al-modulated-bicats}, meaning that
  \[
    \begin{tikzcd}
      X\ar[r, "\id_X"]\ar[d, "\id_X"']&X\ar[d, "f"]\\
      X\ar[r, "f"]&Y
    \end{tikzcd}
  \]
  is a \emph{bipullback}, or equivalently that $f_* \colon \CC(A,X) \to \CC(A,Y)$ is full on isomorphisms and faithful for all $A \in \CC$.
\end{rmk}

\begin{propn}\label{propn:more-closure-props}
  We have the following further closure properties for $p$-small morphisms:
  \begin{enumerate}[(i)]
  \item[(iii)] $p$-small morphisms are closed under pullback.
  \item[(iv)] $p$-small morphisms satisfy the following 2-of-3 rule: given morphisms $P\tox{f}Q\tox{g}A$, if $g$ and $fg$ are both $p$-small, then $f$ is $p$-small.
  \end{enumerate}
\end{propn}

\begin{proof}
  Property (iii) follows immediately from the definition (by pasting pullback squares).

  Property (iv) follows from (i)-(iii), since $f$ factors as a composite of $p$-small DOFs $P \tox{\br{\id_P,f}} P \times_A Q \tox{\pi_1} Q$.

  Here, $\pi_1$ is the pullback of a $p$-small DOF, and $\br{\id_P,f}$ is a monomorphism, and it is also a DOF by the 2-of-3 property for DOFs, since $\id_P$ and $\pi_1$ are DOFs (the latter being the pullback of a DOF).
\end{proof}

\subsection{SOFs}\label{subsec:sofs}
As mentioned in the introduction, we will also need to a slight weakening of the concept of DOF.
We recall that the word \emph{setoid} usually denotes a set with an equivalence relation; we use it to mean the equivalent concept of a category which is both a groupoid and a preorder (this is called \emph{bidiscrete} in \cite[\S1.7]{street-fibrations-in-bicategories}).
Thus, just as a discrete fibration is a Grothendieck fibration with discrete fibres, we call a fibration \emph{setoidal} if it has setoidal fibres---or what amounts to the same, if it is faithful and every morphism in the total category is cartesian.
\begin{defn}
  A functor $F \colon \cC \to \cD$ is a \defword{setoidal opfibration} (or \defword{SOF}) if $F$ is faithful, and for each $x \in \cC$ and morphism $f \colon Fx \to b$, there exists a lift $p \colon x \to y$ with $F(p) = f$; and given a second such lift $p' \colon x \to y'$, there is a (necessarily unique) morphism $g \colon y \to y'$ with $p g = p'$ and $F(g) = \id_b$.
  (It follows that $g$ is an isomorphism.)

  We say that a morphism $f \colon X \to Y$ in a 2-category $\CC$ is a SOF if $f_* \colon \CC(A,X) \to \CC(A,Y)$ is a SOF for each $A \in \CC$.
\end{defn}

\begin{rmk}\label{rmk:strictness-and-choice}
  The obvious proof that a functor $f \colon \cC \to \cD$ which is a SOF is also a SOF in the 2-category $\Cat$ makes use of the axiom of choice (and in fact it cannot be proven without choice).
  This is one of several instances where choice is needed to prove that the 2-categorical generalization of some property of categories or functors does in fact specialize to the original notion in $\Cat$ (another is that of an object having finite limits, Definition~\ref{defn:has-finite-limits}).
  In particular, we need choice even to prove that $\Cat$ is an elementary 2-topos in our sense.

  However, these things should be seen side-effects of working in a strict setting.
  Specifically, if we instead consider the \emph{bicategory} $\mathcal{Ana}$ of categories and \emph{anafunctors} \cite{makkai-avoiding-choice}, we can indeed show \emph{without} choice that, for example, a SOF between two categories is a SOF in $\mathcal{Ana}$ (or rather the appropriate bicategorical analogue: the ``Street DOFs'' mentioned in \S\ref{subsec:intro-power-set}).
\end{rmk}

Next, we discuss the procedure of replacing a SOF by a DOF.

\begin{defn}
  Given morphisms $E \tox{p} B \xot{p'} E'$ in a 2-category $\CC$, an \defword{equivalence of morphisms over $B$} between $p$ and $p'$ consists of morphisms $f \colon E \rightleftarrows E' \colon g$ with $f p' = p$ and $g p = p'$, and invertible 2-cells $\alpha \colon f g \to \id_E$ and $\beta \colon g f \to \id_E'$ with $\alpha p = \id_p$ and $\beta p = \id_{p'}$.
  We say that $g$ is a \defword{quasi-inverse of $f$ over $B$}.
  \begin{equation}\label{eq:equiv-of-mors}
    \begin{tikzcd}[column sep=0pt]
      E\ar[rr, bend left, "f"]\ar[dr, "p"']
      \ar[loop, distance=3em, in=south west, out=north west, start anchor={[yshift=-5pt]north west}, end anchor={[yshift=5pt]south west}, ""{name=lloop}, "fg"']
      \ar[from=lloop, to=1-1, Rightarrow, "{}_\alpha"{pos=0.3}, "{}^\sim"'{pos=0.3}]
      &&
      E'\ar[ll, bend left, "g"]\ar[dl, "p'"]
      \ar[loop, distance=3em, in=south east, out=north east, start anchor={[yshift=-5pt]north east}, end anchor={[yshift=5pt]south east}, ""'{name=rloop, pos=0.47}, "gf"]
      \ar[from=rloop, to=1-3, Rightarrow, "{}_\beta"'{pos=0.3}, "{}^\sim"{pos=0.3}]
      \\
      &B&
    \end{tikzcd}
  \end{equation}
\end{defn}

\begin{propn}\label{propn:equiv-of-mors-props}
  Let $\CC$ be a 2-category and let $E \tox{p} B \xot{p'} E'$ be morphisms in $\CC$.
  \begin{enumerate}[(i)]
  \item\label{item:equiv-of-more-props-dof-sof} If $p$ is a DOF (or more generally, a SOF) and $p'$ is equivalent to $p$ over $B$, then $p'$ is a SOF.
  \item\label{item:equiv-of-more-props-extends} If $p$ is an isofibration, then any equivalence $f \colon E \to E'$ with $f p' = p$ extends to an equivalence of morphisms over $B$ as in (\ref{eq:equiv-of-mors}).
  \item\label{item:equiv-of-more-props-sof-unique} If $p$ is a SOF (or more generally, if $p$ is faithful), and $f \colon E \to E'$ is an equivalence over $B$ with quasi-inverse $g$, then there is a \emph{unique} 2-cell $\alpha \colon f g \to \id_E$ with $\alpha p = \id_p$.
  \item\label{item:equiv-of-more-props-dof-good} If $p$ is a DOF, then any equivalence $f \colon E \to E'$ over $B$ has a \emph{unique} quasi-inverse $g$, and moreover (with $\alpha$ and $\beta$ as in (\ref{eq:equiv-of-mors})) we have the equations $f g = \id_E$, $\alpha = \id_{\id_E}$, and $\beta g = \id_g$.
  \end{enumerate}
\end{propn}

\begin{proof}
  Applying the 2-functors $\CC(A,-)$ for $A \in \CC$, reduces \ref{item:equiv-of-more-props-dof-sof} to the case $\CC = \Cat$, where it is easily checked directly.

  For \ref{item:equiv-of-more-props-extends}, fix $\tilde{g} \colon E' \to E$ and an invertible 2-cell $\tilde \beta \colon \tilde{g} f \toi \id_{E'}$ (these exist by the assumption that $f$ is an equivalence).
  Since $p$ is an isofibration, we can find $g \colon E' \to E$ and an invertible 2-cell $\gamma \colon \tilde{g} \toi g$ with $g p = p'$ and $\gamma p = \tilde{\beta} p'$.
  \[
    \begin{tikzcd}
      &[50pt] E\ar[d, "p"]\\[20pt]
      E'
      \ar[ru, "\tilde{g}", ""{name=fg}, bend left=30pt]
      \ar[ru, "g"', ""'{name=g}, dotted]
      \ar[r, "\tilde{g} f p' = \tilde{g} p", ""{name=gp}]
      \ar[r, "p'"', bend right, ""'{name=p}]
      &B
      \ar["\gamma", shorten <=3pt, shorten >=3pt, Rightarrow, from=fg, to=g, dashed, "\sim"' sloped]
      \ar["\tilde \beta p'", shorten <=3pt, shorten >=3pt, Rightarrow, from=gp, to=p, "\sim"' sloped]
    \end{tikzcd}
  \]
  We now set $\beta = (\gamma\I f) \tilde{\beta} \colon g f \toi \id_{E'}$ and we then have $\beta p' = (\gamma\I p) (\tilde{\beta} p') = \id_{p'}$ as desired.
  Next, since $f$ (being an equivalence) is fully faithful, there is a unique (invertible) 2-cell $\alpha \colon f g \toi \id_E$ with $\alpha f = f \beta \colon f g f \to f$, and we then have $\alpha p = \alpha f p' = f \beta p' = f \id_{p'} = \id_p$, as desired.

  The statement \ref{item:equiv-of-more-props-sof-unique} is obvious.

  The equations asserted in \ref{item:equiv-of-more-props-dof-good} follow from the DOF property of $p$, since $\alpha p = \id_p$ and $\beta g p = \beta p' = \id_{p'}$.
  The uniqueness of $g$ also follows from the DOF property of $p$: given equivalences $(f,g_i,\alpha_i,\beta_i)$ over $B$ for $i=1,2$, we have a 2-cell $g_1 \tox{g_1 \alpha_2\I} g_1 f g_2 \tox{\beta_1 g_2} g_2$ with $\pbig{(g_1 \alpha_2\I)(\beta_1 g_2)} p = (g_1 \id_p)(\beta_1 p') = \id_{p'} \cdot \id_{p'} = \id_{p'}$.
\end{proof}

\begin{defn}\label{defn:dof-collapse}
  Given a SOF $E'\tox{p'}B$ in a 2-category $\CC$, a \defword{DOF-collapse} of $p'$ is a DOF $E\tox{p}B$ together with an equivalence $g \colon E' \to E$ with $g p = p'$.
  \[
    \begin{tikzcd}[column sep=0pt]
      E\ar[rd, "p"']\ar[from=rr, "g"']&&E'\ar[dl, "p'"]\\
      &B&
    \end{tikzcd}
  \]

  By Proposition~\ref{propn:equiv-of-mors-props}, $g$ admits a section $f \colon E \to E'$ (i.e., $f g = \id_{E}$), and for any such $g$ there exists a unique $\beta \colon g f \toi \id_{E'}$ with $\beta g = \id_g$.

  We also note that (again by Proposition~\ref{propn:equiv-of-mors-props}) given a DOF $E \tox{p} B$, any equivalence $f \colon E \to E'$ with $f p' = p$ is a section of a unique DOF collapse $g \colon E' \to E$.
\end{defn}

\begin{propn}
  The DOF-collapse is uniquely determined up to isomorphism: given a SOF $p' \colon E' \to B$ in a 2-category $\CC$, DOFs $p_i \colon E_i \to B$ ($i = 1,2$) and equivalences $g_i \colon E' \to E_i$ with $g_i p_i = p'$, there is a unique morphism $h \colon E_1 \to E_2$ with $g_1 h = g_2$, and $h$ is moreover an isomorphism and satisfies $h p_2 = p_1$.
  \[
    \begin{tikzcd}
      &[10pt]E' \ar[ld, "g_1"'] \ar[rd, "g_2"] \ar[dd, "p'", pos=0.6]&[10pt]\\[-10pt]
      E_1 \ar[rd, "p_1"'] \ar[rr, "h"', near start, dashed, crossing over]
      &&E_2 \ar[dl, "p_2"]\\[10pt]
      &B&
    \end{tikzcd}
  \]
\end{propn}

\begin{proof}
  Uniqueness of $h$, and the fact that it is an isomorphism, follows from $g_1,g_2$ being epimorphisms in the 1-categorical sense (since they each admit a section).

  For existence, let $f_1 \colon E_1 \to E'$ be a section of $g_1$ and set $h = f_1 g_2$.
  We then have $h p_2 = f_1 p' = p_1$, and we have a 2-cell $\beta_1 g_2 \colon g_1 h = g_1 f_1 g_2 \toi g_2$ with $(\beta_1 g_2) p_2 = \beta_1 p' = \id_{p'}$, and hence $g_1 h = g_2$ since $p_2$ is a DOF.
\end{proof}

Next, we discuss the relationship of SOFs with generic DOFs.
\begin{defn}
  Given a generic DOF $p \colon \s_* \to \s$ in a 2-category $\CC$, we say that a SOF $p' \colon E' \to B$ is \defword{$p$-small} (or $\s$-small) if there exists a bipullback square
  \begin{equation}\label{eq:psmall-sof-square}
    \begin{tikzcd}
      E'\ar[r, "h"]\ar[d, "p'"']&\s_*\ar[d, "p"]\\
      B\ar[r, "f"]&\s.
    \end{tikzcd}
  \end{equation}
\end{defn}

\begin{propn}\label{propn:small-sofs}
  The following are equivalent:
  \begin{enumerate}[(i)]
  \item $p' \colon E' \to B$ is $p$-small.
  \item There exist a \emph{strictly commuting} bipullback square (\ref{eq:psmall-sof-square}).
  \item $p'$ has a $p$-small DOF-collapse.
  \end{enumerate}
\end{propn}

\begin{proof}
  The equivalence of (i) and (ii) follows from the properties of DOFs with respect to pullbacks stated in \S\ref{subsec:dofs}.

  That (iii) implies (i) follows from the general fact that given a strictly commutative diagram
  \[
    \begin{tikzcd}
      E'\ar[dr, "g"]\ar[rrd, "h", bend left]\ar[rdd, "p'"', bend right]&[-10pt]{}\\[-10pt]
      &E\ar[r, "h"]\ar[d, "q"']&\s_*\ar[d, "p"]\\
      &B\ar[r, "f"]&\s,
    \end{tikzcd}
  \]
  in which the inner square is a strict pullback and $g$ is an equivalence, the outer square is a bipullback (this holds more generally if the inner square and two triangles only commute up to isomorphism, and the inner square is only a bipullback).

  To see that (ii) implies (iii), supposing we have a strictly commuting bipullback square (\ref{eq:psmall-sof-square}), we may form a strict pullback of $p$ along $f$ to obtain a strictly commuting diagram as above with $q$ a DOF.
  Since both squares are bipullbacks, it follows from the universal property of bipullbacks that $g$ is an equivalence.
\end{proof}

\subsection{Internal finite limits and the monomorphism fibration}\label{subsec:mon-fib}
We now describe the construction of the object of monomorphisms $\Mon(X)$ of a given object $X$.
This will require the assumption that $X$ have finite limits in the sense of \cite[(9.14)]{street-cosmoi-of-internal-cats}.

\begin{defn}\label{defn:has-finite-limits}
  We say an object $X$ in a 2-category $\CC$ \defword{has finite limits} if $\CC(A,X)$ has finite limits for each $A \in \CC$, and $f^* \colon \CC(A,X) \to \CC(A',X)$ preserves finite limits for any $f \colon A' \to A$ in $\CC$.
\end{defn}
In \loccit, it is remarked that, as in the familiar case $\CC = \Cat$, if $\CC$ has finite cotensors, then $X$ has finite limits if and only if the diagonal morphism $\Delta \colon X \to X^J$ (classifying the constant functor $J \tox{\const_{\id_X}} \CC(X,X)$) has a right adjoint for each finite category $J$.
In \S\ref{subsec:gpd-stuff-finite-lims}, we describe another alternative characterization of having finite limits.

\begin{rmk}\label{rmk:pointwise-limits}
  We say that a limit of a given diagram $J \to \CC(A,X)$ is \defword{stable} it if is preserved under $f^* \colon \CC(A,X) \to \CC(A',X)$ for all $f \colon A' \to A$ in $\CC$ (thus $X$ has finite limits if and only if for all $A \in \CC$, each finite diagram in $\CC(A,X)$ has a \emph{stable} limit).

  When $\CC = \Cat$, the stability condition is fulfilled as soon as it holds with $A' = \tm$ the terminal category---in other words, the stable limits are in this case the \emph{pointwise} limits.

  While the property $X \in \Cat$ having finite limits, in the usual sense, is indeed characterized by the above definition (assuming the axiom of choice---see Remark~\ref{rmk:strictness-and-choice}), it is also characterized by the simpler condition that $\Cat(A,X)$ have (not necessarily stable) finite limits.
  However, the present definition is the more natural one in the 2-categorical context, as is confirmed by the alternative characterizations mentioned above.
\end{rmk}

\begin{defn}
  A 2-cell $\alpha \colon A \tocell B$ in a 2-category $\CC$ is a \defword{stable monomorphism} if it is a monomorphism and $h\alpha$ is a monomorphism for every $h \colon C \to A$ in $\CC$.
\end{defn}
As with the notion of stable limit, the notion of stable monomorphism specializes to that of \emph{pointwise monomorphism} in the case $\CC = \Cat$.

\begin{defn}
  Given an object $X$ in a 2-category $\CC$, a \defword{monomorphism object for $X$} is an object $\Mon(X)$ together with a stable monomorphism 2-cell $\partial \colon \Mon(X) \tocellud{\partial_0}{\partial_1} X$ such that for any $A \in \CC$, the functor $\partial_* \colon \CC(A,\Mon(X)) \to \CC(A,X)^\to$ is an isomorphism onto the full subcategory of $\CC(A,X)^\to$ consisting of stable monomorphisms.
\end{defn}

We will generally use the above notation to denote a monomorphism object and the associated universal 2-cell.

Note that in the case $\CC = \Cat$, every $\cC \in \CC$ has a monomorphism object, given by the full subcategory $\Mon(\cC)\subset\cC^\to$ with objects the monomorphisms of $\cC$.

\begin{propn}\label{propn:mon-exists}
  Let $\CC$ be a pita 2-category.
  Then any $X \in \CC$ with finite limits admits a monomorphism object.
\end{propn}
One can give a fairly direct construction of $\Mon(X)$, but in Corollary~\ref{cor:pbt-monos}, we obtain it as a consequence of a more general construction of cotensors of $X$ with arbitrary \emph{finite limit sketches}.

\begin{propn}
  Given an object $X$ in a 2-category $\CC$, a monomorphism object $\Mon(X)$, and cores $\Mon(X)^\iso$ and $X^\iso$, the induced morphism $\partial_1^\iso \colon \Mon(X)^\iso \to X^\iso$ is a SOF.
\end{propn}

\begin{proof}
  We need to see that $\CC(A,\Mon(X)^\iso)\tox{(\partial_1^\iso)_*}\CC(A,X^\iso)$ is a SOF for all $A \in \CC$.
  Now, we have a commuting square of functors
  \[
    \begin{tikzcd}
      \CC(A,\Mon(X)^\iso)\ar[r, "(i\partial)_*"]\ar[d, "(\partial_1^\iso)_*"']
      &\Mon(\CC(A,X))^\iso\ar[d, "\partial_1^\iso"]\\
      \CC(A,X^\iso)\ar[r, "i_*"]&\CC(A,X)^\iso
    \end{tikzcd}
  \]
  with the horizontal arrows fully faithful and injective on objects.
  The objects in the image of $i_*$ are by definition the arrow-wise isos, and the objects in the image of the $(i\partial)_*$ are the stable monos $\alpha \colon f \to g$ with $f,g \colon A \to X$ both arrow-wise isos.

  Now, using that since $\CC(A,X)$ has pullbacks, one easily sees that the functor $\partial_1^\iso$ is a SOF.
  Because $(i\partial)_*$ is fully faithful, to see that $(\partial_1^\iso)_*$ is SOF, it then suffices to show that whenever the domain of a cocartesian morphism in $\Mon(\CC(A,X))^\iso$ is in the image of $(i\partial)_*$, then so is its codomain.
  But this is clear, since $\Mon(\CC(A,X))^\iso$ is a groupoid, and the properties of being a stable mono and of being an arrow-wise iso are invariant under isomorphism.
  \qedhere
\end{proof}

Recall the notion of pre-plentiful generic DOF from Definition~\ref{defn:generic-dof}.
\begin{defn}\label{defn:plentiful}
  We say that a pre-plentiful generic DOF $p \colon \s_* \to \s$ in a 2-category $\CC$ is \defword{plentiful} if there exists a monomorphism object $\Mon(\s)$ and cores $\s^\iso$ and $\Mon(\s)^\iso$, and the SOF
  \[
    \partial_1^\iso \colon \Mon(\s)^\iso \to \s^\iso
  \]
  is $p$-small.
\end{defn}
Below, we will be assuming that $\CC$ is corepita, and we will see in \S\ref{subsec:s-has-finite-lims} that any pre-plentiful DOF classifier has finite limits, so the assumptions that $\Mon(\s)$, $\s^\iso$, and $\Mon(\s)^\iso$ exist will hold automatically.

\section{Congruences}\label{sec:congruences}
We take for granted the definition of an internal category $C$ in a 1-category $\cC$, for which we will typically use the notation
\begin{equation}\label{eq:cong-diag}
  \begin{tikzcd}
    C_2\ar[r,shift left=10pt, "\pi_{01}"]\ar[r,shift right=10pt, "\pi_{12}"]\ar[r, "\pi_{02}"]&
    C_1\ar[r, shift left=10pt, "\pi_0"]\ar[r, shift right=10pt, "\pi_1"]\ar[from=r, "e"']&
    C_0,
  \end{tikzcd}
\end{equation}
as well as the definitions of internal functor and internal natural transformation---see, e.g., \cite[Definition~5.2]{helfer-sentai}.
We recall in particular that each object $A \in \cC$ gives rise to a 1-category with objects $\cC(A,C_0)$ and arrows $\cC(A,C_1)$, which we denote by $\cC(A,C)$---and each morphism $u \colon A' \to A$ gives rise to a functor $u^* \colon \cC(A',C) \to \cC(A,C)$.

We introduce some further notation connected with an internal category $C$.
The associativity condition in the definition of $C$ being an internal category makes use of the pullback $C_3 = C_2 \times_{C_1} \times C_2$ (which is thus assumed to exist in $\cC$), which we call an \defword{object of triples} of $C$.
We write $\pi_{012},\pi_{123} \colon C_3 \to C_2$ for the associated projection maps:
\[
  \begin{tikzcd}
    C_3\ar[r, "\pi_{012}"]\ar[d, "\pi_{123}"']\pb&C_2\ar[d, "\pi_{12}"]\\
    C_2\ar[r, "\pi_{01}"]&C_1.
  \end{tikzcd}
\]
We also define $\pi_{02}\defeq \pi_{012}\pi_{02} \colon C_3 \to C_1$ and $\pi_{12}\defeq \pi_{012}\pi_{12} = \pi_{123}\pi_{01} \colon C_3 \to C_1$ and define $\pi_{01},\pi_{13},\pi_{23} \colon C_3 \to C_1$ and $\pi_0,\pi_1,\pi_2,\pi_3 \colon C_3 \to C_0$ similarly.
Given an object $A \in \cC$ and morphisms $f_{01},f_{12},f_{23} \colon A \to C_1$, we write $\br{f_{01},f_{12},f_{23}} \colon A \to C_3$ for the unique morphism with $\br{f_{01},f_{12},f_{23}}\pi_{i(i+1)} = f_{i(i+1)}$ for $i = 0,1,2$.

We now turn to the 2-categorical case.
\begin{defn}\label{defn:int-doub-cat}
  An \defword{internal double category}%
  \footnote{%
    This terminology is consistent with using ``internal category'' in a 2-category $\CC$ to simply refer to an object of $\CC$ (and ``internal elementary topos'' to refer to a certain kind of object), and using ``internal set'' in a 1-category $\cC$ to refer to an object of $\cC$.
  }
  $C$ in a pita 2-category $\CC$ is an internal category in $\abs{\CC}$ such that the morphism $\br{\pi_0,\pi_1} \colon C_1 \to C_0 \times C_0$ is an isofibration.
  We say that $C$ is \defword{groupoidal} if $C_0,C_1,C_2$ are.

  The definition of $C$ being an internal category in $\abs{\CC}$ involves the pullback condition $C_2 \cong C_1 \times_{C_0} \times C_1$ and also involves the object of triples $C_3 = C_2 \times_{C_1} \times C_2$.
  We note that these pullbacks in $\abs{\CC}$ exist and are strict pullbacks in $\CC$, since $\pi_0 \colon C_1 \to C_0$ is an isofibration (being the composite $C_1\tox{\br{\pi_0,\pi_1}}C_0 \times C_0\tox{\pi_0}C_0$ of isofibrations) and $\pi_{01} \colon C_2 \to C_1$ is an isofibration (being a pullback of $\pi_0$).
\end{defn}
\begin{rmk}
  The only reason we impose the condition that $\br{\pi_0,\pi_1}$ is an isofibration is so that the relevant pullback squares are in fact bipullbacks, and hence that our notion of internal double category becomes a special case of the obvious ``weak/bicategorical'' version of the same concept (as in \cite[\S1]{street-bicategories-of-stacks}).
  We note in particular that the ``weak'' analogue of an isofibration (which might be called a ``Street isofibration'' following \cite{johnstone-fibrations}) is simply an arbitrary morphism.

  In any case, this will not matter very much, since the internal double categories we will actually be considering will satisfy the stronger condition that $\br{\pi_0,\pi_1}$ is a \emph{DOF}.
  Similarly, the assumption in Definition~\ref{defn:int-doub-cat} that $\CC$ is pita is convenient, but not really necessary (just as one needn't assume that a category $\cC$ has all finite limits in order to consider internal categories in $\cC$).

  Finally, we note that we should really call this notion a ``strict internal double category'', and likewise for the other notions introduced in this section, but we will drop the adjective ``strict''.
\end{rmk}

We recall that an internal double category $C$ in $\Cat$ is a \emph{double category} in the usual sense (see e.g., \cite{kelly-street-elements}), with objects $\Ob C_0$, \emph{horizontal} arrows $\Ob C_1$, \emph{vertical} arrows $\Ar C_0$, and 2-cells (squares) $\Ar C_1$.
The \emph{horizontal category} of $C$ is $C_0$, and we write $\Ob C$ for the \emph{vertical category}.
The condition that $C$ is groupoidal corresponds to all the vertical arrows being invertible, and all the 2-cells being vertically invertible (i.e., invertible with respect to vertical composition).

Given a 2-cell $\alpha$ in a double category, we refer to the two horizontal arrows and two vertical arrows forming its horizontal and vertical source and target as its \defword{boundary}.

Now if $C$ is an internal double category in an arbitrary pita 2-category $\CC$, applying the limit-preserving 2-functors $\CC(A,-) \colon \CC \to \Cat$ gives us double categories $\CC(A,C)$ for each $A \in \CC$.
We have that $C$ is groupoidal if and only if each $\CC(A,-)$ is.

\begin{defn}
  A \defword{groupoidal bo-congruence} (or just \defword{gbo-congruence}) $C$ in a pita 2-category $\CC$ is a groupoidal internal double category such that $\br{\pi_0,\pi_1} \colon C_1 \to C_0 \times C_0$ is a DOF.

  Note that $C$ is a gbo-congruence if and only if $\CC(A,C)$ is one for each $A \in \CC$.
\end{defn}

\begin{rmk}
  The name ``bo-congruence'' comes from \cite[Proposition~22]{bourke-garner-2-reg-ex}, where they are called ``$\cF_\bo$-congruences'' (``bo'' is for ``bijective-on-objects'').
  That paper studies several natural variants of the notion of congruence in a 2-category.

  The condition that $\br{\pi_0,\pi_1}$ is a DOF is analogous to the requirement that $R \to X \times X$ be a monomorphism in the definition of an internal equivalence relation $R$ on an object $X$ in a 1-category.

  In the non-groupoidal case, the corresponding condition is that $C_0\xot{\pi_0}C_1\tox{\pi_1}C_0$ is a \emph{2-sided discrete fibration}, see \loccitnodot.
  Most of the results of this section have counterparts in the non-groupoidal case.
\end{rmk}

Let us now see what it means for a double category to be a gbo-congruence.
Fix a groupoidal internal double category $C$ in $\Cat$.

We define a relation $vRh$ between vertical arrows $v \colon A \to B$ and horizontal arrows $h \colon A \to B$, by declaring that $v R h$ if there exists a 2-cell
  \begin{equation}\label{eq:R-def-diags}
    \begin{tikzcd}
      A\ar[d, "v"']\ar[r, "h"]\ar[dr, phantom, "\alpha"]&B\ar[d, "\id_B"]\\
      B\ar[r, "\id_B"']&B
    \end{tikzcd}
    \quad
    \text{or equivalently}
    \quad
    \begin{tikzcd}
      A\ar[d, "\id_A"']\ar[r, "\id_A"]\ar[dr, phantom, "\alpha"]&A\ar[d, "v"]\\
      A\ar[r, "h"']&B.
    \end{tikzcd}
  \end{equation}
  (To see that they are equivalent: given $\alpha$ as on the left, consider the vertical composite $\id_v\cdot\alpha\I$, where $\alpha\I$ is the vertical inverse of $\alpha$.)

  \label{propn:R-funct-props}
  It is immediate that $\id_A R\,\id_A$ and one sees upon contemplating the diagram
  \[
    \begin{tikzcd}[sep=25pt]
      A\ar[r, "h_1"]\ar[d, "v_1"']
      \ar[dr, phantom, "\alpha_1"]
      &B\ar[r, "h_2"]\ar[d, "\id"']
      \ar[dr, phantom, "\id_{h_2}"]
      &C\ar[d, "\id"]\\
      B\ar[r, "\id"]\ar[d, "v_2"']
      \ar[dr, phantom, "\id_{v_2}"]
      &B\ar[r, "h_2"]\ar[d, "v_2"]
      \ar[dr, phantom, "\alpha_2"]
      &C\ar[d, "\id"]\\
      C\ar[r, "\id"']&C\ar[r, "\id"']&C
    \end{tikzcd}
  \]
  that $v_1 R h_1$ and $v_2 R h_2$ imply $(v_1v_2)R(h_1h_2)$.

  We note that the functor $\br{\pi_0,\pi_1} \colon C_1 \to C_0 \times C_0$ is \emph{faithful} (a weaker condition than being a DOF) if any only if any two 2-cells in $C$ with the same boundary are equal---hence, in this case the above 2-cell $\alpha$ is uniquely determined.

\begin{propn}\label{propn:cat-gbo-char}
  Let $C$ be a groupoidal internal double category in $\Cat$.
  Then the functor $\br{\pi_0,\pi_1} \colon C_1 \to C_0 \times C_0$ is a DOF if and only if the following two conditions hold:
  \begin{enumerate}[(i)]
  \item The above relation $R$ is a function (i.e., for all $v$ there is a unique $h$ with $vRh$, in which case we write $h = R(v)$), and
  \item Given morphisms $h,k,u,v$ as below, there is at most one 2-cell $\alpha$ with the given boundary, and such an $\alpha$ exists if and only if $h\cdot R(v) = R(u)\cdot k$.
    \[
      \begin{tikzcd}
        A\ar[d, "u"']\ar[r, "h"]\ar[dr, phantom, "\alpha"]&B\ar[d, "v"]\\
        C\ar[r, "k"']&D
      \end{tikzcd}
    \]
  \end{enumerate}
\end{propn}

\begin{proof}
  First suppose $\br{\pi_0,\pi_1}$ is a DOF.

  It follows immediately that $R$ is a function, since given three of the four sides of either square in (\ref{eq:R-def-diags}) (namely the sides except for $h$) the DOF property implies that there is a unique $h$ and a unique $\alpha$ for which the boundary of $\alpha$ is the given square.

  Regarding (ii), we have already noted that uniqueness of $\alpha$ follows from $\br{\pi_0,\pi_1}$ being faithful.
  Now supposing that such an $\alpha$ exists, we can form the composite 2-cell
  \[
    \begin{tikzcd}
      A\ar[r, "\id"]\ar[d, "\id"']&
      A\ar[d, "u"']\ar[r, "h"]\ar[dr, phantom, "\alpha"]&B\ar[d, "v"]
      \ar[r, "\id"]&B\ar[d, "\id"]
      \\
      A\ar[r, "R(u)"']&
      C\ar[r, "k"']&D
      \ar[r, "R(v)\I"']&B,
    \end{tikzcd}
  \]
  and it then follows from the DOF condition that $h = R(u) \cdot k \cdot R(v)\I$ as desired.
  The converse (that this equation implies the existence of a 2-cell $\alpha$) is proved similarly.

  Now, conversely, suppose that conditions (i)-(ii) hold.
  To prove that $\br{\pi_0,\pi_1}$ is a DOF, we need to show that given morphisms $h,u,v$ as in (ii) above, there is a unique morphism $k$ and 2-cell $\alpha$ as shown.
  However, condition (ii) says that we must take $k = R(u)\I\cdot h\cdot R(v)$, and that there will then be a unique such $\alpha$.
\end{proof}

The function $R$ appearing above in the case of a double category has a counterpart for internal double categories in a general 2-category:
\begin{defn}
  Let $C$ be a groupoidal internal double category in a pita 2-category $\CC$, and fix an arrow object $C_0^\to$ of $C_0$.

  We say that a morphism $r \colon C_0^\to \to C_1$ is a \defword{reflection morphism} if there exists a 2-cell $\bar \rho \colon C_0^\to \tocellud{r}{\partial_1e} C_1$ with $\bar \rho\br{\pi_0,\pi_1} = \br{\partial,\id_{\partial_1}}$, or equivalently, a 2-cell $\und \rho \colon C_0^\to \tocellud{\partial_0e}{r} C_1$ with $\und \rho\br{\pi_0,\pi_1} = \br{\partial,\id_{\partial_1}}$.
  \[
    \begin{tikzcd}
      &[40pt] C_1\ar[d, "\br{\pi_0,\pi_1}"]\\[20pt]
      C_0^\to
      \ar[r, "\br{\partial_0,\partial_1}", ""{name=d0d1, pos=0.57}]
      \ar[r, "\br{\partial_1,\partial_1}"', bend right, ""'{name=d1d1}]
      \ar[ru, "r", ""{name=R}, bend left=30pt]
      \ar[ru, "\partial_1e"', ""'{name=d1e}]
      &C_0 \times C_0
      \ar["\bar{\rho}", shorten <=3pt, shorten >=3pt, Rightarrow, from=R, to=d1e]
      \ar["\br{\partial,\id_{\partial_1}}"'{pos=0.4}, shorten <=3pt, shorten >=3pt, Rightarrow, from=d0d1, to=d1d1]
    \end{tikzcd}
    \qquad
    \begin{tikzcd}
      &[40pt] C_1\ar[d, "\br{\pi_0,\pi_1}"]\\[20pt]
      C_0^\to
      \ar[r, "\br{\partial_0,\partial_0}", ""{name=d0d0, pos=0.57}]
      \ar[r, "\br{\partial_0,\partial_1}"', bend right, ""'{name=d0d1}]
      \ar[ru, "\partial_0e", ""{name=d0e}, bend left=30pt]
      \ar[ru, "r"', ""'{name=R}]
      &C_0 \times C_0
      \ar["\und\rho", shorten <=3pt, shorten >=3pt, Rightarrow, from=d0e, to=R]
      \ar["\br{\id_{\partial_0},\partial}"'{pos=0.4}, shorten <=3pt, shorten >=3pt, Rightarrow, from=d0d0, to=d0d1]
    \end{tikzcd}
  \]
  (To see the equivalence of the two conditions: given $\und \rho$, set $\bar{\rho} = \und\rho\I(\partial e)$.)
\end{defn}

\begin{rmk}\label{rmk:reflector-in-cat}
  If $\br{\pi_0,\pi_1}$ is faithful, then there is at most one reflection morphism, and if $\br{\pi_0,\pi_1}$ is a DOF, then there is exactly one.
  In the latter case, for each $A \in \CC$, the induced functor $\CC(A,C_0)^\to \toix{(\partial_*)\I} \CC(A,C_0^\to) \tox{r_*} \CC(A,C_1)$ is given on objects by the function $R$ appearing in Proposition~\ref{propn:cat-gbo-char}, and on arrows, it takes each square
  \[
    \begin{tikzcd}
      W\ar[r, "u"]\ar[d, "w"']&X\ar[d, "x"]\\
      Y\ar[r, "v"]&Z
    \end{tikzcd}
  \]
  in the vertical category $\CC(A,C_0)$ of $\CC(A,C)$ (considered as a morphism $u \to v$ in $\CC(A,C_0)^\to$) to the unique 2-cell
  \[
    \begin{tikzcd}
      W\ar[r, "Ru"]\ar[d, "w"']\ar[dr, phantom, "\alpha"]&X\ar[d, "x"]\\
      Y\ar[r, "Rv"']&Z
    \end{tikzcd}
  \]
  in $\CC(A,C)$ with the displayed boundary.

  In particular (by taking $A$ to be the terminal category), this gives a description of the reflection morphism for gbo-congruences in $\Cat$.
\end{rmk}
We also have a counterpart of Proposition~\ref{propn:cat-gbo-char} in a general pita 2-category $\CC$.
To state it, fix a groupoidal internal double category $C$ in $\CC$, suppose $r$ is a reflection morphism for $C$, and consider the strict pullbacks (of isofibrations)
\[
  \begin{tikzcd}
    C_1 \times_{C_0} C_0^\to\ar[r, ""]\ar[d, ""'] \pb& C_0^\to\ar[d, "\partial_0"]\\
    C_1\ar[r, "\pi_1"]&C_0
  \end{tikzcd}
  \qquad
  \begin{tikzcd}
    C_0^\to \times_{C_0} C_1\ar[r, ""]\ar[d, ""'] \pb&C_1\ar[d, "\pi_0"]\\
    C_0^\to\ar[r, "\partial_1"]&C_0
  \end{tikzcd}
\]
\[
  \begin{tikzcd}
    (C_1 \times_{C_0} C_0^\to) \times_{C_1} (C_0^\to \times_{C_0} C_1)\ar[r, ""]\ar[d, ""'] \pb&
    C_0^\to \times_{C_0} C_1\ar[d, "(r \times_{C_0} \id_{C_1})\pi_{02}"]\\
    C_1 \times_{C_0} C_0^\to\ar[r, "(\id_{C_1} \times_{C_0} r)\pi_{02}"]&C_1.
  \end{tikzcd}
\]
We then have a morphism
\[
  \brbig{
    \br{\partial_0,\pi_1^\to},
    \br{\pi_0^\to,\partial_1}
  } \colon
  C_1^\to \to
  (C_1 \times_{C_0} C_0^\to) \times_{C_1} (C_0^\to \times_{C_0} C_1).
\]
\begin{propn}
  A groupoidal internal double category $C$ in a pita 2-category $\CC$ is a gbo-congruence if and only if it admits a reflection morphism $r$ such that the above morphism
  $\brbig{
    \br{\partial_0,\pi_1^\to},
    \br{\pi_0^\to,\partial_1}
  }$
  is an isomorphism.
\end{propn}
\begin{proof}
  We have already noted that $C$ being a gbo-congruence implies the existence of a reflection morphism $r$.

  Now, given a reflection morphism $r$, the condition that
  $\brbig{
    \br{\partial_0,\pi_1^\to},
    \br{\pi_0^\to,\partial_1}
  }$
  is an isomorphism is equivalent to
  $\brbig{
    \br{\partial_0,\pi_1^\to},
    \br{\pi_0^\to,\partial_1}
  }_* \colon
  \Ob\CC(A,C_1^\to)
  \to
  \Ob\pbig{
    A,
    (C_1 \times_{C_0} C_0^\to) \times_{C_1} (C_0^\to \times_{C_0} C_1)
  }
  $
  being a bijection for all $A \in \CC$.
  But unfolding the definitions, we see that this condition is equivalent to each of the double categories $\CC(A,C)$ satisfying the condition (ii) from Proposition~\ref{propn:cat-gbo-char}, and is thus equivalent to each $\CC(A,C)$ being a gbo-congruence, and thus to $C$ itself being a gbo-congruence.
\end{proof}

The fact the function $R$ from Proposition~\ref{propn:cat-gbo-char} preserves composites and identity morphisms also has a counterpart for the reflection morphism $r$.
Given a groupoid $X$ in a pita 2-category $\CC$, we obtain a (groupoidal) internal double category (this is precisely the \emph{nerve} $\Nv(X)$ of $X$---see \S\ref{subsec:nerves} below)
\[
  \begin{tikzcd}
    X^{[2]}\ar[r,shift left=10pt, "\pi_{01}"]\ar[r,shift right=10pt, "\pi_{12}"]\ar[r, "\pi_{02}"]&
    X^\to\ar[r, shift left=10pt, "\pi_0"]\ar[r, shift right=10pt, "\pi_1"]\ar[from=r, "e"']&
    X,
  \end{tikzcd}
\]
where $[2]$ is the category corresponding to the linear order $\set{0<1<2}$.
\begin{propn}\label{propn:refl-is-internal-fun}
  Given a groupoidal internal double category $C$ in a pita 2-category $\CC$ and a reflection morphism $r$ for $C$, the morphisms $\id_{C_0} \colon C_0 \to C_0$ and $r \colon C_0^\to \to C_1$ constitute an internal functor $\Nv(C_0) \to C$.
\end{propn}
\begin{proof}
  It suffices to show that for each $A \in \CC$, the maps $\id_* \colon \abs{\CC}(A,C_0) \to \abs{\CC}(A,C_0)$ and $r_* \colon \abs{\CC}(A,C_0^\to) \to \abs{\CC}(A,C_1)$ constitute a functor $\abs{\CC}\pbig{A,\Nv(C)} \to \abs{\CC}(A,C)$.
  It is easy to see that the category $\abs{\CC}\pbig{A,\Nv(C_0)}$ is identified (via $\id \colon \abs{\CC}\pbig{A,C_0} \to \Ob\CC(A,C_0)$ and $\partial_* \colon \abs{\CC}(A,C_0^\to) \to \Ar\CC(A,C_0)$) with the category $\CC(A,C_0)$---i.e., the category of vertical arrows in the double category $\CC(A,C)$---whereupon $r_*$ is identified with (a function which is a subset of the relation) $R$ from p.~\pageref{propn:R-funct-props}.

  The claim thus follows from the properties $\id_A R\,\id_A$ and $(v_1 R h_2 \wedge v_2 R h_2) \To (v_1v_2) R (h_1h_2)$ of $R$.
\end{proof}

\subsection{Complete congruences}
We now discuss Rezk's \emph{complete Segal} condition, as introduced in \cite{rezk-homotopy-of-homotopies}.
The condition we are about to state is slightly different from the one in \emph{op. cit.}, but in Proposition~\ref{propn:rezk-equiv-crit}, we will show that they are equivalent.

\begin{defn}\label{defn:rezk-complete}
  Given a groupoid $X$ in a 2-category and an arrow object $X^\to$ for $X$, we define the \defword{inversion morphism} $\tau \colon X^\to \to X^\to$ of $X^\to$ to be the unique morphism  satisfying $\tau \partial = \partial\I$.

  Next, we say that a gbo-congruence $C$ in a pita 2-category $\CC$ is \defword{Rezk complete} (or just \defword{complete}) if the square
  \begin{equation}\label{eq:completeness-cond-square}
    \begin{tikzcd}
      C_0^\to\ar[r, "\br{r,\tau r,r}"]\ar[d, "\br{\partial_0,\partial_1}"']&C_3\ar[d, "\br{\pi_{02},\pi_{13}}"]\\
      C_0 \times C_0\ar[r, "e \times e"]&C_1 \times C_1
    \end{tikzcd}
  \end{equation}
  is a strict pullback square, where $r$ is the reflection morphism.
  We note that if this is a strict pullback square, it is a bipullback; indeed, we have that the composite
  \[
    C_3
    \tox{\br{\pi_{02},\pi_{13}}}
    C_1 \times C_1
    \tox{\br{\pi_0\pi_0,\pi_1\pi_0,\pi_0\pi_1,\pi_1\pi_1}}
    C_0 \times C_0 \times C_0 \times C_0,
  \]
  is a DOF by (two applications of) Lemma~\ref{lem:ob-proj-dof} below, and the second factor is a also a DOF, and hence the first factor is as well by the 2-of-3 property of DOFs.

  We will use \defword{complete congruence} as a shorthand for ``complete gbo-congruence''.
  We note that a gbo-congruence $C$ in $\CC$ is complete if and only if $\CC(A,C)$ is complete for each $A \in \CC$.
\end{defn}

\begin{lem}\label{lem:ob-proj-dof}
  Let $A,B,C,X,Y$ be objects in a 2-category $\CC$, and suppose we are given strict products $A \times B$ and $B \times C$, morphisms $\br{f,g} \colon X \to A \times B$ and $\br{h,k} \colon Y \to B \times C$, and a strict pullback square
  \[
    \begin{tikzcd}
      X \times_BY\ar[r, "\pi_1"]\ar[d, "\pi_0"']&Y\ar[d, "h"]\\
      X\ar[r, "g"]&B.
    \end{tikzcd}
  \]
  Then if $\br{f,g}$ and $\br{g,h}$ are DOFs, so is
  \[
    \br{\pi_0f,\pi_0g,\pi_1k} =
    \br{\pi_0f,\pi_1h,\pi_1k} \colon X \times_BY \to A \times B \times C.
  \]
\end{lem}
\begin{proof}
  This follows from the closure of DOFs under composition and pullbacks, since the morphism in question is just
  \[
    \br{f,g} \times_B\br{g,h} \colon
    X \times_BY \to (A \times B) \times_B(B \times C).\qedhere
  \]
\end{proof}

Next, we clarify the notion of completeness by means of an alternative characterization.

\begin{defn}\label{defn:ob-of-isos}
  Given an internal double category $C$ in a pita 2-category $\CC$ and a morphism $f \colon A \to C_1$, we say that a morphism $g \colon A \to C_1$ is a \defword{$C$-inverse} of $f$ (and that $f$ is \defword{$C$-invertible}) if it is an inverse to $f$ in the category $\abs{\CC}(A,C)$ (i.e., a horizontal inverse in $\CC(A,C)$).
  Note that if $f$ has a $C$-inverse, it is uniquely determined.

  An \defword{object of isomorphisms} of $C$ is an object $C_{\cong} \in \CC$ equipped with a $C$-invertible morphism $c \colon C_{\cong} \to C_1$ and such that each $C$-invertible morphism $f \colon A \to C_1$ factors uniquely through $c$ (so in particular, $c$ is a monomorphism%
  \footnote{
    It will follow from Proposition~\ref{propn:ob-of-iso-char} that, for an object of isomorphism $c \colon C_{\cong} \to C_1$, it is also the case that each 2-cell $\alpha \colon A \tocell C_1$---i.e., 2-cell in $\CC(A,C)$---which is horizontally invertible factors through a unique 2-cell $\alpha \colon A \tocell C_{\cong}$, and hence that $c$ is fully faithful.
    We mention this because it suggests what the ``weak'' analogue of an object of isomorphisms should be.
  }%
  ).
\end{defn}
Of course, if an object of isomorphisms exists, it is determined uniquely up to isomorphism.
Let us next see that one always exists:
\begin{propn}\label{propn:ob-of-iso-char}
  Let $C$ be an internal double category in a pita 2-category $\CC$.
  Then
  \begin{enumerate}[(i)]
  \item Given a morphism $f \colon A \to C_1$ in $\CC$, there exists a commutative square
    \[
      \begin{tikzcd}
        A\ar[r, "\br{f,g,h}"]\ar[d, "\br{x,y}"']&
        C_3\ar[d, "\br{\pi_{02},\pi_{13}}"]\\
        C_0 \times C_0\ar[r, "e \times e"]&C_1 \times C_1
      \end{tikzcd}
    \]
    if and only if $f$ is $C$-invertible.
    In this case, there is a unique such square; more precisely, we have $f = h$, $g$ is the $C$-inverse of $f$, and $\br{x,y} = f\br{\pi_0,\pi_1}$.
  \item If $f \colon A \to C_1$ is $C$-invertible, then the above commutative square is a strict pullback square if and only if $f \colon A \to C_1$ is an object of isomorphisms for $C$.
  \end{enumerate}
\end{propn}

\begin{proof}
  Regarding (i), a morphism $\br{f,g,h} \colon A \to C_3$ is the same as three consecutive morphisms $f,g,h$ in the category $\abs{\CC}(A,C)$, and a morphism $\br{x,y} \colon A \to C_0 \times C_0$ is the same as two objects $x,y$ in $\abs{\CC}(A,C)$.
  The commutativity of the above square then says precisely that the composition in $\abs{\CC}(A,C)$ of $f$ and $g$ is the identity at the object $x$, and the composition of $g$ and $h$ is the identity at $y$.
  But this is the same as saying that $f$ and $h$ are equal and have $C$-inverse $g$, and that $x = \dom_{\abs{\CC}(A,C)}(f) = f\pi_0$ and $y = \cod_{\abs{\CC}(A,C)}(f) = f\pi_1$.

  Regarding (ii), let $D$ be the cospan $C_0 \times C_0\tox{e \times e}C_1 \times C_1\xot{\br{\pi_{02},\pi_{13}}}C_3$, and for any object $A' \in \CC$, let $\Cone(A',D)$ be the category of cones (in $\abs{\CC}$) over $D$ with vertex $A'$, and write $\abs{\CC}(A',C_1)_\inv$ for the set of $C$-invertible morphisms in $\abs{\CC}(A',C)$.
  Write $\Phi \colon \Cone(A',D) \to \abs{\CC}(A',C_1)_\inv$ for the map taking a cone $C_0 \times C_0 \ot A' \tox{u} C_3$ to $u \pi_{01}$.
  By (i), $\Phi$ is a bijection.

  Then, given a commuting square as in (i), we have, for each $A'$ in $\CC$, a commuting triangle
  \[
    \begin{tikzcd}
      &\Cone(A',D)\ar[dd, "\Phi", "\sim"'{sloped}]\\[-20pt]
      \abs{\CC}(A',A)\ar[ru]\ar[rd, "f_*"']\\[-20pt]
      &\abs{\CC}(A',C_1)_\inv.
    \end{tikzcd}
  \]
  But now the square in (i) is a pullback square if and only if the upper horizontal morphism in the above triangle is a bijection for all $A'$, and $f \colon A \to C_1$ is a object of isomorphisms if and only if the lower morphism $f_*$ is a bijection for all $A'$.
  The claim follows.
\end{proof}
We conclude:
\begin{cor}\label{cor:complete-iff-ref-is-iso-ob}
  Given a gbo-congruence $C$ in a pita 2-category $\CC$, the reflection isomorphism $r \colon C_0^\to \to C_1$ is an object of isomorphisms if and only if $C$ is complete.
\end{cor}

\begin{rmk}\label{rmk:reflector-is-hor-invertible}
  We note that, for any gbo-congruence $C$ in a pita 2-category $\CC$, the reflection morphism $r \colon C_0^\to \to C_1$ has $C$-inverse $\tau r$ (i.e., the square (\ref{eq:completeness-cond-square}) in Definition~\ref{defn:rezk-complete} is always commutative).
  This follows from Proposition~\ref{propn:refl-is-internal-fun} and the fact that $\tau$ is a horizontal inverse morphism for the internal double category $\Nv(C)$ (i.e., $f \tau$ is a $C$-inverse of $f$ for each $f \colon A \to C_0^\to$ in $\CC$).
\end{rmk}

We next consider complete congruences in $\Cat$.
\begin{propn}\label{propn:cat-gbo-cong-ob-of-isos}
  Given a gbo-congruence $C$ in $\Cat$, the inclusion $j \colon C_{\mathrm{horiso}} \hto C_1$ of the full subcategory $C_{\mathrm{horiso}}$ of $C_1$ whose objects are the horizontal isomorphisms is an object of isomorphisms for $C$.
\end{propn}
\begin{proof}
  Note first that by the assumption that $C$ is a gbo-congruence, every 2-cell between horizontal isomorphisms in $C$ is horizontally invertible.
  We thus have an inversion functor $\tau \colon C_{\mathrm{horiso}} \to C_{\mathrm{horiso}}$ taking each horizontal isomorphism in $C$ (i.e., object in $C_{\mathrm{horiso}}$) to its inverse, and taking each 2-cell between horizontal isomorphisms in $C$ (i.e., morphism in $C_{\mathrm{horiso}}$) to its horizontal inverse.

  By Proposition~\ref{propn:ob-of-iso-char}, to prove our claim, it now suffices to show that
  \[
    \begin{tikzcd}
      C_{\mathrm{horiso}}\ar[r, "\br{j,\tau j,j}"]\ar[d, "\br{\pi_0,\pi_1}"']&
      C_3\ar[d, "\br{\pi_{02},\pi_{13}}"]\\
      C_0 \times C_0\ar[r, "e \times e"]&C_1 \times C_1.
    \end{tikzcd}
  \]
  is a pullback square in $\Cat$, i.e., that for each pair of objects $(X,Y) \in (C_0 \times C_0) \times C_3$, there is a unique object $Z \in C_{\cong}$ mapping to $X$ and $Y$ under the given functors, and likewise with arrows.
  This is easily verified.
\end{proof}

\begin{propn}\label{propn:cat-compl-gbo-char}
  For a gbo-congruence $C$ in $\Cat$, the following are equivalent:
  \begin{enumerate}[(i)]
  \item $C$ is complete
  \item The function $R$ appearing in Proposition~\ref{propn:cat-gbo-char} is a bijection between vertical morphisms of $C$ and horizontal \emph{isomorphisms} of $C$.
  \item $C$ is isomorphic to the \emph{double category of vertically invertible commutative squares} in the horizontal category $D = \Ob C$ of $C$---this is the double category with $D$ as its horizontal category, the core $D^\iso$ of $D$ as its vertical category, and whose 2-cells are the commutative squares in $D$ whose left and right sides are isomorphisms.
  \end{enumerate}
\end{propn}

\begin{proof}
  The equivalence of (ii) and (iii) is clear in light of Proposition~\ref{propn:cat-gbo-char}.

  The equivalence of (i) and (ii) follows from Corollary~\ref{cor:complete-iff-ref-is-iso-ob} using the explicit descriptions of the reflection morphism and object of isomorphisms of $C$ from Remark~\ref{rmk:reflector-in-cat} and Proposition~\ref{propn:cat-gbo-cong-ob-of-isos}, respectively.
\end{proof}

Let us now see that our definition of completeness is equivalent to (an analogue of) the original definition from \cite{rezk-homotopy-of-homotopies}.
Note that the morphism $e \colon C_0 \to C_1$ admits a $C$-inverse (namely $e$ itself) and hence factors through a unique morphism $\bar{e} \colon C_0 \to C_{\cong}$.

\begin{propn}\label{propn:rezk-equiv-crit}
  Let $C$ be a gbo-congruence in a pita 2-category $\CC$.
  Then $C$ is complete if and only if the morphism $\bar{e} \colon C_0 \to C_{\cong}$ is an equivalence.
\end{propn}

\begin{proof}
  Since gbo-congruences, completeness, equivalences, and objects of isomorphisms are preserved and jointly reflected by the 2-functors $\CC(A,-)$, we may reduce the statement to the case $\CC = Cat$.

  In this case, we have by Proposition~\ref{propn:cat-gbo-cong-ob-of-isos} that $C_{\cong} \subset C_1$ is the category of horizontal isomorphisms.
  Now, the functor $\bar{e} \colon C_0 \to C_{\cong}$ (taking each object and vertical morphism to its horizontal identity) is always faithful (since it admits a retraction $\pi_0 \colon C_1 \to C_0$), and one verifies that it is \emph{full} if and only if the function $R$ from Proposition~\ref{propn:cat-gbo-char} is \emph{injective}, and is \emph{essentially surjective} if and only if $R$ is \emph{surjective}.
  Hence, the claim follows from Proposition~\ref{propn:cat-compl-gbo-char}.
\end{proof}

\subsection{Nerves and quotients}\label{subsec:nerves}
For $k \in \Z_{\ge0}$, let $[k]$ denote the linear order with $k+1$ elements $\set{0,\ldots,k}$, which we consider as a category.
We write $\delta_{a_0 \ldots a_k} \colon [k] \to [l]$ for the functor taking $i$ to $a_i$ for each $i = 0,\ldots,k$.

Now fix an object $X$ in a pita 2-category $\CC$, and fix cotensors $X^{[0]}$, $X^{[1]} = X^\to$ and $X^{[2]}$, where we may take $X^{[0]} = X$.

We have the usual internal category
\begin{equation}\label{eq:internal-cat-in-cat}
  \begin{tikzcd}
    {[2]} {
      \ar[r,shift left=10pt, "{}_{\delta_{01}}"]
      \ar[r,shift right=10pt, "{}_{\delta_{12}}"]
      \ar[r, "{}_{\delta_{02}}"]
    }&
    {[1]} {
      \ar[r, shift left=10pt, "{}_{\delta_{0}}"]
      \ar[r, shift right=10pt, "{}_{\delta_{1}}"]
      \ar[from=r, "{}_{\delta_{00}}"']
    }&
    {[0]}
  \end{tikzcd}
\end{equation}
in $\abs{\Cat}^\op$ and hence an internal category
\begin{equation}\label{eq:pre-nerve}
  \begin{tikzcd}
    X^{[2]} {
      \ar[r,shift left=10pt, "{}_{X^{\delta_{01}}}"]
      \ar[r,shift right=10pt, "{}_{X^{\delta_{12}}}"]
      \ar[r, "{}_{X^{\delta_{02}}}"]
    }&
    X^\to {
      \ar[r, shift left=10pt, "{}_{X^{\delta_0}=\partial_0}"]
      \ar[r, shift right=10pt, "{}_{X^{\delta_1}=\partial_1}"]
      \ar[from=r, "{}_{X^{\delta_{00}}}"']
    }&[10pt]
    X
  \end{tikzcd}
\end{equation}
in $\abs{\CC}$; this is indeed an internal category (and in fact an internal double category in $\CC$) since the (partially defined) anafunctor $X^{(-)} \colon \abs{\Cat}^\op \to \abs{\CC}$ preserves finite limits (as it is a partially defined right adjoint to the functor $\CC(-,X) \colon \abs{\CC} \to \abs{\Cat}^\op$).

\begin{defn}\label{defn:nerve}
  Given an object $X$ in a pita 2-category $\CC$, we define a \defword{nerve} of $X$ to be an internal double category $\Nv(X)$ of the form
  \[
    \begin{tikzcd}
      (X^{[2]})^\iso {
        \ar[r,shift left=12pt, "\pi_{01}"]
        \ar[r,shift right=12pt, "\pi_{12}"]
        \ar[r, "\pi_{02}"]
      }&
      (X^\to)^\iso {
        \ar[r, shift left=12pt, "\pi_0"]
        \ar[r, shift right=10pt, "\pi_1"]
        \ar[from=r, "e"']
      }&
      X^\iso,
    \end{tikzcd}
  \]
  where $\pi_{ij} = \pbig{X^{\delta_{ij}}}^\iso$, $\pi_j = (X^{\delta_j})^\iso = (\partial_j)^\iso$, and $e = \pbig{X^{\delta_{00}}}^\iso$.

  That this is indeed an internal double category follows from the fact that the anafunctor $(-)^\iso \colon \CC \to \CC$ preserves pullbacks (Lemma~\ref{lem:core-preserves-pullbacks} below).

  For each $A \in \CC$, the double category $\CC\pbig{A,\Nv(X)}$ is isomorphic (via $i_* \colon \CC(A,X^\iso) \to \CC(A,X)$ and $(i\partial)_* \colon \CC\pbig{A,(X^\to)^\iso} \to \CC(A,X)^\to$) to the double category of vertically invertible commutative squares in the full subcategory of $\CC(A,X)$ on the arrow-wise isos (see Proposition~\ref{propn:cat-compl-gbo-char}).

  In particular (taking $A$ to be the terminal category), we see that for $X \in \Cat$, $\Nv(X)$ is isomorphic to the double category of vertically invertible commutative squares in $X$.
\end{defn}
\begin{rmk}
  The universal property of the core of a cotensor $(X^J)^\iso$ can be described directly, without reference to (or even assuming the existence of) the intermediate object $X^{J}$: there is a functor $\wt\ev \colon J \to \CC\pbig{(X^J)^\iso,X}$ (namely $\wt\ev = \ev_{X,J}\circ i^*$) such that for any $A \in \CC$, the functor $\CC(A,(X^J)^\iso) \to \CC(A,X)^J$ (taking $f$ to $\wt\ev\circ f^*$) is an isomorphism onto the subcategory of $\CC(A,X)^J$ consisting of functors valued in arrow-wise isomorphisms, and natural isomorphisms between these.

  In particular, for each 2-cell $\alpha \colon A \tocell X$ between arrow-wise isos, there is a unique morphism $f \colon A \to (X^\to)^\iso$ with $fi\partial = \alpha$.
\end{rmk}

\begin{rmk}
  The nerve as defined here is precisely the ``2-kernel complex'' of $e \colon X^\iso \to X$ in the sense of \cite[S2]{makkai-duality-and-definability}, and the ``$\cF_\bo$-kernel'' of \cite[\S5.1]{bourke-garner-2-reg-ex}.
  (By contrast, the ``congruence associated with an arrow'' of \cite[\S1.9]{street-two-sheaf} contains more data, and corresponds to the ``$\cF_\mathrm{so}$-kernel'' of \cite[\S5.2]{bourke-garner-2-reg-ex}.)
\end{rmk}

\begin{lem}\label{lem:core-preserves-pullbacks}
  Given a strict pullback square
  \[
    \begin{tikzcd}
      A\ar[r, "h"]\ar[d, "k"']\pb&B\ar[d, "f"]\\
      C\ar[r, "g"]&D
    \end{tikzcd}
  \]
  in a 2-category, if the objects $A,B,C,D$ all have strict cores, then the induced square
  \[
    \begin{tikzcd}
      A^\iso\ar[r, "h^\iso"]\ar[d, "k^\iso"']&B^\iso\ar[d, "f^\iso"]\\
      C^\iso\ar[r, "g^\iso"]&D^\iso
    \end{tikzcd}
  \]
  is also a strict pullback square.
\end{lem}
\begin{proof}
  This comes down to the claims, for any object $U$, that (i) given arrow-wise isos $b \colon U \to B$ and $c \colon U \to C$ such that $bf = cg$, the induced morphism $\br{b,c} \colon U \to A$ is again an arrow-wise iso, and (ii), given invertible 2-cells $\beta \colon U \tocell B$ and $\gamma \colon U \tocell C$, the induced 2-cell $\br{\beta,\gamma} \colon U \tocell A$ is also invertible.

  One sees that (i) reduces to (ii), and (ii) holds since $\br{\beta\I,\gamma\I} = \br{\beta,\gamma}\I$.
\end{proof}

\begin{propn}
  If $\CC$ is a pita 2-category, then any nerve in $\CC$ is a complete congruence.
  If $\CC = \Cat$, then conversely, any complete congruence is a nerve.
\end{propn}

\begin{proof}
  The case $\CC = \Cat$ follows from Proposition~\ref{propn:cat-compl-gbo-char} since the nerves in $\Cat$ are (up to isomorphism) precisely the double categories of vertically invertible commutative squares.
  The general case follows since $\CC\pbig{A,\Nv(X)}$ is isomorphic to a double category of vertically commutative invertible squares for each $A \in \CC$, hence each $\CC\pbig{A,\Nv(X)}$ is a complete gbo-congruence, hence $\Nv(X)$ is one as well.
\end{proof}

\begin{rmk}
  One of our axioms for a 2-topos $\CC$ will be that every complete congruence is a nerve.
\end{rmk}

We will also want later on to identify the reflection morphism in a nerve.

\begin{lem}\label{lem:nv-refl}
  Given an object $X$ in a corepita 2-category $\CC$, the reflection morphism $r \colon (X^\iso)^\to \to (X^\to)^\iso$ of $\Nv(X)$ is the morphism determined by the condition $r i \partial = \partial i$.
\end{lem}

\begin{proof}
  Let $f \colon (X^\iso)^\to \to (X^\to)^\iso$ be the morphism satisfying $f i \partial = \partial i$.
  To show that $f = r$, it suffices to show that $f_* = r_* \colon \CC(A,(X^\iso)^\to) \to \CC(A,(X^\to)^\iso)$ for all $A \in \CC$.
  Now, $r_*$ is the reflection morphism for the gbo-congruence $\CC(A,\Nv(X))$ in $\Cat$.
  But using the explicit description of the latter (Remark~\ref{rmk:reflector-in-cat}), we see that it agrees with $f_*$ as well.
\end{proof}

Next, we discuss quotients of congruences.
Let $C$ be an internal double category in a pita 2-category $\CC$, or in fact any diagram of the shape (\ref{eq:cong-diag}) on p.~\pageref{eq:cong-diag} satisfying the applicable identities $\pi_{01}\pi_1 = \pi_{12}\pi_0$, $\pi_{02}\pi_0 = \pi_{01}\pi_0$, $\pi_{02}\pi_1 = \pi_{12}\pi_1$, and $e\pi_0 = e\pi_1 = \id_{C_0}$.

We define a quotient of $C$ to be a weighted colimit with respect to the weight given by the internal category (\ref{eq:internal-cat-in-cat}) on p.~\pageref{eq:internal-cat-in-cat} in $\abs{\Cat}^\op$.

Let us spell this out:
\begin{defn}
  Given an object $X$ in a 2-category $\CC$, and with $C$ as above, a \defword{cocone under $C$ with vertex $X$} consists of a morphism $\gamma \colon C_0 \to X$ and 2-cell $\gamma_{01} \colon \pi_0\gamma \to \pi_1\gamma$
\[
  \begin{tikzcd}
    C_2\ar[r,shift left=10pt, "\pi_{01}"]\ar[r,shift right=10pt, "\pi_{12}"]\ar[r, "\pi_{02}"]&
    C_1\ar[r, shift left=10pt, "\pi_0"]\ar[r, shift right=10pt, "\pi_1"]\ar[from=r, "e"']
    \ar[d, shift right=7pt, shorten <=5pt, "\pi_0\gamma"'{name=0}]
    \ar[d, shift left=7pt, shorten <=5pt, "\pi_1\gamma"{name=1}]&
    C_0\ar[ld, "\gamma", bend left, shorten <=8pt]\\[10pt]
    &X
    \ar[from=0, to=1, Rightarrow, shorten=1pt, "\gamma_{01}"]
  \end{tikzcd}
\]
satisfying
\[
  \begin{array}{rl}
    e\gamma_{01} \hspace{-8pt} & = \id_\gamma\\
    \pi_{02}\gamma_{01} \hspace{-8pt} & =
    (\pi_{01}\gamma_{01})
    (\pi_{12}\gamma_{01})
  \end{array}
  \qquad
  \begin{tikzcd}[row sep=10pt, column sep=40pt]
    \pi_{01}\pi_0\gamma\ar[dd, equals]\ar[r, "\pi_{01}\gamma_{01}"]&
    \pi_{01}\pi_1\gamma\ar[d, equals]\\
    &\pi_{12}\pi_0\gamma\ar[r, "\pi_{12}\gamma_{01}"]&\pi_{12}\pi_1\gamma\ar[d, equals]\\
    \pi_{02}\pi_0\gamma\ar[rr, "\pi_{02}\gamma_{01}"]&&\pi_{02}\pi_1\gamma.
  \end{tikzcd}
\]
We obtain a category $\Cocone(C,X)$ of cocones under $C$ with vertex $X$, where a morphism $(\gamma,\gamma_{01}) \to (\gamma',\gamma_{01}')$ is a 2-cell $\alpha \colon \gamma \to \gamma'$ with $(\pi_0\alpha)\gamma_{01}' = \gamma_{01}(\pi_1\alpha)$, and composition is defined in the obvious way.

A cocone $(\gamma,\gamma_{01})$ with vertex $X$ is a (strict) \defword{quotient of $C$} if for each $Y \in \CC$, the evident functor $\CC(X,Y) \to \Cocone(C,Y)$ is an \emph{isomorphism}.
\end{defn}

\begin{rmk}
  There is also a notion of \emph{weak} quotient, defined as a \emph{weak} weighted colimit (or ``weighted bicolimit''), rather than a strict one.
  We note that, since the weight involved is a \emph{pie weight} (see \cite{power-robinson-pie-limits}), any strict quotient is also a weak one.
\end{rmk}

Any nerve comes with a canonical cocone
\[
  \begin{tikzcd}
    (X^{[2]})^\iso\ar[r,shift left=10pt, "\pi_{01}"]
    \ar[r,shift right=10pt, "\pi_{12}"]\ar[r, "\pi_{02}"]&
    (X^{\to})^\iso\ar[r, shift left=10pt, "\pi_0"]
    \ar[r, shift right=10pt, "\pi_1"]\ar[from=r, "e"']
    \ar[d, shift right=7pt, shorten <=5pt, "\pi_0\gamma"'{name=0}]
    \ar[d, shift left=7pt, shorten <=5pt, "\pi_1\gamma"{name=1}]&
    X^\iso\ar[ld, "\gamma=i", bend left, shorten <=8pt]\\[10pt]
    &X
    \ar[from=0, to=1, Rightarrow, shorten=1pt, "\gamma_{01}"]
  \end{tikzcd}
\]
where $\gamma_{01} = i\partial$.

\begin{defn}
  We call the above cocone the \defword{tautological cocone} under $\Nv(X)$.
\end{defn}

\section{The axioms}\label{sec:axioms}
Let $\CC$ be a 2-category.
Our axioms are:
\begin{enumerate}
\item[\rabel{(L)}\label{item:ax-L}]
  $\CC$ has pita limits.
\item[\rabel{(C)}\label{item:ax-C}]
  $\CC$ has strict cores.
\item[\rabel{(N)}\label{item:ax-N}]
  Every complete congruence in $\CC$ is a nerve.
\item[\rabel{(Q)}\label{item:ax-Q}]
  The tautological cocone under any nerve is a strict quotient.
\item[\rabel{(D)}\label{item:ax-D}]
  $\CC$ has a plentiful generic DOF.
\end{enumerate}
(In fact, we will not really make use of Axiom~\ref{item:ax-D}: though we will often be working with and proving things about a fixed plentiful generic DOF $\s$, we will not draw any consequences from the mere existence of one---see also Axiom~\ref{item:ax-UA} below.)

In order to have everything in one place, let us also repeat here the definition of \emph{plentiful} (Definitions~\ref{defn:generic-dof}~and~\ref{defn:plentiful}): a DOF classifier $\s$ is plentiful if
\begin{enumerate}
  \item[\rabel{(D1)}\label{item:plentiful-ax-monos}\label{item:plentiful-ax-first}] every DOF which is a monomorphism is $\s$-small,
  \item[\rabel{(D2)}\label{item:plentiful-ax-composites}\label{item:plentiful-ax-pre-last}] the composite of $\s$-small morphisms is $\s$-small, and
  \item[\rabel{(D3)}\label{item:plentiful-ax-pow}\label{item:plentiful-ax-last}] the monomorphism fibration $\partial_1^\iso \colon \Mon(\s)^\iso \to \s^\iso$ is $\s$-small.
\end{enumerate}

We note, concerning Axiom~\ref{item:ax-L}, that the main point is to demand all finite bilimits; the choice to require pita limits is simply a matter of convenience.
This brings out a certain arbitrariness in the selection of the strict 2-topos axioms: the fully weak versions of the axioms are more definite, and the strict versions consist in assuming the existence of certain convenient ``strictifications''.
Of course, these strictifications should ideally satisfy the constraint that any bicategory satisfying the weak 2-topos axioms should be biequivalent to a strict 2-category satisfying the strict versions.

Another remark about Axiom~\ref{item:ax-L} is that, as noted in \S\ref{subsubsec:pita-limits}, everything in this paper goes through if \ref{item:ax-L} is replaced by the weaker, and possibly more natural, axiom
\begin{enumerate}
\item[\rabel{(L$_{\mathrm n}$)}\label{item:ax-Ln}]
  $\CC$ has pnita limits.
\end{enumerate}

We will loosely use the term \defword{2-topos} to refer to a 2-category satisfying the above axioms, as well as those we discuss below, or whatever subset of these axioms is relevant to the situation at hand.
However, it will be useful to have a name for a 2-category satisfying the first four axioms.
\begin{defn}\label{defn:groupoidant}
  We say that a 2-category is \defword{groupoidant} if it satisfies Axioms~\ref{item:ax-L},~\ref{item:ax-C},~\ref{item:ax-N},~and~\ref{item:ax-Q}.
\end{defn}

\subsection{Further natural axioms}
We now discuss further axioms which we will not make any use of in this paper, but which are natural to include.

The first, from \cite{weber-2-toposes} is:
\begin{enumerate}
\item[\rabel{(CC)}\label{item:ax-CC}] $\CC$ is a cartesian-closed 2-category.
\end{enumerate}
As we noted in \S\ref{subsec:intro-groupoids}, the remaining axiom---or rather, datum---from \cite{weber-2-toposes}, namely that of a duality involution, is omitted, since we will prove from our axioms that one exists; see \S\ref{subsec:opposites} below.

Next, concerning \ref{item:ax-D}, we note that \cite{weber-2-toposes} in fact takes a fixed generic DOF as one of the \emph{data} of a 2-topos.
As mentioned in \S\ref{subsec:intro-groupoids}, we generally consider that a 2-topos should be a 2-category with extra \emph{properties}, and so we make the trivial modification of demanding the \emph{existence} of a generic DOF.

However, this brings out a very interesting difference between 1-topoi and 2-topoi: whereas the universal property of the subobject classifier determines it up to isomorphism, the DOF classifier is not even uniquely determined up to equivalence.
Indeed, $\Cat$ has a DOF classifier $\Set_{<\kappa}$ for each strongly inaccessible cardinal $\kappa$.
This situation is familiar from the context of large cardinal axioms in set theory, and is probably unavoidable:
in set theory, given any fixed set of axioms, one can always introduce a new axiom asserting the existence of a universe satisfying all of the previous axioms.

In this connection, it is natural, rather than just assuming the existence of \emph{one} plentiful generic DOF, to introduce an axiom for 2-topoi which is an analogue of the Axiom of Universes stating that every set belongs to some Grothendieck universe:
\begin{enumerate}
\item[\rabel{(UA)}\label{item:ax-UA}] For any two SOFs $p_1,p_2$ in $\CC$, there is a plentiful generic DOF $p$ such that $p_1$ and $p_2$ are both $p$-small.
\end{enumerate}
By induction (by applying the axiom to the plentiful generic DOFs themselves), this implies that any finite set of SOFs are classified by a single DOF classifier.

We note that, in light of Proposition~\ref{propn:small-sofs}, \ref{item:ax-UA} is equivalent to the corresponding statement only involving \emph{DOFs} \( p_1 \) and \( p_2 \), together with the axiom: every SOF in \( \CC \) has a DOF collapse.
We also note that it would be natural to demand the condition in \ref{item:ax-UA} not only for \emph{SOFs}, but for their ``fully weak'' analogues, the ``Street DOFs'' mentioned in \S\ref{subsec:intro-power-set}~and~Remark~\ref{rmk:strictness-and-choice}.
However, assuming \( \CC \) is pita, this already follows from \ref{item:ax-UA}, since in a pita 2-category, every Street DOF is equivalent to a SOF.

Regarding the non-uniqueness of DOF classifiers, see also Theorem~\ref{thm:plentiful-dof-ess-unique}.

Finally, the Axioms~\ref{item:ax-N}~and~\ref{item:ax-Q} are essentially special cases of the ``(bijective-on-objects,fully-faithful)-exactness'' condition of \cite{bourke-garner-2-reg-ex}.
Hence, we could instead simply assume this wholesale:
\begin{enumerate}
\item[\rabel{(EX)}\label{item:ax-EX}] $\CC$ is $\cF_{\bo}$-exact in the sense of \cite[\S5.1]{bourke-garner-2-reg-ex}.
\end{enumerate}
Again, it would be very interesting if, as in the case of 1-topoi, this axiom followed automatically from the others.

In the presence of \ref{item:ax-EX}, the Axioms~\ref{item:ax-N}~and~\ref{item:ax-Q} could then (as we prove in Lemma~\ref{lem:exact-implies-groupoidant}) be replaced by:
\begin{enumerate}
\item[\rabel{(BO)}\label{item:ax-BO}] For any $X \in \CC$ having a strict core, the core inclusion $X^\iso\tox{i}X$ is ``bijective on objects'', i.e., is an $\cF_{\bo}$-strong epi in the sense of \cite[\S\S2,5.1]{bourke-garner-2-reg-ex}; and for any $\cF_{\bo}$-strong epi $f \colon A \to X$ with $A$ a groupoid, the induced morphism $\bar{f} \colon A \to X^\iso$ with $\bar{f} i = f$ is also an $\cF_{\bo}$-strong epi.
\end{enumerate}
We suspect that the second clause in Axiom~\ref{item:ax-BO} may follow from the first.
At any rate, the examples we consider in \S\ref{sec:examples} will satisfy the stronger condition that for \emph{any} two morphisms $A \tox{f} B \tox{g} C$, if $f$ and $f g$ are $\cF_{\bo}$-strong epis, then so is $f$.

Next, as we mentioned in the introduction, it is natural to demand
\begin{enumerate}
\item[\rabel{(BC)}\label{item:ax-BC}] $\CC$ has finite bicolimits
\end{enumerate}
(or some stricter variant of this), assuming that it doesn't follow from the remaining axioms.
Relatedly, we would also want
\begin{enumerate}
\item[\rabel{(CO)}\label{item:ax-CO}] $\CC$ has (strict) cocores
\end{enumerate}
in the sense of Remark~\ref{rmk:cocores}---though it seems plausible that this in fact follows from \ref{item:ax-BC}.

A related important operation which we would want to have is that of \emph{localizing} a category at a given set of morphisms.
However, this can be achieved using \ref{item:ax-BC} and \ref{item:ax-CO}, as the localization of \( \cC \) at the morphisms in the image of \( \cC' \to \cC \) is the pushout of \( \ol{\cC'} \ot \cC' \to \cC \), where \( \ol{\cC'} \) is the cocore of \( \cC' \).

\section{\alttext{$\s$}{S} is an internal topos}
\subsection{Reformulation of generic DOF}\label{subsec:sof-reformulation}
We now set about proving that in a 2-category satisfying the axioms of the previous section---more specifically, in any corepita 2-category---any plentiful DOF classifier $\s$ is an internal elementary topos, in the sense of Definition~\ref{defn:power-objs} below.

In doing so, it will be convenient to reformulate the genericity condition from Definition~\ref{defn:generic-dof} as the full-faithfulness of the pullback (ana)functor $\CC(A,\s) \to \DOF(A)$ for each $A \in \CC$.
This latter condition is in fact the definition of generic DOF given in \cite{weber-2-toposes}.

Throughout this section, we fix a 2-category $\CC$ and a DOF $p \colon \s_* \to \s$ in $\CC$.

\begin{defn}
  For any object $A \in \CC$, we define the category $\DOF(A)$ as having objects the DOFs in $\CC$ over $A$, and morphisms strictly commuting triangles.
\end{defn}
\begin{rmk}
  $\DOF(A)$ has a natural 2-category structure, but the hom-categories are all discrete.
\end{rmk}

Suppose we have two pullback squares as in (\ref{eq:pb-squares}) on p.~\pageref{eq:pb-squares}, and let us denote them by $\check{F} = (A,F,f,\dot{f},\ddot{f})$ and $\check{H} = (A,H,h,\dot{h},\ddot{h})$.
Recall from Definition~\ref{defn:generic-dof} that the condition for $p$ to be generic was that for each such $\check{F}$ and $\check{H}$ and each \( g \colon F \to H \) with \( g \dot h = \dot f \), there is a unique $\gamma \colon f \to h$ for which there exists a (necessarily unique) $\ddot{\gamma} \colon \ddot{f} \to g \ddot{h}$ with \( \dot{f} \gamma = \ddot{\gamma} p \).

We first show that this is precisely the condition that a certain function (namely $\gamma \mapsto g$) is a bijection.
We leave the following to the reader; see also \cite[\S4]{weber-2-toposes}.
\begin{lem}\label{lem:gr-is-a-function}
  Given $\check{F}$ and $\check{H}$ as above, for each 2-cell $\gamma \colon A \tocellud{f}{h} S$, there is a unique $g \colon F \to H$ satisfying $g \dot{h} = \dot{f}$ and for which there exists a 2-cell $\ddot{\gamma} \colon \ddot{f} \to g \ddot{h}$ with $\ddot{\gamma} p = \dot{f} \gamma$.
\end{lem}

\begin{defn}
  We write $\Gr^A_{\check{F},\check{H}} \colon \CC(A,\s)(f,h) \to \DOF(A)(\dot{f},\dot{h})$ for the function taking each $\gamma \colon f \to h$ to the unique $g$ as in Lemma~\ref{lem:gr-is-a-function}.
\end{defn}

We conclude
\begin{cor}
  The DOF $p \colon \s_* \to \s$ is generic if and only if the function $\Gr^A_{\check{F},\check{H}} \colon \CC(A,\s)(f,h) \to \DOF(A)(\dot{f},\dot{h})$ is a bijection for each $A \in \CC$ and $\check{F},\check{H}$ as above.
\end{cor}

We would now like to further reformulate this condition as saying that a certain (ana)functor $\CC(A,\s) \to \DOF(A)$ is fully faithful for each $A$.

\begin{defn}
  We define the anafunctor $\Gr^A \colon \CC(A,\s) \to \DOF(A)$ for each $A \in \CC$ as follows.
  Given $f \in \CC(A,\s)$, a specification of $\Gr^A$ for $f$ is a pullback square $\check{F} = (A,F,f,\dot{f},\ddot{f})$ as above, and $\Gr^A_{\check{F}}(f) = \dot{f}$.
  Next, given $f,h \in \CC(A,\s)$ with specifications $\check{F},\check{H}$, the map $\Gr^A_{\check{F},\check{H}} \colon \CC(A,\s)(f,h) \to \DOF(A)(\dot{f},\dot{h})$ is the one defined above.
\end{defn}

\begin{propn}
  $\Gr^A$, thus defined, is an anafunctor.
\end{propn}
\begin{proof}
  We have that $\Gr^A_{\check{F},\check{F}}(\id_f) = \id_{\dot{f}}$, since, taking $g = \id_F$ and $\ddot{\gamma} = \id_{\ddot{f}}$, we have $\ddot{\gamma} p = \id_{\ddot{f} p} = \id_{\dot{f} f} = \dot{f} \gamma$.

  Next, given $f_1,f_2,f_3 \in \CC(A,\s)$ with specifications $\check{F}_i = (A,F_i,f_i,\dot{f}_i,\ddot{f}_i)$ and 2-cells $f_1\tox{\gamma_{12}}f_2\tox{\gamma_{23}}f_3$ and morphisms $\dot{f}_1\tox{g_{12}}\dot{f}_2\tox{g_{23}}\dot{f}_3$ with $\Gr^A_{\check{F},\check{H}}(\gamma_{ij}) = g_{ij}$, we have $\Gr^A_{\check{F},\check{H}}(\gamma_{12}\gamma_{23}) = g_{12}g_{23}$.

  Indeed, given 2-cells $\ddot{\gamma}_{ij} \colon \ddot{f}_i \to \ddot{g}_{ij}f_j$ with $\ddot{\gamma}_{ij} p = \dot{f}_i\gamma_{ij}$, we define $\ddot{\gamma}_{13} = \ddot{\gamma}_{12}(g_{12}\ddot{\gamma}_{23}) \colon \ddot{f}_1 \to (g_{12}g_{23})\ddot{f}_3$, and we then have
  \[
    \ddot{\gamma}_{13}p =
    (\ddot{\gamma}_{12} p)\plbig{g_{12}\ddot{\gamma}_{23}p} =
    (\dot{f}_1\gamma_{12})\plbig{g_{12}\dot{f}_2\gamma_{23}}
    =\dot{f}_1(\gamma_{12}\gamma_{23}),
  \]
  as required.
  
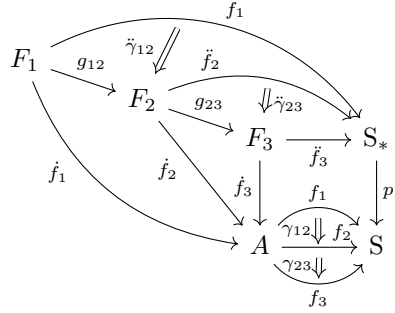
\[
    \begin{tikzcd}[row sep=25pt, column sep=25pt, baseline=(f3.base)]
      |[yshift=30pt]|F_1\ar[rrr, "\ddot{f}_1", ""{name=ddf1, pos=0.38}, bend left=40pt]
      \ar[rrd, "\dot{f}_1"'{pos=0.3, name=df1}, bend right]
      \ar[r, "g_{12}"{pos=0.3}]
      &|[alias=F2,yshift=15pt]|F_2
      \ar[rr, bend left=30pt, "\ddot{f}_2"{pos=0.3}, ""name=ddf2]
      \ar[dr, "\dot{f}_2"'{pos=0.3,name=df2}]
      \ar[r, "g_{23}"{pos=0.3}]
      &|[alias=F3]|F_3\ar[r, "\ddot{f}_3"']\ar[d, "\dot{f}_3"'{pos=0.4}]
      &\s_*\ar[d, "p"]\\
      &&A\ar[r, "f_1"{name=f1}, bend left=50pt]
      \ar[r, "f_2"{pos=0.8}, ""{name=f2}]
      \ar[r, "f_3"'{name=f3}, bend right=50pt]
      &\s
      \ar[r, from=ddf1, to=F2, Rightarrow, shorten >=2pt, shorten <=4pt, "\ddot{\gamma}_{12}"'{pos=0.7}]
      \ar[r, from=ddf2, to=F3, Rightarrow, shorten >=2pt, shorten <=4pt, "\ddot{\gamma}_{23}"{pos=0.3}]
      \ar[r, from=f1, to=f3, Rightarrow, shorten >=15pt, shorten <=4pt, "\gamma_{12}"'{pos=0.3}]
      \ar[r, from=f2, to=f3, Rightarrow, shorten >=2pt, shorten <=4pt, "\gamma_{23}"']
    \end{tikzcd}
    \qedhere
  \]
\end{proof}
We conclude:
\begin{cor}
  The DOF $p \colon \s_* \to \s$ is generic if and only if the anafunctor $\Gr^A \colon \CC(A,\s) \to \DOF(A)$ is fully faithful for each $A \in \CC$.
\end{cor}

\subsection{\alttext{$\s$}{S} has finite limits}\label{subsec:s-has-finite-lims}
\begin{propn}\label{propn:generic-dof-limits-are-stable}
  Let $\CC$ be a 2-category having strict pullbacks of DOFs.
  For any generic DOF $p \colon \s_* \to \s$ and any morphism $f \colon A' \to A$ in $\CC$, the functor $f^* \colon \CC(A,\s) \to \CC(A',\s)$ preserves finite limits.
\end{propn}

\begin{proof}
  This follows from the commutativity up to isomorphism of the square of anafunctors
  \[
    \begin{tikzcd}
      \CC(A,\s)\ar[r, "f^*"]\ar[d, "\Gr^A"']&\CC(A',\s)\ar[d, "\Gr^{A'}"]\\
      \DOF(A)\ar[r, "f^*"]&\DOF(A'),
    \end{tikzcd}
  \]
  where the $f^*$ on the bottom is the usual pullback anafunctor, since the vertical anafunctors are fully faithful, and the bottom anafunctor preserves limits (more generally, the anafunctor $f^* \colon \abs{\CC}/A' \to \abs{\CC}/A$ preserves limits, since it has a left adjoint $\Sigma_f$).

  Let us spell out the commutativity above the square.
  To define an isomorphism $\alpha \colon \Gr^A f^* \toi f^*\Gr^{A'}$ of anafunctors, it suffices to specify, for each $g \in \CC(A,S)$, each specification $s_1$ for $\Gr^A$ at $g$, and each specification $s_2$ for $f^*$ at $\Gr^A_{s_1}(g)$, a specification $s_3$ of $\Gr^{A'}$ at $f^*(g)$ and an isomorphism $\alpha_{s_1,s_2,s_3} \colon f^*_{s_2}\pbig{\Gr^A_{s_1}(g)} \toi \Gr^{A'}_{s_3}(f^*g)$ satisfying the appropriate naturality conditions \cite[1.7~and~p.~126]{makkai-avoiding-choice}.

  The specifications $s_1$, $s_2$ are given by pullback squares as shown below.
  \[
    (s_1)
    \quad
    \begin{tikzcd}
      G\ar[r, "\ddot{g}"]\ar[d, "\dot{g}"']\pb&\s_*\ar[d, "p"]\\
      A\ar[r, "g"]&\s
    \end{tikzcd}
    \quad
    \qquad
    (s_2)
    \quad
    \begin{tikzcd}
      f^*G\ar[r, "\pi_2"]\ar[d, "\pi_1"']\pb&G\ar[d, "\dot{g}"]\\
      A'\ar[r, "f"]&A
    \end{tikzcd}
    \quad
    \qquad
    (s_3)
    \quad
    \begin{tikzcd}
      f^*G\ar[r, "\pi_2 \ddot{g}"]\ar[d, "\pi_1"']\pb&\s_*\ar[d, "p"]\\
      A'\ar[r, "fg"]&\s
    \end{tikzcd}
  \]
  Given these, we can take the specification $s_3$ to be the displayed pullback square (the pasting of $s_1$ and $s_2$).

  We then have $f^*_{s_2}\pbig{\Gr^A_{s_1}(g)} = (f^*G,\pi_1) = \Gr^{A'}_{s_3}(f^*g)$, and we can take $\alpha_{s_1,s_2,s_3} = \id$.
  Given this definition of $\alpha_{s_1,s_2,s_3}$, the naturality statement becomes: given a morphism $\beta \colon g \to g'$ in $\CC(A,S)$ and specifications $s_1',s_2',s_3'$ as above, we have $f^*_{s_2,s_2'}\pbig{\Gr^A_{s_1,s_1'}(\beta)} = \Gr^{A'}_{s_3,s_3'}(f^*\beta)$.
  After unravelling the definitions, this follows from inspecting the diagram
\[
    \begin{tikzcd}
      f^*G\ar[rr, "\pi_2"]
      \ar[ddr, "\pi_1"', bend right]\ar[dr, "f^*b"]&&
      G\ar[rrd, bend left, "\ddot{g}"{name=dg}]
      \ar[ddr, "\dot{g}"', bend right]\ar[dr, "b"]
      \ar[from=dg, to=2-4, Rightarrow, shorten >=0pt, shorten <=4pt, "\ddot{\beta}"']\\
      &f^*H\ar[d, "\pi_1"]\ar[rr, "\pi_2"{pos=0.4}, crossing over]&&H\ar[r, "\ddot{h}"]\ar[d, "\dot{h}"]&
      \s_*\ar[d, "p"]\\
      &A'\ar[rr, "f"]&&A
      \ar[r, "g"{name=g}, bend left]
      \ar[r, "h"'{name=h}, bend right]
      &\s
      \ar[from=g, to=h, Rightarrow, shorten=3pt, "\beta"]
    \end{tikzcd}
  \]
  and in particular noting that $(\pi_2\ddot{\beta})p = \pi_1(f\beta)$.
\end{proof}

In light of the above proposition, to show that $\s$ has finite limits (Definition~\ref{defn:has-finite-limits}), it suffices to show that $\CC(A,\s)$ has finite limits for each $A \in \CC$.
In fact, we have:
\begin{propn}\label{propn:gr-preserves-reflects-limits}
  Let $\CC$ be a 2-category having strict pullbacks of DOFs.

  Then for each $A \in \CC$, the category $\DOF(A)$ has finite limits, and if $p \colon \s_* \to \s$ is any pre-plentiful generic DOF in $\CC$, the essential image of the fully faithful anafunctor $\Gr^A \colon \CC(A,\s) \to \DOF(A)$ (i.e., the set of $\s$-small DOFs) is closed under finite limits.
\end{propn}
\begin{proof}
  $\DOF(A)$ has a terminal object, namely $\id_A$, which is $\s$-small by assumption of $\s$ being pre-plentiful.

  Now suppose we are a cospan $(Q,q)\tox{f}(S,s)\xot{g}(R,r)$ in $\DOF(A)$.
  We wish to show that it has a pullback in $\DOF(A)$.

  Form a strict pullback square in $\CC$:
  \[
    \begin{tikzcd}
      P\ar[r, "\pi_1"]\ar[d, "\pi_2"']\pb&Q\ar[d, "f"]\\
      R\ar[r, "g"']&S
    \end{tikzcd}
  \]
  Since $q,r,s$ are DOFs, it follows from the 2-of-3 property that $f$ and $g$ are as well, and hence by stability under pullback that $\pi_1$ and $\pi_2$ are.

  By closure under composition, it thus follows that $(P,p)\defeq(P,\pi_1 q) = (P,\pi_2 r)$ is a DOF over $A$, hence the above square gives a square in $\DOF(A)$, which is immediately seen to be a pullback.

  We need to see that if $q,r,s$ are $\s$-small, then so is $p$.
  By the by the 2-of-3 property for $\s$-small morphisms, it follows that $f$ is $\s$-small, hence by the stability under pullbacks, $\pi_2$ is $\s$-small, and hence $p = \pi_2 r$ is $\s$-small by closure under composition.
\end{proof}

\begin{cor}\label{cor:s-has-limits}
  If $\CC$ is a 2-category having strict pullbacks of DOFs, then any pre-plentiful DOF classifier $\s \in \CC$ has finite limits.
\end{cor}

\subsection{Universal property of \alttext{$\wt\pow$}{P̃}}\label{subsec:pow-univ-prop}
Let $\CC$ be a corepita 2-category, and let $p \colon \s_* \to \s$ be a pre-plentiful generic DOF.
Fix a monomorphism fibration $\partial_1^\iso \colon \Mon(\s)^\iso \to \s^\iso$ for $\s$ (\S\ref{subsec:mon-fib}) and a DOF collapse $\dot{\pow} \colon \wt\pow \to \s^\iso$ of $\partial_1^\iso$.
\begin{equation}
  \begin{tikzcd}[column sep=0pt]
    \Mon(\s)^\iso\ar[rr, "\pi"]\ar[dr, "\partial_1^\iso"']&&
    \wt\pow\ar[dl, "\dot{\pow}"]\\
    &\s^\iso&
  \end{tikzcd}
\end{equation}
In the case $\CC = \Cat$, the DOF $\dot{\pow} \colon \wt\pow \to \s$ is isomorphic to the subset fibration mentioned in \S\ref{subsec:intro-power-set}.

We now prove a universal property characterizing $\wt\pow$, so that we can work with it directly, without reference to $\Mon(\s)^\iso$.
\begin{propn}\label{propn:pow-univ-prop}
  If $\CC$ is a corepita 2-category, and $p \colon \s_* \to \s$ is a pre-plentiful generic DOF in $\CC$, then the DOF $\dot{\pow} \colon \wt\pow \to \s^\iso$ defined above is equipped with a stable monomorphism 2-cell
  \[
    \twocell{\wt\pow}{\s}{\tilde{\partial}_0}{\tilde{\partial}_1\defeq\dot{\pow} i}
    {\tilde{\partial}}
  \]
  with the following universal property: for any $A \in \CC$ and any stable monomorphism 2-cell
  \begin{equation}\label{eq:pow-univ-prop-data}
    \twocell{A}{\s}{f}{g}{\alpha}
  \end{equation}
  with $f,g$ arrow-wise isos (hence for arbitrary $f,g$ if $A$ is groupoidal), there is a unique $u \colon A \to \wt\pow$ such that $u\tilde{\partial}_1 = g$ and such that there exists a (necessarily unique) isomorphism $\xi \colon u\tilde{\partial}_0 \toi f$ making the following square in $\CC(A,\s)$ commute.
  \begin{equation}\label{eq:pow-univ-prop-square}
    \begin{tikzcd}
      u\tilde{\partial}_0\ar[r, "\sim"', "\xi"]\ar[d, >->, "u\tilde{\partial}"']&f\ar[d, >->, "\alpha"]\\
      u\tilde{\partial}_1\ar[r, equals]&g
    \end{tikzcd}
  \end{equation}
\end{propn}
\begin{proof}
  Recall from Definition~\ref{defn:dof-collapse} that we have a morphism $\sigma \colon \wt\pow \to \Mon(\s)^\iso$ and an invertible 2-cell $\phi \colon \pi\sigma \toi \id_{\Mon(\s)^\iso}$ satisfying $\sigma\pi = \id_{\wt\pow}$ and $\phi \pi = \id_{\pi}$ (and hence also $\sigma \partial_1^\iso = \dot{\pow}$ and $\phi \partial_1^\iso = \id_{\partial_1^\iso}$).
  \vspace{-0.5cm}
  \begin{equation}\label{eq:mon-collapse}
    \hspace{-1cm}
    \begin{tikzcd}[column sep=0pt]
      \Mon(\s)^\iso\ar[rr, "\pi"']\ar[dr, "\partial_1^\iso"']
      \ar[loop, distance=3em, in=south west, out=north west, start anchor={[yshift=-5pt]north west}, end anchor={[yshift=5pt]south west}, ""{name=loop}", "\pi\sigma"']
      \ar[from=loop, to=1-1, Rightarrow, "{}_\phi"{pos=0.3}, "{}^\sim"'{pos=0.3}]
      &&
      \wt\pow\ar[ll, bend right, "\sigma"']\ar[dl, "\dot{\pow}"]
      \\
      &\s^\iso&
    \end{tikzcd}
  \end{equation}
  Recall also that we have a universal stable monomorphism 2-cell $\partial \colon \Mon(\s) \tocellud{\partial_0}{\partial_1} \s$.

  We define $\tilde{\partial}_0 = \sigma i\partial_0$ and $\tilde{\partial} = \sigma i\partial$.
  Note that $\tilde{\partial}_1 = \dot{\pow} i = \sigma\partial_1^\iso i = \sigma i\partial_1$, so that indeed $\tilde{\partial} \colon \tilde{\partial}_0 \to \tilde{\partial}_1$.

  Now suppose we are given data as in (\ref{eq:pow-univ-prop-data}).
  By the universal property of $\Mon(\s)$, there is a unique morphism $v \colon A \to \Mon(\s)$ with $v\partial = \alpha$.

  From the fact that $f,g$ are arrow-wise isos and using the 2-categorical part of the universal property of $\Mon(\s)$, it follows that $v$ is a arrow-wise iso, and hence we have an induced morphism $\bar{v} \colon A \to \Mon(\s)^\iso$ with $\bar{v} i = v$.

  \begin{equation}\label{eq:induced-pow-mor}
    \begin{tikzcd}
      &&\wt\pow\ar[d, "\dot{\pow}"]\ar[dl, bend right, "\sigma"']\\
      &\Mon(\s)^\iso\ar[ru, "\pi"]\ar[r, "\partial_1^\iso"]\ar[d, "i"]&\s^\iso\ar[d, "i"]\\
      &\Mon(\s)\ar[r, bend left, "\partial_0"{name=par1,pos=0.4}]
      \ar[r, bend right=5pt, "\partial_1"'{name=par2,pos=0.3}]&\s\\
      A\ar[ru, "v"]\ar[ruu, "\bar{v}"]
      \ar[rru, "f"{name=f}, bend right=10pt]
      \ar[rru, "g"'{name=g}, bend right=35pt]&
      \ar[Rightarrow, from=f, to=g, shorten=3pt, "\alpha"]
      \ar[Rightarrow, from=par1, to=par2, shorten=3pt, "\partial"]
    \end{tikzcd}
  \end{equation}
  We now set $u = \bar{v}\pi$ and we have
  \[
    u\tilde{\partial}_1 =
    \bar{v}\pi\dot{\pow} i =
    \bar{v}\partial_1^\iso i =
    \bar{v} i\partial_1 =
    v\partial_1 =
    g
  \]
  as required.
  Taking $\xi = \bar{v}\phi i\partial_0$, the square (\ref{eq:pow-univ-prop-square}) then becomes
  \[
    \begin{tikzcd}
      \bar{v}\pi\sigma i\partial_0\ar[r, "\sim"', "\bar{v}\phi i\partial_0"]\ar[d, "\bar{v}\pi\sigma i\partial"']&\bar{v} i\partial_0\ar[d, "\bar{v}i\partial"]\\
      \bar{v}\pi\sigma i\partial_1\ar[r, equals]&\bar{v} i\partial_1.
    \end{tikzcd}
  \]
  Now the composite of the top and right morphisms is seen to be the horizontal composite
  \[
    \begin{tikzcd}
      A {
        \ar[r, bend left, "\bar{v}\pi\sigma"{name=sp}, pos=0.55,
            start anchor=north east, end anchor=north west]
        \ar[r, bend right, "\bar{v}"'{name=id}, pos=0.55,
            start anchor=south east, end anchor=south west]
      }&
      \Mon(\s)^\iso {
        \ar[r, bend left, "i\partial_0"{name=p1}, pos=0.45,
            start anchor=north east, end anchor=north west]
        \ar[r, bend right, "i\partial_1"'{name=p2}, pos=0.45,
            start anchor=south east, end anchor=south west]
      }&
      \s.
      \ar[from=sp, to=id, "\bar{v}\phi", "\sim" sloped, Rightarrow, shorten=3pt]
      \ar[from=p1, to=p2, "i\partial", Rightarrow, shorten=3pt]
    \end{tikzcd}
  \]
  But this same horizontal composite is also equal to
  \[
    \plbig{(\bar{v}\pi\sigma)(i\partial)}
    \plbig{(\bar{v}\phi)(i\partial_1)} =
    \plbig{\bar{v}\pi\sigma i\partial}
    \plbig{\bar{v}\phi\partial_1^\iso i} =
    \plbig{\bar{v}\pi\sigma i\partial}
    \plbig{\bar{v}\id_{\partial_1^\iso}i} =
    \bar{v}\pi\sigma i\partial,
  \]
  as required.

  It remains to verify the uniqueness of $\bar{v}\pi \colon A \to \wt\pow$.
  Supposing we had another $u$ with the same property, we would then have a commuting square
  \[
    \begin{tikzcd}
      u\sigma i\partial_0\ar[r, "\sim"]\ar[d, "u\sigma i\partial"']&\bar{v} i\partial_0\ar[d, "\bar{v}i\partial"]\\
      u\sigma i\partial_1\ar[r, equals]&\bar{v} i\partial_1.
    \end{tikzcd}
  \]
  The universal properties of $\Mon(\s)$ and $\Mon(\s)^\iso$ then imply that there is an isomorphism $\beta \colon u \sigma \toi \bar{v}$ with $\beta\partial_1^\iso = \id_{\bar{v}\partial_1^\iso}$.
  We therefore have an isomorphism
  \[
    u = u \sigma \pi
    \xrightarrow{\beta\pi}
    \bar{v} \pi
  \]
  with $(\beta\pi) \dot{\pow} = \id_{\bar{v}\pi\dot{\pow}}$.
  Because $\dot{\pow}$ is a DOF, it follows that $\beta \pi = \id_u$ and that $u = \bar{v} \pi$, as desired.
\end{proof}

\subsection{Internal power objects}\label{subsec:internal-power}
Given an object $X \in \cC$ in a category $\cC$ with finite limits, recall that a \defword{power object} of $X$ is an object $P(X) \in \cC$ equipped with a universal monomorphism
\[
  j \colon \epsilon_X\tto X \times P(X),
\]
i.e., for any other product $X \times Y$ with $X$ and monomorphism $R\tto X \times Y$, there is a unique $r \colon Y \to P(X)$ such that there exists a pullback square.
\[
  \begin{tikzcd}
    R\ar[r, ""]\ar[d, >->, ""']\pb&[10pt]\epsilon_X\ar[d, >->, "j"]\\
    X \times Y\ar[r, "\id \times r"]&X \times P(X)
  \end{tikzcd}
\]
We say that $\cC$ \defword{has power objects} if every $X \in \cC$ has a power object.
We recall that $\cC$ is a topos if and only if it has finite limits and power objects \cite[IV.1]{mac-lane-moerdijk}.

Next, a finite limit preserving functor $F \colon \cC \to \cD$ \defword{preserves the power object $P(X)$} if it preserves the product $X \times P(X)$, and $F(j) \colon F(X) \to F(X \times P(X))$ is again a universal monomorphism.
A \defword{morphism of elementary topoi}, or \defword{logical functor} is a functor preserving finite limits and all power objects.

\begin{defn}\label{defn:power-objs}
  An object $X \in \CC$ with finite limits in a 2-category $\CC$ \defword{has power objects} (or is an \defword{internal elementary topos}---or just an \defword{elementary topos}) if $\CC(A,X)$ has power objects for each \emph{groupoid} $A \in \CC$, and $f^* \colon \CC(A,X) \to \CC(A',X)$ preserves finite limits and power objects for each groupoid $A' \in \CC$ and $f \colon A' \to A$.
\end{defn}

\begin{rmk}
  Remark~\ref{rmk:pointwise-limits} applies, \emph{mutatis mutandis}, with regard to $X$ having power objects rather than finite limits.
\end{rmk}

\begin{rmk}
  The reader may find it unsettling that the definition of $X \in \CC$ having power objects makes reference only to the groupoids in $\CC$.
  As mentioned in \S\ref{subsec:intro-groupoids}, the reason is that it is simply not the case, even for $\CC = \Cat$, that $\CC(\cC,\cD)$ is an elementary topos whenever $\cD$ is, if we don't assume $\cC$ is a groupoid.

  More directly, we can observe that the universal monomorphism $\epsilon_X \tto X \times P(X)$ is not functorial in $X$ (except along isomorphisms!), hence given a functor $F \colon \cC \to \cD$, taking power objects $P(F(X))$ object-wise does not result in a power object of $F$ in $\Fun(\cC,\cD)$ (again, unless $\cC$ is a groupoid).

  That this is nonetheless a reasonable definition, at least when $\CC$ is groupoidant, is evidenced by Theorem~\ref{thm:gpd-yoneda-rstr} stating that the groupoidal objects are in this case \emph{dense} in $\CC$, as well as by the results of Appendix~\ref{sec:groupoid-stuff}
\end{rmk}

We now set out to show that a plentiful generic DOF classifier has power objects.

Let $\CC$ be a corepita 2-category, and let $p \colon \s_* \to \s$ be a plentiful generic DOF.
Fix a monomorphism fibration $\partial_1^\iso \colon \Mon(\s)^\iso \to \s^\iso$ and a DOF collapse $\dot{\pow} \colon \wt\pow \to \s^\iso$ as in \S\ref{subsec:pow-univ-prop}.

Since $\s$ is plentiful, the DOF $\dot{\pow}$ is classified by a morphism $\pow \colon \s^\iso \to \s$, so that we have a pullback square
\[
  \begin{tikzcd}
    \wt\pow\ar[r, "\ddot{\pow}"]\ar[d, "\dot{\pow}"']\pb&\s_*\ar[d, "p"]\\
    \s^\iso\ar[r, "\pow"]&\s.
  \end{tikzcd}
\]
As mentioned in \S\ref{subsec:intro-power-set}, in the case $\CC = \Cat$, the functor $\pow$ is precisely the power set functor.

\begin{propn}\label{propn:s-power-univ-first}
  If $\CC$ is a corepita 2-category, and $\s \in \CC$ is a plentiful DOF classifier, then with $\pow \colon \s^\iso \to \s$ defined as above, we have:

  For each groupoid $U \in \CC$ and each $a \colon U \to \s$, the object $\bar{a}\pow \in \CC(U,\s)$ is a power object of $a$, where $\bar{a} \colon U \to \s^\iso$ is the unique morphism with $\bar{a}i = a$.
\end{propn}
Fix such a $U$ and $a$.
Proving this proposition amounts to constructing a monomorphism into $a \times (\bar{a} \pow)$, which we will do presently, and proving that it is universal, which will be done in Proposition~\ref{propn:s-power-univ-final} below.

Now, as in \S\ref{subsec:s-has-finite-lims}, we recast this problem inside of $\DOF(U)$ via the fully faithful anafunctor $\Gr^{U} \colon \CC(U,\s) \to \DOF(U)$.
Fix pullbacks
\begin{equation}\label{eq:pow-setup}
  \begin{tikzcd}
    A\ar[r, "\ddot{a}"]\ar[d, "\dot{a}"']\pb&\s_*\ar[d, "p"]\\
    U\ar[r, "a"]&\s
  \end{tikzcd}
  \qquad\text{and}\qquad
  \begin{tikzcd}
    P_A\ar[r, "\ddot{P}_A"]\ar[d, "\dot{P}_A"']\pb&\s_*\ar[d, "p"]\\
    U\ar[r, "\bar{a}\pow"]&\s.
  \end{tikzcd}
\end{equation}
and note that we have an induced pullback square
\[
  \begin{tikzcd}
    P_A\ar[rr, "\ddot{P}_A", bend left]\ar[d, "\dot{P}_A"']
    \ar[r, "\pi_A", dashed]\pb
    &\wt\pow\ar[r, "\ddot{\pow}"]\ar[d, "\dot{\pow}"']\pb&\s_*\ar[d, "p"]\\
    U\ar[r, "\bar{a}"]&\s^\iso\ar[r, "\pow"]&\s.
  \end{tikzcd}
\]

Now, by definition, the anafunctor $\Gr^U$ takes $a,\bar{a}\pow \in \CC(U,\s)$ to $\dot{a},\dot P_A \in \DOF(U)$, respectively.

Hence, we need to construct a morphism $j \colon \epsilon_A \to A \times_{U} P_A$ in $\DOF(U)$, where we have fixed a strict pullback square
\[
  \begin{tikzcd}
    A \times_UP_A \ar[r, "\pi_1"] \ar[d, "\pi_0"'] \pb & P_A \ar[d, "\dot{P}_A"]\\
    A \ar[r, "\dot{a}"] & U.
  \end{tikzcd}
\]
We also set $\pi \defeq \pi_0 \dot{a} = \pi_1\dot{P}_A \colon A \times_UP_A \to U$.

Fixing a universal stable monomorphism
\[
  \twocell{\wt\pow}{\s}{\tilde{\partial}_0}{\tilde{\partial}_1 = \dot{\pow} i}{\tilde{\partial}}
\]
as in Proposition~\ref{propn:pow-univ-prop}, we now set $\dot{e}_A = \Gr^{P_A}(\pi_A\tilde{\partial}_0) \colon \epsilon_A \to P_A$ and let $j_A \colon \epsilon_A \to A \times_{U} P_A$ be the morphism $\Gr^{P_A}(\pi_A\tilde{\partial})$ in $\DOF(P_A)$ corresponding to $\pi_A\tilde{\partial}$:
\[
  \begin{tikzcd}
    \epsilon_A\ar[dr, "j_A"']\ar[ddr, bend right, "\dot{e}_A"']
    \ar[drr, "\ddot{e}_A"{name=dde}, bend left=8pt]&&[5pt]\\
    &A \times_{U} P_A\ar[r, "\pi_0\ddot{a}"']\ar[d, "\pi_1"']&\s_*\ar[d, "p"]\\
    &P_A\ar[r, "\pi_A\dot{\pow} i"'{name=dp}, bend right, pos=0.47]
    \ar[r, "\pi_A\tilde{\partial}_0"{name=p1s}, bend left, pos=0.47]&\s,
    \ar[r, from=p1s, to=dp, Rightarrow, shorten=3pt, "\pi_A\tilde{\partial}"]
    \ar[r, from=dde, to=2-2, Rightarrow, shorten <=3pt, shorten >=-5pt]
  \end{tikzcd}
\]
where we are using that the DOF $\pi_1 \colon A \times_U P_A \to P_A$ is classified by
\[
  \dot{P}_Aa =
  \dot{P}_A\bar{a} i =
  \pi_A\dot{\pow} i.
\]
Since $\Gr^{P_A} \colon \CC(P_A,\s) \to \DOF(P_A)$ preserves, finite limits, it follows that $j_A$ is a monomorphism in $\DOF(P_A)$.
In fact, since $\dom \colon \DOF(P_A) \to \CC$ preserves and reflects pullbacks (see the proof of Proposition~\ref{propn:gr-preserves-reflects-limits}), this is equivalent to $j_A$ being a monomorphism in $\abs{\CC}$, and in particular, $j_A$ remains a monomorphism in $\DOF(U)$:
\[
  \begin{tikzcd}[column sep=0pt]
    \epsilon_A\ar[rr, >->, "j_A"]\ar[dr, "\dot{e}_A\dot{P}_A"']&&
    A \times_{U} P_A\ar[dl, "\pi"]\\
    &U.&
  \end{tikzcd}
\]

We can now state a more precise version of Proposition~\ref{propn:s-power-univ-first}:
\begin{propn}\label{propn:s-power-univ-final}
  If $\CC$ is a corepita 2-category, and $\s \in \CC$ is a plentiful DOF classifier, then with $\pow \colon \s^\iso \to \s$ defined as above, we have, for each groupoid $U \in \CC$ and each $(A\tox{\dot{a}}U) \in \DOF(U)$:

  The monomorphism $j_A \colon \epsilon_A\tto A \times_{U} P_A$ in $\DOF(U)$ defined above is universal, i.e., it makes $P_A \tox{\dot{P}_A} U$ into a power object of $A \tox{\dot{a}} U$ in $\DOF(U)$.
\end{propn}
\begin{proof}
  Fix DOFs $B\tox{\dot{b}}U$ and $R\tox{\dot{r}}U$ and a monomorphism
  \[
    \begin{tikzcd}[column sep=0]
      R\ar[rr, rightarrowtail, "k"]\ar[rd, "\dot{r}"']&&A \times_{U} B\ar[ld, "\pi"]\\
      &U
    \end{tikzcd}
  \]
  in $\DOF(U)$.
  From this, we obtain a morphism $v \colon B \to P_A$ over $U$ as follows.

  First, note that $\pi_1 \colon A \times_{U} P_A \to P_A$ is a DOF (being a pullback of a DOF), and so is $k$ (by the 2-of-3 property of DOFs) and hence also $k\pi_1 \colon R \to B$.
  Thus, we also have a monomorphism in $\DOF(B)$
  \[
    \begin{tikzcd}[column sep=0]
      R\ar[rr, rightarrowtail, "k"]\ar[rd, "k\pi_1"']&&A \times_{U} B\ar[ld, "\pi_1"]\\
      &B.
    \end{tikzcd}
  \]

  Now, $\pi_1$ is $\s$-small, being classified by $\dot{b} a \colon B \to \s$, and $k$ is $\s$-small, being a monic DOF, and hence $k\pi_1$ is $\s$-small as well.
  Letting $\rho \colon B \to \s$ be a morphism classifying $k\pi_1$, we thus have by the genericity of $p$ that $k$ is classified by a 2-cell $\kappa \colon B \tocellud{\rho}{\dot{b} a} \s$ which is a monomorphism since $\Gr^{B}$ is fully faithful; it is moreover a stable monomorphism by Proposition~\ref{propn:generic-dof-limits-are-stable}.

  Thus, by Proposition~\ref{propn:pow-univ-prop}, we obtain a morphism $u \colon B \to \tilde{\pow}$ (note that $B$ is groupoidal since $U$ is and $\dot{b}$ is a DOF) satisfying $u\dot{\pow} i = \dot{b}a$ and an isomorphism $\beta \colon u\tilde{\partial}_0 \toi \rho$ making the following square commute.
  \begin{equation}\label{eq:s-power-univ-final-square}
    \begin{tikzcd}
      u\tilde{\partial}_0\ar[r, "\sim"', "\beta"]\ar[d, >->, "u\tilde{\partial}"']&\rho\ar[d, >->, "\kappa"]\\
      u\dot{\pow} i\ar[r, equals]&\dot{b}a
    \end{tikzcd}
  \end{equation}

  We now obtain our desired morphism $v \colon B \to P_A$ from the diagram
  \[
    \begin{tikzcd}
      B\ar[ddr, "\dot{b}"', bend right]\ar[rd, dashed, "v"']
      \ar[rrd, "u", bend left=13pt]&&\\
      &P_A\ar[d, "\dot{P}_A"']\ar[r, "\pi_A"]\pb&\wt\pow\ar[d, "\dot{\pow}"]&\\
      &U\ar[r, "\bar{a}"]&\s^\iso.&
    \end{tikzcd}
  \]
  Note that, by construction, $v$ is a morphism over $U$ as required.

  Let us now see that $v$ has the required property, namely that we have a pullback square in $\DOF(U)$ (where the dotted morphism is sought):
  \begin{equation}\label{eq:pow-induced-pb-square}
    \begin{tikzcd}
      R\ar[r, dashed]\ar[d, "k", rightarrowtail]\pb&[10pt]\epsilon_A\ar[d, "j_A", rightarrowtail]\\
      A \times_{U} B\ar[r, "\id \times_{U} v"]&A \times_{U} P_A.
    \end{tikzcd}
  \end{equation}

  Consider the enlarged diagram
  \begin{equation}\label{eq:pow-induced-pb-square-enlarged}
    \begin{tikzcd}
      R\ar[rr, dashed]\ar[dr, "k"', rightarrowtail]&[-30pt]&[-30pt]
      v^*\epsilon_A\ar[r]\ar[dl, "v^*j_A"{near start}]\pb&[20pt]\epsilon_A\ar[d, "j_A", rightarrowtail]&[3pt]\\
      &A \times_{U} B\ar[rr, "\id_A \times_{U} v"]\ar[d, "\pi_1"']\pb
      &&A \times_{U} P_A\ar[d, "\pi_1"]&
      \\
      &B\ar[rr, "v"]&{}&
      P_A
      \ar[r, bend left, "\pi_A\tilde{\partial}_0"{name=p1s}, pos=0.45]
      \ar[r, bend right, "\dot{\pow} i"'{name=p2s}, pos=0.45]
      \ar[from=p1s, to=p2s, Rightarrow, "\pi_A\tilde{\partial}", shorten=3pt]&
      \s,
    \end{tikzcd}
  \end{equation}
  where we have chosen a pullback $v^*\epsilon_A$, and we now seek the dotted morphism making the triangle commute.
  But now, since by definition $j_A = \Gr^{P_A}(\pi_A\tilde{\partial})$, we have (with respect to appropriate specifications) that $v^*j_A = \Gr^B(v\pi_A\tilde{\partial})$ (\emph{cf.} Proposition~\ref{propn:generic-dof-limits-are-stable}).
  Since by definition $k = \Gr^B(\kappa)$, we may, by the commutativity of (\ref{eq:s-power-univ-final-square}), take the dotted morphism to be $\Gr^B(\beta)$ (where we recall that $u = v\pi_A$).

  Now suppose we had a second such morphism $v' \colon B \to P_A$, i.e., with $v'\dot{P}_A = \dot{b}$ and for which there is a pullback square (\ref{eq:pow-induced-pb-square}) with $v'$ in place of $v$.
  We would then have corresponding diagrams (\ref{eq:pow-induced-pb-square-enlarged}) and hence (\ref{eq:s-power-univ-final-square}).
  Thus $v' = v$ by the uniqueness part of Proposition~\ref{propn:pow-univ-prop}.
\end{proof}

\begin{thm}
  If $\CC$ is a corepita 2-category, then any plentiful DOF classifier $\s$ is an internal elementary topos.
\end{thm}
\begin{proof}
  In light of Corollary~\ref{cor:s-has-limits} and Proposition~\ref{propn:s-power-univ-first}, the only thing remaining to prove is that $f^* \colon \CC(U,\s) \to \CC(U',\s)$ preserves power objects for each morphism $f \colon U' \to U$ with $U,U'$ groupoidal.

  This amounts to showing that for each $\s$-small DOF $\dot{a} \colon A \to U$ and pullback square
  \[
    \begin{tikzcd}
      A'\ar[r, "g"]\ar[d, "\dot{a}'"']\pb&A\ar[d, "\dot{a}"]\\
      U'\ar[r, "f"]&U,
    \end{tikzcd}
  \]
  $f^* \colon \DOF(U) \to \DOF(U')$ takes some power object for $\dot{a}$ to one for $\dot{a}'$.

  Recalling the construction of $P_A$ in Proposition~\ref{propn:s-power-univ-first}, and defining $P_{A'}$ similarly, we see immediately that we have a pullback square
  \[
    \begin{tikzcd}
      P_{A'}\ar[r, "h", dashed]\ar[d, "\dot{P}_{A'}"']\ar[rr, bend left, "\ddot{P}_{A'}"]\pb&
      P_A\ar[d, "\dot{P}_A"]\ar[r, "\ddot{P}_{A}"]&\s_*\ar[d, "p"]\\
      U'\ar[r, "f"]\ar[rr, bend right, "\bar{a}'\pow"]&U\ar[r, "\bar{a}\pow"]&S
    \end{tikzcd}
  \]
  so that $f^*P_A = P_{A'}$.
  Thus, it remains only to see that $f^*$ takes the universal monomorphism for $P_A$ to the one for $P_{A'}$, i.e., that we have a dotted morphism forming a pullback square
  \[
    \begin{tikzcd}
      \epsilon_{A'}\ar[r, dashed]\ar[d, "j_{A'}", rightarrowtail]\pb
      &[20pt]\epsilon_A\ar[d, "j_A", rightarrowtail]&\\
      A' \times_{U'} P_{A'}\ar[r, "g \times_{f} h"]\ar[d, "\pi_1"]\pb
      &A \times_{U} P_A\ar[d, "\pi_1"]&
      \\
      P_{A'}\ar[r, "h"]&
      P_A
      \ar[r, bend left, "\pi_A\tilde{\partial}_0"{name=p1s}, pos=0.45]
      \ar[r, bend right, "\pi_A\dot{\pow} i"'{name=p2s}, pos=0.45]
      \ar[from=p1s, to=p2s, Rightarrow, "\pi_A\tilde{\partial}", shorten=3pt]&
      \s.
    \end{tikzcd}
  \]
  Recalling that $j_A = \Gr^{P_A}(\pi_A\tilde{\partial})$ and $j_{A'} = \Gr^{P_{A'}}(\pi_{A'}\tilde{\partial})$, the claim follows since $h\pi_{A} = \pi_{A'}$ (\emph{cf.} the last part of the proof of Proposition~\ref{propn:s-power-univ-final}).
\end{proof}

We end this section by showing that a plentiful DOF classifier is ``essentially unique'' in the presence of Axiom~\ref{item:ax-UA} from \S\ref{sec:axioms}.
\begin{thm}\label{thm:plentiful-dof-ess-unique}
  Let $\CC$ be a corepita 2-category.
  \begin{enumerate}[(i)]
  \item Let $p_i \colon \s_{i*} \to \s_i$ ($i=1,2$) be plentiful generic DOFs in $\CC$, and suppose that $p_1$ is $p_2$-small, and let $f \colon \s_1 \to \s_2$ be a morphism classifying $p_1$.
    Then $f$ is fully faithful and logical (meaning that $f_* \colon \CC(A,\s_1) \to \CC(A,\s_2)$ is a logical functor for all groupoids $A \in \CC$).
  \item Suppose $A \in \CC$ is such that every DOF over $A$ is $\s$-small for \emph{some} plentiful DOF classifier $\s$.
    Then $\DOF(A)$ is an elementary topos, and for each plentiful DOF classifier $\s$, the anafunctor $\Gr^A \colon \CC(A,\s) \to \DOF(A)$ is a logical functor.
  \end{enumerate}
\end{thm}
\begin{proof}
  Regarding (i), one verifies that the triangle of anafunctors
  \[
    \begin{tikzcd}[row sep=5pt]
      \CC(A,\s_1)\ar[dd, "f_*"']\ar[rd, "\Gr^A"{pos=0.4}]\\
      &\DOF(A)\\
      \CC(A,\s_2)\ar[ru, "\Gr^A"'{pos=0.4}]
    \end{tikzcd}
  \]
  commutes up to isomorphism for each $A \in \CC$.
  Since the two horizontal arrows are fully faithful, it follows that $f_*$ is as well.

  Similarly, since both horizontal arrows preserve finite limits and power objects (as follows from Propositions~\ref{propn:gr-preserves-reflects-limits}~and~\ref{propn:s-power-univ-final}), $f_*$ does as well.

  Regarding (ii), we have by Proposition~\ref{propn:gr-preserves-reflects-limits} that for \emph{any} $A$, $\DOF(A)$ has finite limits and $\Gr^A$ preserves them.
  And as we just mentioned, $\Gr^A$ also preserves power objects; hence, if every object of $\DOF(A)$ is in the image of $\Gr^A$ for some $\s$, it follows that $\DOF(A)$ has power objects.
\end{proof}

\section{Sufficiency of groupoids}\label{sec:suff-of-groupoids}
In this section, we express in two ways that a groupoidant 2-category $\CC$ is ``fully determined'' by its groupoidal objects.

The first way (\S\ref{subsec:internal-cats}) shows that $\CC$ is (strictly) equivalent to the 2-category of complete congruences in $\CC$ (a sub-2-category of the 2-category of internal categories in $\abs{\CC}$).
We also prove a variant of this, which says that $\CC$ admits a pita embedding into the 2-category of internal categories in a finitely complete category, and we apply this to show that we can define a duality involution in the sense of \cite{weber-2-toposes} on $\CC$.

The second way, (\S\ref{subsec:2-yoneda}) shows that the full sub-2-category $\CC_\gpd$ of $\CC$ on the groupoidal objects is \emph{dense}, in the sense that the 2-Yoneda embedding $\CC \to 2\Fun(\CC^\op,\Cat)$ remains an embedding after composing with the restriction 2-functor $\Fun(\CC^\op,\Cat) \to 2\Fun(\CC^\op_\gpd,\Cat)$.
Thus, we can define objects, morphisms, and 2-cells of $\CC$ ``representably'' simply by specifying a 2-functor (or strict natural transformation or modification, respectively) on the sub-2-category $\CC_{\gpd}$.

In \S\ref{subsec:opposites}, we apply this result to characterize the opposite of an object in $\CC$ (i.e., its image under the aforementioned duality involution), and in Appendix~\ref{sec:groupoid-stuff}, we apply it to ``legitimize'' the Definition~\ref{defn:power-objs} of internal topoi, which is given in terms of groupoids.

\subsection{Internal categories and the duality involution}\label{subsec:internal-cats}
\begin{defn}
  For any 1-category $\cC$, we write $\Cat(\cC)$ for the 2-category of internal categories in $\cC$, functors between them, and natural transformations between these (see, e.g., \cite[Definition~5.2]{helfer-sentai}).

  For a 2-category $\CC$, we define $\Cong(\CC) \subset \Cat(\abs{\CC})$ to be the full sub-2-category on the complete congruences in $\CC$.
\end{defn}

\begin{defn}
Given a groupoidant 2-category $\CC$, we define an ana-2-functor $\Nv \colon \CC \to \Cong(\CC)$ as follows.

On objects, we define $\Nv(X)$ to be ``the'' nerve
\[
  \begin{tikzcd}
    (X^{[2]})^\iso {
      \ar[r,shift left=12pt, "\pi_{01}"]
      \ar[r,shift right=12pt, "\pi_{12}"]
      \ar[r, "\pi_{02}"]
    }&
    (X^\to)^\iso {
      \ar[r, shift left=12pt, "\pi_0"]
      \ar[r, shift right=10pt, "\pi_1"]
      \ar[from=r, "e"']
    }&
    X^\iso.
  \end{tikzcd}
\]
More precisely, a specification for $\Nv$ at $X$ consists of a choice of core $X^\iso$ and cotensor-cores $(X^\to)^\iso$ and $(X^{[2]})^\iso$ (together with the associated data, $i \colon X^\iso \to X$, $i \partial \colon (X^\to)^\iso \tocell X$, and $\pi_{01} i \partial, \pi_{12} i \partial \colon (X^{[2]})^\iso \tocell X$).

Given a morphism $f \colon X \to Y$, we obtain an induced morphism $f^{[i]} \colon X^{[i]} \to Y^{[i]}$ for any $i$, and hence $(f^{[i]})^\iso \colon (X^{[i]})^\iso \to (Y^{[i]})^\iso$, and it is clear that these form an internal functor $\Nv(f) \colon \Nv(X) \to \Nv(Y)$.
It is also clear this makes $\Nv$ into a (1-)functor $\CC \to \Cong(\CC)$.

Finally, given a 2-cell $\alpha \colon X \tocellud{f}{g} Y$, we obtain an internal natural transformation $\Nv(\alpha) \colon \Nv(f) \to \Nv(g)$ as follows.
Such a natural transformation is by definition a morphism $\Nv(\alpha) \colon X^\iso \to (Y^\to)^\iso$, which we take to be the unique morphism with $\Nv(\alpha) i \partial = i \alpha$.
This morphism then by definition satisfies $\Nv(\alpha) \cdot \pi_1 = \Nv(f)_0$ and $\Nv(\alpha) \cdot \pi_1 = \Nv(g)_0$.

The remaining condition for $\Nv(\alpha)$ to be a natural transformation follows from the equality of the two ways of writing the horizontal composite of the 2-cells
\[
  \begin{tikzcd}
    (X^{[2]})^\iso {
      \ar[r, bend left, "i\partial_0"{name=sp}, pos=0.45,
          start anchor=north east, end anchor=north west]
      \ar[r, bend right, "i\partial_1"'{name=id}, pos=0.45,
          start anchor=south east, end anchor=south west]
    }&
    X {
      \ar[r, bend left, "f"{name=p1}, pos=0.5,
          start anchor=north east, end anchor=north west]
      \ar[r, bend right, "g"'{name=p2}, pos=0.5,
          start anchor=south east, end anchor=south west]
    }&
    Y.
    \ar[from=sp, to=id, "i\partial", "\sim"' sloped, Rightarrow, shorten=3pt]
    \ar[from=p1, to=p2, "\alpha", Rightarrow, shorten=3pt]
  \end{tikzcd}
\]
as a composition of two whiskerings.

The verification that $\Nv$, thus defined is (not only a 1-functor but) a 2-functor is a (somewhat lengthy but) straightforward exercise that we leave to the reader.
\end{defn}

\begin{thm}\label{thm:internal-cats-equivalence}
  If $\CC$ is a groupoidant 2-category, then the 2-functor $\Nv \colon \CC \to \Cong(\CC)$ is a strict equivalence of 2-categories.
\end{thm}

\begin{proof}
  It follows immediately from Axiom~\ref{item:ax-N} in Definition~\ref{defn:groupoidant} that $\Nv$ is essentially surjective.
  Hence, it remains to show that it is strictly fully faithful.

  Given an internal functor $f \colon \Nv(X) \to \Nv(Y)$, composing with the tautological cocone under $\Nv(Y)$ gives a cocone under $\Nv(X)$ with vertex $Y$, and hence (by Axiom~\ref{item:ax-Q}) a morphism $\bar{f} \colon X \to Y$, and one verifies that $f = \Nv(\bar{f})$ and that $\ol{\Nv(g)} = g$ for $g \colon X \to Y$.
  This shows that $\Nv \colon \CC(X,Y) \to \Cong(\CC)(\Nv(X),\Nv(Y))$ is bijective on objects.

  It remains to check, given $f,g \colon X \to Y$, that $\Nv$ gives a bijection from 2-cells $f \to g$ to internal natural transformations $\Nv(f) \to \Nv(g)$.

  Suppose we have an internal natural transformation $\alpha \colon \Nv(f) \to \Nv(g)$.
  We now define a 2-cell $\bar{\alpha} \colon f \to g$ using (the 2-categorical part of) the universal property of the quotient $X$ of $\Nv(X)$.
Recalling that $\alpha$ is by definition a morphism $\alpha \colon X^\iso \to (Y^\to)^\iso$ (satisfying certain conditions), we have a 2-cell
$\alpha i\partial \colon X^\iso \tocellud{if}{ig} Y$.
  To obtain $\bar{\alpha}$, it remains to show that the following square in $\CC\plbig{(X^\to)^\iso,Y}$ commutes.
  \[
    \begin{tikzcd}
      \partial_0^\iso i f {
        \ar[r, "i \partial f"]
        \ar[d, "\partial_0^\iso \alpha i \partial"']
      }&
      \partial_1 i f {
        \ar[d, "\partial_1^\iso \alpha i \partial"]
      }
      \\
      \partial_0 i g {
        \ar[r, "i \partial g"]
      }&
      \partial_1 i g
    \end{tikzcd}
  \]
  However, the commutativity of this square is precisely the naturality condition for $\alpha$.

  We thus obtain $\bar{\alpha} \colon f \to g$ satisfying $i\bar{\alpha} = \alpha i\partial$, and it is immediate from the definitions that $\alpha\mapsto\bar{\alpha}$ is an inverse to $\alpha\mapsto\Nv(\alpha)$.
\end{proof}

Let us call a square in a 2-category $\CC$ a \emph{1-pullback square} if it is a pullback square in $\abs{\CC}$.
Similarly, let us call an object $X^\to \in \CC$ with a 2-cell $\partial \colon X^\to \tocell X$ a \emph{1-arrow object} if $\partial_* \colon \CC(A,X^\to) \to \CC(A,X)^\to$ is a \emph{bijection on objects} for all $A \in \CC$.

As we noted in \S\ref{subsec:limits}, if a given cospan \emph{has} a strict pullback, then any 1-pullback of it will also be a strict pullback.
Similarly, if $X$ has a strict arrow object, any 1-arrow object $X^\to$ will be a strict arrow object.

\begin{lem}\label{lem:nv-preserves}
  For any groupoidant 2-category $\CC$, the ana-2-functor $\Nv \colon \CC \to \Cat(\abs{\CC})$ preserves terminal objects, 1-arrow objects, and 1-pullback squares.
\end{lem}

\begin{proof}
  That $\Nv$ preserves terminal objects is obvious.

  Next, we note that the functor $U \colon \absbig{\Cat(\abs{\CC})} \to \abs{\CC}^3$ taking $C$ to $(C_0,C_1,C_2)$ reflects finite limits \cite[(3.6)]{street-cosmoi-of-internal-cats}, hence it suffices to show that the functor $\abs{\Nv} \circ U \colon \abs{\CC} \to \abs{\CC}^3$ preserves pullbacks.

  But now by Lemma~\ref{lem:cotensor-preserves-pullbacks} below, if we start with a strict pullback square $P$ in $\CC$, the resulting squares (using hopefully self-evident notation) $P^\to$ and $P^{[2]}$ will still be pullback squares, and by Lemma~\ref{lem:core-preserves-pullbacks}, so will the squares $P^\iso$, $(P^\to)^\iso$, and $(P^{[2]})^\iso$.
  Hence, we see that the square $\Nv(P)$ of internal categories will map under $U$ to a pullback square in $\abs{\CC}^3$, as desired.

  Next, let us see that $\Nv$ preserves 1-arrow objects, i.e., that $\Nv(\partial) \colon \Nv(X^\to) \tocell \Nv(X)$ is a 1-arrow object in $\Cat(\abs{\CC})$ for any $X \in \CC$.

  So we fix an arbitrary internal category $C \in \Cat(\abs{\CC})$, and we suppose that we are given an internal natural transformation $\alpha \colon C \tocellud{f}{g} \Nv(X)$, i.e., a morphism $\alpha \colon C_0 \to \Nv(X)_1 = (X^\to)^\iso$ (satisfying the relevant conditions).
  We want to show there is a unique internal functor $F \colon C \to \Nv(X^\to)$ with $F \cdot \Nv(\partial) = \alpha$.

  From the definitions, it follows that we must take $F_0 = \alpha \colon C_0 \to \Nv(X^\to)_0 = (X^\to)^\iso = \Nv(X)_1$, whereupon (no matter how we define $F_1$), we will have $F \cdot \Nv(\partial) = \alpha$, as desired.

  Next, we define $F_1 \colon C_1 \to \Nv(X^\to)_1 = \pbig{(X^\to)^\to}^\iso$.
  This is uniquely determined by $F_1 i\partial \colon C_1 \tocell X^\to$, which is in turn determined by the commutative square in $\CC(C_1,X)$ shown below on the left, which we define as shown on the right.
  \[
    \begin{tikzcd}
      F_1 i \partial_0 \partial_0
      \ar[r, "F_1i\partial\partial_0"]
      \ar[d, "F_1i\partial_0\partial"']&[20pt]
      F_1 i \partial_1 \partial_0
      \ar[d, "F_1i\partial_1\partial"]\\
      F_1 i \partial_0 \partial_1
      \ar[r, "F_1i\partial\partial_1"]&
      F_1 \partial_1 \partial_1
    \end{tikzcd}
    \qquad\qquad
    \begin{tikzcd}
      \pi_0f_0i\ar[r, "f_1i\partial"]\ar[d, "\pi_0\alpha i\partial"']&
      \pi_1f_0i\ar[d, "\pi_1\alpha i\partial"]\\
      \pi_0g_0i\ar[r, "g_1i\partial"]&\pi_1g_0i
    \end{tikzcd}
  \]
  In fact, the choice of vertical arrows is forced by the conditions $F_1\pi_j = \pi_j F_0 = \pi_j\alpha$ for $j = 0,1$, and the horizontal arrows are forced by the conditions $F \cdot \Nv(\partial_0) = f$ and $F \cdot \Nv(\partial_1) = g$ (it follows that $F$ is uniquely determined).
  That this square actually commutes is precisely the naturality condition for $\alpha$.

  It remains to see that $F$, thus defined, is an internal functor.
  To show that $F_2\pi_{02} = \pi_{02}F_1 \colon C_2 \to \Nv(X^\to)_1$, it suffices to show the equality of the (parallel) 2-cells $F_2\pi_{02}i\partial,\pi_{02}F_1i\partial \colon C_2 \tocell X^\to$, which would in turn follow from the equalities $F_2\pi_{02}i\partial\partial_j = \pi_{02}F_1i\partial\partial_j \colon C_2 \to X$ for $j = 0,1$.
  But now using that $i\partial\partial_j = (\partial_j^\to)^\iso i\partial = \Nv(\partial_j)_1 i\partial$, these last equalities follow from $f_2\pi_{02} = \pi_{02}f_1,g_2\pi_{02} = \pi_{02}g_1 \colon C_2 \to \Nv(X)_1$ (functoriality of $f$ and $g$).
  The equation $F_0e = eF_1 \colon C_0 \to \Nv(X^\to)_1$ is proven similarly.
\end{proof}

\begin{rmk}
  If $\Cat(\abs{\CC})$ is pita (for example, if $\CC$ itself has all strict finite limits \cite[(4.4)]{street-cosmoi-of-internal-cats}), it follows from Lemma~\ref{lem:nv-preserves} that $\Nv$ is pita (\emph{cf.} Theorem~\ref{thm:internal-cat-embedding} below).
  We have not checked whether $\Nv$ is pita in general.
\end{rmk}

\begin{lem}\label{lem:cotensor-preserves-pullbacks}
  Given a strict pullback square
  \[
    \begin{tikzcd}
      A\ar[r, "h"]\ar[d, "k"']\pb&B\ar[d, "f"]\\
      C\ar[r, "g"]&D
    \end{tikzcd}
  \]
  in a 2-category $\CC$ and a 1-category $J$, if the objects $A,B,C,D$ all have strict cotensors with $J$, then the induced square
  \[
    \begin{tikzcd}
      A^J\ar[r, "h^J"]\ar[d, "k^J"']&B^J\ar[d, "f^J"]\\
      C^J\ar[r, "g^J"]&D^J
    \end{tikzcd}
  \]
  is also a strict pullback square.
\end{lem}

\begin{proof}
  For any object $U \in \CC$, the universal properties of the pullback $A$ and the cotensors $A^J,B^J,C^J,D^J$ give the following chain of isomorphisms of categories
  \[
    \begin{split}
      \CC(U,A^J)&\toi
      \Fun\pbig{J,\CC(U,A)}\\&\toi
      \Fun\pbig{J,\CC(U,B) \times_{\CC(U,D)} \CC(U,C)}\\&\toi
      \Fun\pbig{J,\CC(U,B)} \times_{\Fun\pbig{J,\CC(U,D)}} \Fun\pbig{J,\Cat(U,C)}\\&\toi
      \CC(U,B^J) \times_{\CC(U,D^J)} \Cat(U,C^J)
    \end{split}
  \]
  Moreover, it is not hard to check that the image of an object or morphism $f$ in $\CC(U,A^J)$ under the composite isomorphism is $(fh^J,fk^J)$, which proves the claim.
\end{proof}

We deduce the following ``embedding theorem'':
\begin{thm}\label{thm:internal-cat-embedding}
  For any groupoidant 2-category $\CC$, there exists a finitely complete 1-category $\cC$ and a strictly fully faithful pita 2-functor $\CC \to \Cat(\cC)$.
\end{thm}
\begin{proof}
  If $\abs{\CC}$ itself is finitely complete, we can simply take $\cC = \abs{\CC}$ and take the 2-functor $\Nv \colon \CC \to \Cat(\cC)$.

  In general, we fix a category $\cC$ with finite limits and a finite-limit preserving fully faithful functor $F \colon \abs{\CC} \to \cC$ (for example, take $\cC = \Set^{\abs{\CC}^\op}$---or more conservatively, the full subcategory thereof consisting of finite limits of representables).
  Since $F$ is finite-limit preserving, it induces a 2-functor $\Cat(F) \colon \Cat(\abs{\CC}) \to \Cat(\cC)$, which also preserves finite 1-categorical limits (since the forgetful functors $\abs{\Cat(\abs{\CC})} \to \abs{\CC}^3$ and $\abs{\Cat(\cC)} \to \abs{\cC}^3$ preserve and reflect 1-limits \cite[(3.6)]{street-cosmoi-of-internal-cats}).
  Since $F$ is fully faithful, it follows that $\Cat(F)$ is strictly fully faithful.

  Thus, it follows from Theorem~\ref{thm:internal-cats-equivalence} that the composite $\Nv \circ \Cat(F) \colon \CC \to \Cat(\cC)$ is strictly fully faithful, and it follows from Lemma~\ref{lem:nv-preserves} that it preserves terminal objects, 1-arrow objects, and 1-pullback squares.
  Since $\Cat(\cC)$ has all finite limits \cite[(4.4)]{street-cosmoi-of-internal-cats}, it follows that $\Nv \circ \Cat(F)$ preserves all pita limits.
\end{proof}

\begin{cor}\label{cor:duality-involution}
  If $\CC$ is groupoidant, then $\CC$ admits a duality involution $\CC^\co \to \CC$ in the sense of \cite[Definition~2.14]{weber-2-toposes}.
\end{cor}

\begin{proof}
  Fix a pita strictly fully faithful 2-functor $\cF \colon \CC \to \Cat(\cC)$ given as a composite $\CC \toi \Cong(\CC) \hto \Cat(\abs{\CC}) \tox{\Cat(F)} \Cat(\cC)$ as in the proof of Theorem~\ref{thm:internal-cat-embedding}, where $\cC$ is a finitely complete 1-category and $F \colon \abs{\CC} \to \cC$ is a functor preserving finite limits.

  According to \cite[Example~2.16]{weber-2-toposes}, the 2-category $\Cat(\cC)$ admits a duality involution $D \colon \Cat(\cC)^\co \to \Cat(\cC)$, whose action on objects is given by interchanging the source and target maps of internal categories.
  Since the essential image of $\cF$ is clearly invariant under $D$, we obtain an induced 2-functor $D' \colon \Cong(\CC)^\co \to \Cong(\CC)$.
  (There is a slight technicality here since \cite{weber-2-toposes} requires a duality involution to be a \emph{strict} involution, which $D'$ might not be; but this can be arranged using the axiom of choice.)

  We now want to see that the duality involution structure on $D$ induces one on $D'$.
  This structure on $D$ is given by a family of equivalences
  \begin{equation}\label{eq:duality-inv-equiv}
    \DFib(\Cat(\cC))(A \times B,C) \toi
    \DFib(\Cat(\cC))(A, D(B) \times C)
  \end{equation}
  for $A,B,C \in \Cat(\cC)$, where $\DFib$ denotes the category of 2-sided discrete fibrations (2SDFs) \cite[\S2.3]{weber-2-toposes}.
  Now $\cF$ preserves products and, by Remark~\ref{rmk:dofs-internally}, preserves DOFs (and preserves 2SDFs for similar reasons), 
  hence it suffices to show that the equivalences (\ref{eq:duality-inv-equiv}) take 2SDFs in the image of $\cF$ to 2SDFs in the image of $\cF$.

  Thus, given a 2SDF $A \times B \ot E \to C$ in $\CC$, and the 2SDF $\cF(A) \ot E' \to D(\cF(B)) \times \cF(C)$ obtained by applying (\ref{eq:duality-inv-equiv}) to its image under $\cF$, we need to show that $E'$ is in the essential image of $\cF$.

  Referring to the description of $E'$ in \cite[Example~2.16]{weber-2-toposes}, we see that $E'_0 = \cF(E)_0$, and so it remains to see that $E_1'$ is in the image of $F$ and that the resulting object of $\Cat(\abs{\CC})$ is a complete congruence.

  Now $E_1'$ is given as a pullback
  \[
    \begin{tikzcd}
      E_1' \ar[r, ""] \ar[d, ""'] \pb & \cF(A)_1 \times_{\cF(A)_0} \cF(E)_0 \ar[d, "\ell"] \\
      \cF(E)_0 \times_{\cF(B)_0 \times \cF(C)_0} (\cF(B)_1 \times \cF(C)_1) \ar[r, "\ell'"'] & \cF(E)_0
    \end{tikzcd}
  \]
  in $\cC$, where the morphisms $\ell$ and $\ell'$ are the ones given by the lifting condition in the 2SDF $\cF(A) \times \cF(B) \ot \cF(E) \to \cF(C)$.
  Now, one can check that the morphism $A^\to \times_B E \to E$ in $\CC$ corresponding to the right side of the above square is an isofibration, and hence so is the induced morphism on cores $(A^\to)^\iso \times_{A^\iso} E^\iso \to E^\iso$.
  Since the right-hand morphism in the above square is the image under $F$ of the latter morphism, it follows that the corresponding pullback square exists in $\CC$, and hence that $E_1'$ is indeed in the essential image of $F$; say $E_1' = F(E_1'')$ with $E_1'' \in \CC$.
  One can also check directly that the resulting morphism $E_1'' \to E^\iso \times E^\iso$ is a DOF.

  We thus have a gbo-congruence $E''$ in $\CC$.
  To see that it is complete, we note that under the isomorphism
  \[
    E''_1 \toi
    \pbigg{
      \pBig{E \times_{B \times C} \pbig{B^\to \times C^\to}}
      \times_{E}
      (A^\to \times_A E)
    }^\iso,
  \]
  a given element of $\CC(U,E_1'')$ (with $U \in \CC$) is a horizontal isomorphism if and only if its projections to $\CC(U,(A^\to)^\iso)$, $\CC(U,(B^\to)^\iso)$, and $\CC(U,(C^\to)^\iso)$ are.
  Using the corresponding description of $\Nv(E)_1 = (E^\to)^\iso$ in terms of $A$, $B$, and $C$, and the fact that $\Nv(E)_0 = E''_0$, the completeness of $E''$ follows from that of $\Nv(E)$.
\end{proof}

\subsection{The restricted 2-Yoneda embedding}\label{subsec:2-yoneda}
We write $2\Fun(\CC^\op,\Cat)$ for the 2-category of (strict) 2-functors, strict natural transformations, and modifications.
We have the strictly fully faithful (strict) Yoneda embedding $\CC \to 2\Fun(\CC^{\op},\Cat)$ taking $X \in \CC$ to $\wh{X} \defeq \CC(-,X)$.
Restriction to the full sub-2-category $\CC_{\gpd}$ on the groupoids in $\CC$ gives a 2-functor $2\Fun(\CC^{\op},\Cat) \to 2\Fun(\CC_\gpd^{\op},\Cat)$, and thus by composing we have 2-functor $\CC \to 2\Fun(\CC_\gpd^{\op},\Cat)$.
We write $(\wh{X})_\gpd$ for the image of $X$ under this 2-functor.

\begin{thm}\label{thm:gpd-yoneda-rstr}
  If $\CC$ is groupoidant, then the restricted Yoneda 2-functor $\CC \to 2\Fun(\CC_\gpd^{\op},\Cat)$ is strictly fully faithful.
  In detail:
  \begin{enumerate}[(i)]
  \item Two morphisms $f,g \colon X \to Y$ in $\CC$ are equal if and only if the functors $(f_*)_A,(g_*)_A \colon \CC(A,X) \to \CC(A,Y)$ are equal for all groupoids $A \in \CC$.
  \item Two 2-cells $\alpha,\beta \colon X \tocellud{f}{g} Y$ are equal if and only if the natural transformations $(\alpha_*)_A,(\beta_*)_A \colon (f_*)_A \to (g_*)_A$ are equal for all groupoids $A \in \CC$.
  \item Let $X,Y \in \CC$ and let $F \colon (\wh{X})_\gpd \to (\wh{Y})_\gpd$ be a strict natural transformation, i.e., a family of functors $F_A \colon \CC(A,X) \to \CC(A,Y)$ for all groupoids $A \in \CC$ which are natural with respect to all morphisms $u \colon A' \to A$ (i.e., $F_{A}u^* = u^*F_{A'} \colon \CC(A,X) \to \CC(A',Y)$---see also Remark~\ref{rmk:yoneda-surprise}).

    Then there is a (by (i) unique) morphism $f \colon X \to Y$ with $F_A = (f_*)_A$ for all $A \in \CC_{\mathrm{gpd}}$.
  \item Let $f,g \colon X \to Y$ be morphisms in $\CC$ and let $\alpha \colon f_* \to g_*$ be a modification, i.e., a family of natural transformations $\alpha_A \colon (f_*)_A \to (g_*)_A$ for all groupoids $A \in \CC$, which are natural with respect to all morphisms $h \colon A' \to A$ (i.e., $\alpha_{A}h^* = h^*\alpha_{A'} \colon \CC(A,X) \tocell \CC(A',Y)$).

    Then there is a (by (ii) unique) 2-cell $\alpha \colon f \to g$ with $\alpha_A = (\alpha_*)_A$ for all groupoids $A \in \CC$.
  \end{enumerate}
\end{thm}

\begin{proof}
  Given $X,Y \in \CC$, form a nerve $\Nv(X)$ of $X$.
  Then the statements (i) and (ii) follow immediately from the uniqueness part of the universal property of $X$ being a quotient of $\Nv(X)$ (since $\Nv(X)$ consists of groupoids).

  Now suppose we are given functors $F_A$ for each groupoid $A \in \CC$ as in (iii).
  Again form a nerve $\Nv(X)$:
  \[
    \begin{tikzcd}
      (X^{[2]})^\iso\ar[r,shift left=10pt, "\pi_{01}"]
      \ar[r,shift right=10pt, "\pi_{12}"]\ar[r, "\pi_{02}"]&
      (X^\to)^\iso\ar[r, shift left=10pt, "\pi_0"]
      \ar[r, shift right=10pt, "\pi_1"]\ar[from=r, "e"']
      \ar[d, shift right=7pt, shorten <=5pt, "\pi_0\gamma"'{name=0}]
      \ar[d, shift left=7pt, shorten <=5pt, "\pi_1\gamma"{name=1}]&
      X^\iso\ar[ld, "\gamma", bend left, shorten <=8pt]\\[10pt]
      &X
      \ar[from=0, to=1, Rightarrow, shorten=1pt, "\gamma_{01}"]
    \end{tikzcd}
  \]
  and let us set $C_i \defeq \Nv(X)_i = (X^{[i]})^\iso$.
  We obtain another cocone under $\Nv(X)$, with vertex $Y$:
  \[
    \begin{tikzcd}[column sep=80pt, row sep=50pt]
      (X^{[2]})^\iso\ar[r,shift left=10pt, "\pi_{01}"]
      \ar[r,shift right=10pt, "\pi_{12}"]\ar[r, "\pi_{02}"]&
      (X^\to)^\iso\ar[r, shift left=10pt, "\pi_0"]
      \ar[r, shift right=10pt, "\pi_1"]\ar[from=r, "e"']
      \ar[d, shift right=18pt, shorten <=5pt, "\pi_0 \cdot F_{C_0}(\gamma)=F_{C_1}(\pi_0\gamma)"', ""{name=0}]
      \ar[d, shift left=18pt, shorten <=5pt, "F_{C_1}(\pi_1\gamma)=\pi_1 \cdot F_{C_0}(\gamma)", ""'{name=1}]&
      X^\iso\ar[ld, "F_{C_0}(\gamma)", bend left=45pt, shorten <=8pt]\\[10pt]
      &Y.
      \ar[from=0, to=1, Rightarrow, shorten=1pt, "F_{C_1}(\gamma_{01})"]
    \end{tikzcd}
  \]
  This is indeed a cocone; using the naturality in $A$ of the collection of functors $F_A$:
  \[
    \begin{split}
      e \cdot F_{C_1}(\gamma_{01})&=
      F_{C_0}(e\gamma_{01})=
      F_{C_0}(\id_\gamma)=
      \id_{F_{C_0}\gamma}\\
      \pbig{\pi_{01} \cdot F_{C_1}(\gamma_{01})}\pbig{\pi_{12} \cdot F_{C_1}(\gamma_{01})}&=
      F_{C_2}(\pi_{01}\gamma_{01}) \cdot F_{C_2}(\pi_{12}\gamma_{01})\\&=
      F_{C_2}\pbig{(\pi_{01}\gamma_{01})(\pi_{12}\gamma_{01})}=
      F_{C_2}(\pi_{02}\gamma_{01})=
      \pi_{02} \cdot F_{C_1}(\gamma_{01}).
    \end{split}
  \]
  Thus, we obtain a morphism $f \colon X \to Y$ with $\gamma f = F_{C_0}(\gamma)$ and $\gamma_{01}f = F_{C_1}(\gamma_{01})$.
  It remains to check that $f$ has the desired property $F_A = f_*$ for all groupoids $A$.

  Fix a groupoid $A$ and a morphism $u \colon A \to X$; this factors through a unique $\bar{u} \colon A \to X^\iso$ since $A$ is a groupoid.
  We then have
  \[
    F_A(u) =
    F_A(\bar{u}\gamma) =
    \bar{u}\cdot F_{C_0}(\gamma) =
    \bar{u} \gamma f =
    uf
  \]
  as required.
  Similarly, given a 2-cell $\alpha \colon A \tocellud{u}{v} X$, we obtain a morphism $a \colon A \to (X^{[1]})^\iso$ with $a\gamma_{01} = \alpha$.
  Then
  \[
    F_A(\alpha) =
    F_A(a\gamma_{01}) =
    a \cdot F_{C_1}(\gamma_{01}) =
    a\gamma_{01}f =
    \alpha f.
  \]

  Next, given $f,g \colon X \to Y$ and a collection of natural transformations $\alpha_A \colon f_* \to g_*$ as in (iv), we obtain a 2-cell $(\alpha_{C_0})_{\gamma} \colon \gamma f \to \gamma g$.
  To see that this is a morphism of cocones $(\gamma f,\gamma_{01}f) \to (\gamma g,\gamma_{01}g)$, we need to verify the commutativity of the following square on the left.
  \[
    \begin{tikzcd}[column sep=40pt]
      \pi_0\gamma f\ar[r, "\pi_0(\alpha_{C_0})_\gamma"]\ar[d, "\gamma_{01}f"']&\pi_0\gamma g\ar[d, "\gamma_{01}g"]\\
      \pi_1\gamma f\ar[r, "\pi_1(\alpha_{C_0})_\gamma"]&\pi_1\gamma g
    \end{tikzcd}
    \qquad
    \begin{tikzcd}[column sep=40pt]
      \pi_0\gamma f\ar[r, "(\alpha_{C_1})_{\pi_0\gamma}"]\ar[d, "f_*\gamma_{01}"']&\pi_0\gamma g\ar[d, "g_*\gamma_{01}"]\\
      \pi_1\gamma f\ar[r, "(\alpha_{C_1})_{\pi_1\gamma}"]&\pi_1\gamma g
    \end{tikzcd}
  \]
  This is equal to the square on the right using the naturality in $A$ of the collection of natural transformations $\alpha_A$.
  But the square on the right commutes by the naturality of $\alpha_{C_1}$.
\end{proof}

\begin{rmk}\label{rmk:yoneda-surprise}
  There is an additional condition in the definition of $\set{F_A}_{A \in \CC}$ being a strict natural transformation which we did not mention in (iii) above, as it is not used in the proof: namely, that the $F_A \colon \CC(A,X) \to \CC(A,Y)$ should also be natural with respect to 2-cells $\alpha \colon A' \tocell A$ (i.e., $F_{A} \alpha^* = \alpha^* F_{A'} \colon \CC(A,X) \tocell \CC(A',Y)$).

  That this is not needed is somewhat surprising, since for a general 2-category $\CC$ and $X,Y \in \CC$, when proving that every strict natural transformation $\wh{X} \to \wh{Y}$ is of the form $f_*$ for some $f \colon X \to Y$, one \emph{does} need this condition.
  The point is that for a \emph{groupoidant} 2-category, this extra condition holds automatically.
  The same is true for an $\cF_\bo$-exact category (in the sense of Axiom~\ref{item:ax-BO} in \S\ref{sec:axioms}).
\end{rmk}

A consequence of (iv) above is:
\begin{cor}\label{cor:yoneda-iso}
  Two morphisms $f,g \colon X \to Y$ in $\CC$ are isomorphic if and only if there exists an invertible modification $f_* \toi g_*$ between the strict natural transformations $f_*,g_* \colon (\wh{X})_{\gpd} \to (\wh{Y})_\gpd$, i.e., a collection of natural isomorphisms $\phi_A \colon (f_*)_A \to (g_*)_A$ for all groupoids $A \in \CC$ with the naturality property $\phi_A u^* = u^* \phi_{A'} \colon \CC(A,X) \tocellud{f_* u^*}{u^* g_*} \CC(A',Y)$ for all $u \colon A' \to A$.
\end{cor}

We also note the following ``weak'' variant of the above theorem.
However, as we will not be making use of it, we omit the somewhat lengthy proof, which parallels that of Theorem~\ref{thm:gpd-yoneda-rstr}~(iii) but is more involved due to the presence of various coherence isomorphisms.

The 2-category $2\Fun(\CC^\op_\gpd,\Cat)$ is a sub-2-category of the 2-category $2\Fun_\ps(\CC^\op_\gpd,\Cat)$ of 2-functors, \emph{pseudo}-natural transformations, and modifications.
Composing with the inclusion $2\Fun(\CC_\gpd^\op,\Cat)\hto 2\Fun_\ps(\CC_\gpd^\op,\Cat)$, we obtain a Yoneda 2-functor $\CC \to 2\Fun_{\ps}(\CC^\op_\gpd,\Cat)$.

We already know from Theorem~\ref{thm:gpd-yoneda-rstr} that this functor is locally fully faithful.
We now claim that it is also locally essentially surjective:
\begin{propn}
  If $\CC$ is groupoidant, then the 2-functor $\CC \to 2\Fun_{\ps}(\CC^\op_\gpd,\Cat)$ is fully faithful.
  In detail:

  Let $X,Y \in \CC$ and let $F \colon (\wh{X})_\gpd \to (\wh{Y})_{\gpd}$ be a pseudo-natural transformation, i.e., a collection of functors $F_A \colon \CC(A,X) \to \CC(A,Y)$ for each groupoid $A \in \CC$, and natural isomorphisms $F_u \colon F_Au^* \toi u^*F_{A'}$ for each $u \colon A' \to A$ such that
  \[
    \begin{tikzcd}
      u'u(F_Ag)\ar[r, "u'F_u"]
      \ar[rr, bend right=15pt, "F_{u' \cdot u}"']
      &u'(F_{A'}(ug))\ar[r, "F_{u'}"]&F_{A''}(u'ug)
    \end{tikzcd}
  \]
  commutes for all $A''\tox{u'}A'\tox{u}A\tox{g}X$.

  Then there exists a (by Corollary~\ref{cor:yoneda-iso} unique up to isomorphism) morphism $f \colon X \to Y$ and an invertible modification $F \toi f_*$.
\end{propn}
We note that there is an additional condition in the definition of pseudo-natural transformation, analogous to the condition in Remark~\ref{rmk:yoneda-surprise}, that we have omitted in the above proposition as it is not needed in the proof; namely that $F_u \cdot (\alpha^* F_{A'}) = (F_{A} \alpha^*) \cdot F_v \colon \CC(A,X) \tocell \CC(A',Y)$ for each 2-cell $\alpha \colon A' \tocellud{u}{v} A$ in $\CC_{\gpd}$.
Similarly, the condition $F_{\id_A} = \id_{F_A}$ for $A \in \CC_{\gpd}$ is likewise not needed in the proof.

\subsection{Opposites}\label{subsec:opposites}
In Corollary~\ref{cor:duality-involution}, we showed that a groupoidant 2-category $\CC$ admits a duality involution, given by applying the equivalence $\Nv \colon \CC \to \Cong(\CC)$, and then interchanging the source and target morphisms on the resulting internal category.
We now study this operation in more detail.

\begin{defn}
  Let $\CC$ be a pita 2-category.

  Given an internal double category $C$ in $\CC$ given by
  \[
    \begin{tikzcd}
      C_2\ar[r,shift left=10pt, "\pi_{01}"]\ar[r,shift right=10pt, "\pi_{12}"]\ar[r, "\pi_{02}"]&
      C_1\ar[r, shift left=10pt, "\pi_0"]\ar[r, shift right=10pt, "\pi_1"]\ar[from=r, "e"']&
      C_0,
    \end{tikzcd}
  \]
  its \defword{opposite} is the internal double category $C^\op$ given by
  \[
    \begin{tikzcd}
      C_2\ar[r,shift left=10pt, "\pi_{12}"]\ar[r,shift right=10pt, "\pi_{01}"]\ar[r, "\pi_{02}"]&
      C_1\ar[r, shift left=10pt, "\pi_1"]\ar[r, shift right=10pt, "\pi_0"]\ar[from=r, "e"']&
      C_0.
    \end{tikzcd}
  \]

  Next, a \defword{contra-cocone under $C$} is defined to be a cocone under $C^\op$.
  In other words, it is a triple $(X,\gamma,\gamma_{10})$ with $X \in \CC$, $\gamma \colon C_0 \to X$, and $\gamma_{10} \colon \pi_1\gamma \to \pi_0\gamma$ satisfying
  \[
    \begin{split}
      i\gamma_{10}&=\id_\gamma\\
      \pi_{02}\gamma_{10}&=
      (\pi_{01}\gamma_{10})
      (\pi_{12}\gamma_{10}).
    \end{split}
    \qquad
    \begin{tikzcd}[row sep=10pt, column sep=40pt]
      \pi_{12}\pi_1\gamma\ar[dd, equals]\ar[r, "\pi_{12}\gamma_{10}"]&
      \pi_{12}\pi_0\gamma\ar[d, equals]\\
      &\pi_{01}\pi_1\gamma\ar[r, "\pi_{01}\gamma_{10}"]&\pi_{01}\pi_0\gamma\ar[d, equals]\\
      \pi_{02}\pi_1\gamma\ar[rr, "\pi_{02}\gamma_{10}"]&&\pi_{02}\pi_0\gamma
    \end{tikzcd}
  \]
  A \defword{contra-quotient} of $C$ is a quotient of $C^\op$.
\end{defn}

\begin{lem}\label{lem:op-reflection}
  If $C$ is a gbo-congruence in a pita 2-category $\CC$ with opposite $C^\op$, and if $r_C,r_{C^\op} \colon C_0^\to \to C_1$ are the respective reflection morphisms, then $r_{C^\op} = \tau r$ where $\tau \colon C_0^\to \to C_0^\to$ is the inversion morphism (Definition~\ref{defn:rezk-complete}).
\end{lem}

\begin{proof}
  By definition of $r_{C^\op}$, to show that $\tau r = r_{\op}$ we must find a 2-cell $\und \rho_{C^\op} \colon C_0^\to \tocellud{\partial_0 e}{\tau r} C_1$ with $\und \rho_{C^\op}\br{\pi_1,\pi_0} = \br{\id_{\partial_0},\partial} \colon C_0^\to \tocell C_0 \times C_0$.
  A simple calculation shows that $\rho_{C^\op}\defeq \tau\bar \rho_C\I$ has the desired property, where by definition $\bar \rho_C \colon C_0^\to \tocellud{r}{\partial_1 e} C_1$ satisfies $\bar \rho_C \br{\pi_0,\pi_1} = \br{\partial,\id_{\partial_1}}$.
\end{proof}

\begin{defn}
  Given two objects $X,X^\op$ in a 2-category $\CC$, an \defword{opposition between $X$ and $X^\op$} is a complete congruence $C$ together with a quotient cocone $(X,\gamma,\gamma_{01})$ under $C$ and a contra-quotient contra-cocone $(X^\op,\gamma^\op,\gamma_{01}^\op)$ under $C$.
  \begin{equation}\label{eq:opposition}
    \begin{tikzcd}[column sep=40pt]
      &C_2 {
        \ar[d, shift right=10pt, "\pi_{01}"']
        \ar[d, "\pi_{02}" description]
        \ar[d, shift left=10pt, "\pi_{12}"]
      }
      \\
      X
      &C_1 {
        \ar[d, shift right=10pt, "\pi_0"']
        \ar[from=d, "e"']
        \ar[d, shift left=10pt, "\pi_1"]
      }
      \ar[l, shift right=5pt, "\pi_0\gamma"'{name=p0g}, xshift=-8pt, shorten=6pt]
      \ar[l, shift left=5pt, "\pi_1\gamma"{name=p1g}, xshift=-8pt, shorten=6pt]
      \ar[r, shift left=5pt, "\pi_1\gamma^\op"{name=p1go}, xshift=8pt, shorten=6pt]
      \ar[r, shift right=5pt, "\pi_0\gamma^\op"'{name=p0go}, xshift=8pt, shorten=6pt]
      \ar[from=p0g, to=p1g, "\gamma_{01}^{\ }", Rightarrow, shorten=1pt]
      \ar[from=p1go, to=p0go, "\gamma_{01}^\op"', Rightarrow, shorten=1pt]
      &
      X^\op
      \\
      &
      C_0 {
        \ar[lu, "\gamma", bend left]
        \ar[ru, "\gamma^\op"', bend right]
      }&
    \end{tikzcd}
  \end{equation}

  An \defword{opposite object} of $X$ is an object $X^\op$ together with an opposition between $X$ and $X^\op$.
  Note that if $\CC$ is groupoidant, every object has an opposite.
\end{defn}

It is clear that if $C$ gives an opposition between $X$ and $X^\op$, then $C^\op$ gives an opposition between $X^\op$ and $X$, and hence $X$ is also an opposite object of $X^\op$.

If $\CC$ is groupoidant and $X$ is a quotient of $C$, it follows that $C$ is a nerve of $X$, and hence that $X^\op$ is uniquely determined up to isomorphism.

We next given a ``representable'' characterization of opposite objects using Theorem~\ref{thm:gpd-yoneda-rstr}.
In what follows, we will sometimes conflate a functor $F \colon \cC^\op \to \cD$ with the functor $F^\op \colon \cC \to \cD^\op$.

\begin{defn}
  Fix objects $X,X^\op$ in a groupoidant 2-category $\CC$ and an opposition as in (\ref{eq:opposition}).
  Given a groupoid \( A \in \CC \) and a morphism $f \colon A \to X$, we obtain $f^\iso \colon A \to C_0$ (with $f^\iso\gamma = f$), and we set $f^\op\defeq f^\iso\gamma^\op \colon A \to X^\op$.
  Next, given a 2-cell $\alpha \colon A \tocellud{f}{g} X$, we obtain a morphism $a \colon A \to C_1$ with $\alpha = a\cdot \gamma_{01}$, and we set $\alpha^\op\defeq a \cdot \gamma_{01}^\op \colon A \tocellud{g^\op}{f^\op} X^\op$.
\end{defn}

\begin{propn}\label{propn:rep-op-defn}
  If $(C,\gamma,\gamma_{01},\gamma^\op,\gamma_{01}^\op)$ is an opposition between objects $X$ and $X^\op$ in a groupoidant 2-category $\CC$, then for any groupoid \( A \in \CC \), the operation $f \mapsto f^\op$ on morphisms and 2-cells defines an isomorphism of categories $(-)^\op \colon \CC(A,X)^\op \toi \CC(A,X^\op)$, and moreover this determines a strict natural isomorphism between $(\wh{X^\op})_\gpd$ and the 2-functor
  \[
    (\CC_\gpd)^\op\toix{\inv}
    (\CC_\gpd)^\coop\tox{(\wh{X})_\gpd^\co}
    \Cat^\co\tox{\op}
    \Cat,
  \]
  where $\inv$ is the 2-functor acting by the identity on objects and 1-cells, and by inversion on 2-cells.

  By Theorem~\ref{thm:gpd-yoneda-rstr}, the existence of this natural isomorphism determines $X^\op$ uniquely up to isomorphism.
\end{propn}
\begin{proof}
  For the first claim, it suffices to check that $(-)^\op$ is a functor, since it then has an inverse $(-)^\op \colon \CC(A,X^\op) \toi \CC(A,X)^\op$.
  But this is more or less immediate from the fact that $C$ is a nerve of $X$ and the opposite of a nerve of $X^\op$.

  For the second claim, we need only check the strict naturality of $(-)^\op$, since it is then clearly an isomorphism.
  That is, for each morphism $u \colon A' \to A$ and each 2-cell $\alpha \colon A' \tocellud{v}{u} A$ in $\CC_\gpd$, we need to check the commutativity of the square
  \[
    \begin{tikzcd}
      \CC(A,X)^\op\ar[r, "(-)^\op"]\ar[d, "(u^*)^\op"']&\CC(A,X^\op)\ar[d, "u^*"]\\
      \CC(A',X)^\op\ar[r, "(-)^\op"]&\CC(A',X^\op)
    \end{tikzcd}
    \qquad \AND \qquad
    \begin{tikzcd}[]
      \CC(A,X)^\op\ar[r, "(-)^\op"]
      \ar[d, "(u^*)^\op"'{name=ustop}, bend right=40pt]
      \ar[d, "(v^*)^\op"{name=vstop}, bend left=40pt]
      \ar[d, "(\alpha^*)^\op", from=ustop, to=vstop, Rightarrow, shorten=5pt]
      &\CC(A,X^\op)
      \ar[d, "u^*"'{name=ust}, bend right=40pt]
      \ar[d, "v^*"{name=vst}, bend left=40pt]
      \ar[d, "(\alpha\I)^*", from=ust, to=vst, Rightarrow, shorten=5pt]\\[30pt]
      \CC(A',X)^\op\ar[r, "(-)^\op"]&\CC(A',X^\op),
    \end{tikzcd}
  \]
  respectively.

  As for the first square, given $f \colon A \to X$, we have $(uf)^\op = (uf)^\iso\gamma^\op = uf^\iso\gamma^\op = uf^\op$.
  The verification for 2-cells $f \colon A \tocell X$ is similar.

  For the second square, fix $f \colon A \to X$.
  We need to show that $(\alpha f)^\op = \alpha\I f^\op$.

  We have the following morphisms:
  \begin{enumerate}[(i)]
  \item[(i)] $a \colon A' \to A^\to$ defined by the condition $a \partial = \alpha$,
  \item[(ii)] $f^\iso \colon A \to C_0$ with $f^\iso \gamma = f$, and by definition $f^\iso \gamma^\op = f^\op$,
  \item[(iii)] $(f^\iso)^\to \colon A^\to \to C_0^\to$ with $(f^\iso)^\to \partial = \partial f^\iso$,
  \item[(iv)] the reflection morphism $r_C \colon C_0^\to \to C_1$ for $C$.
  \end{enumerate}
  Now, by Lemma~\ref{lem:nv-refl}, we have $r_C \gamma_{01} = \partial \gamma \colon C_0^\to \tocell X$.
  It follows that
  \begin{equation}\label{eq:op-rep-pf-1}
    \pbig{a (f^\iso)^\to r_C} \gamma_{01} =
    a (f^\iso)^\to \partial \gamma =
    a \partial f^\iso \gamma =
    \alpha f^\iso \gamma =
    \alpha f,
  \end{equation}
  and hence, by definition, that $\pbig{a (f^\iso)^\to r_C} \gamma_{01}^\op = (\alpha f)^\op$.

  Now using Lemma~\ref{lem:op-reflection}, we have
  \begin{equation}\label{eq:op-rep-pf-2}
    (\alpha f)^\op =
    \pbig{a (f^\iso)^\to \tau \tau r_C} \gamma_{01}^\op =
    \pbig{a (f^\iso)^\to \tau r_{C^\op}} \gamma_{01}^\op =
    \pbig{a \tau (f^\iso)^\to r_{C^\op}} \gamma_{01}^\op,
  \end{equation}
  where $r_{C^\op} \colon C_0^\to \to C_1$ is the reflection morphism for $C^\op$.

  Computing as in (\ref{eq:op-rep-pf-1}), and using that $a \tau \partial = \alpha\I$, we have that the right hand side of (\ref{eq:op-rep-pf-2}) is equal to $\alpha\I f^\op$, as desired.
\end{proof}

We similarly have a representable characterization of the isomorphism $(-)^\op \colon \CC(X,Y)^\op \to \CC(X^\op,Y^\op)$ for general $X,Y \in \CC$:
\begin{propn}\label{propn:hom-op-representably}
  Given a morphism $f \colon X \to Y$ in a groupoidant 2-category $\CC$, there is a unique morphism $f^\op \colon X^\op \to Y^\op$ such that for all groupoids $A \in \CC$, the functor $(f^\op)_* \colon \CC(A,X^\op) \to \CC(A,Y^\op)$ is equal to the composite
  \[
    \CC(A,X^\op)\tox{(-)^\op}
    \CC(A,X)^\op\tox{(f_*)^\op}
    \CC(A,Y)^\op\tox{(-)^\op}
    \CC(A,Y^\op).
  \]
  Similarly, given a 2-cell $\alpha \colon X \tocellud{f}{g} Y$, there is a unique 2-cell $\alpha^\op \colon X^\op \tocellud{g^\op}{f^\op} Y^\op$ such that, for each groupoid $A$, the natural transformation $(\alpha^\op)_* \colon \CC(A,X^\op) \tocellud{(g^\op)_*}{(f^\op)_*} \CC(A,Y^\op)$ is equal to
  \[
    \CC(A,X^\op)\tox{(-)^\op}
    \twocell{\CC(A,X)^\op}{\ \quad\CC(A,Y)^\op}{(g_*)^\op}{(f_*)^\op}{(\alpha_*)^\op}
    \tox{(-)^\op}
    \CC(A,Y^\op).
  \]

  Moreover, these operations define an isomorphism of categories $(-)^\op \colon \CC(X,Y)^\op \to \CC(X^\op,Y^\op)$.
\end{propn}
\begin{proof}
  The uniqueness follows immediately from Theorem~\ref{thm:gpd-yoneda-rstr}~(i)~and~(ii).
  The existence follows from Theorem~\ref{thm:gpd-yoneda-rstr}~(iii)~and~(iv), since using Proposition~\ref{propn:rep-op-defn}, we see that the given prescriptions define a strict natural transformation $(\wh{X^\op})_\gpd \to (\wh{Y^\op})_\gpd$ and a modification $(\wh{X^\op})_\gpd \tocellud{(g^\op)_*}{(f^\op)_*} (\wh{Y^\op})_\gpd$, respectively.

  For the last claim, it suffices to check that $(-)^\op$ is a functor, since it then has an inverse $(-)^\op \colon \CC(X^\op,Y^\op)^\op \to \CC(X,Y)$ (as the reader may verify).
  The fact that it is a functor follows immediately from the fact that given 2-cells $f\tox{\alpha}g\tox{\beta}h$, we have $((\alpha\beta)_*)^\op = (\beta_*)^\op(\alpha_*)^\op$, and similarly $((\id_f)_*)^\op = \id_{(f_*)^\op}$.
\end{proof}

\section{Examples of 2-topoi}\label{sec:examples}
In this section, we will prove that the 2-category $\Cat(\Set^\cC)$ of internal categories in any presheaf topos $\Set^{\cC}$ satisfies all the axioms of \S\ref{sec:axioms}.
We suspect this is true more generally for $\Cat(\cD)$ for any Grothendieck topos $\cD$, and more generally for any \emph{Grothendieck 2-topos} in the sense of \cite{street-two-sheaf}.

In fact, several of the axioms hold more generally in $\Cat(\cC)$ for any finitely complete 1-category $\cC$.
Hence we begin our discussion in this more general context.
Recall from \S\ref{sec:congruences} the notation $\cC(-,C) \colon \cC^\op \to \abs{\Cat}$ for an internal category $C$ in $\cC$.
\begin{lem}\label{lem:quasi-cotensors}
  Let $C$ be an internal category in a finitely complete category $\cC$.

  For any finite category $J$, the functor
  \[
    \Fun(J,\cC(-,C)) \colon
    \cC^\op \to \Set
  \]
  is representable.

  That is, there exists an object $C_J \in \cC$ and a functor $\ev = \ev_{J,C} \colon J \to \cC(C_J,C)$ such that for any $A \in \cC$ and functor $F \colon J \to \cC(A,C)$, there is a unique morphism $f \colon A \to C_J$ such that $F$ is equal to the composite of
  \[
    J {
      \tox{\ev_{J,C}}
    }
    \cC(C_J,C) {
      \tox{f^*}
    }
    \cC(A,C).
  \]
\end{lem}
\begin{proof}
  As in \S\ref{subsec:nerves}, we write \( [k] \in \Cat \) for the linear order on \( k + 1 \) elements.
  When \( J \) is the category \( [k] \) for \( k = 0,1,2 \), then it is easy to see that the functor \( \Fun(J,\cC(-,C)) \colon \cC^\op \to \Set \) is simply represented by the object \( C_k \).

  Now the functor \( \abs{\Cat}^\op \to \Set^{\cC^\op} \) taking \( J \) to \( \Fun(J,\cC(-,C)) \) preserves limits, as it is the composite of the Yoneda embedding \( \abs{\Cat}^\op \to \Set^{\abs{\Cat}} \) with the functor \( \Set^{\abs{\Cat}} \to \Set^{\cC^\op} \) given by precomposition with \( \cC(-,C) \colon \cC^\op \to \abs{\Cat} \), both of which preserve limits.
  Every finite category \( J \) is the colimit in \( \abs{\Cat} \) of the finite diagram \( \dom \colon \Delta_{\le 2}/J \to \Cat \), where \( \Delta_{\le 2} \subset \Cat \) the full subcategory on the objects $[0],[1],[2]$.
  It follows that \( \Fun(J,\cC(-,C)) \) is a finite limit of representables for every finite \( J \), and is thus representable, since \( \cC \) has finite limits.
\end{proof}

Note that it follows from the universal property that any functor $J \to J'$ induces a morphism $C_{J'} \to C_J$, and moreover this gives an anafunctor $C_{(-)} \colon \abs{\FinCat}^\op \to \cC$.
This anafunctor is (a partially defined) right adjoint to $\cC(-,C) \colon \cC \to \abs{\Cat}^{\op}$ and hence preserves limits (i.e., takes colimits in $\abs{\FinCat}$ to limits in $\cC$).

Now consider the category $\toy$ (the freestanding isomorphism).
We have a monomorphism $i \colon C_{\toy} \to C_1$ which, for each $A$, identifies $\CC(A,C_{\toy})$ with the set of isomorphisms in $\CC(A,C)$ (this is essentially the same as the \emph{object of isomorphisms} of Definition~\ref{defn:ob-of-isos}).

\begin{lem}\label{lem:catc-core-nat-trans}
  Given internal functors $F,G \colon C \to D$ between internal categories $C,D \in \Cat(\cC)$ in a finitely complete category $\cC$, an internal natural transformation $\alpha \colon F \to G$, given by a morphism $\alpha \colon C_0 \to D_1$ is invertible if and only if $\alpha$ factors through $i \colon D_{\toy} \to D_1$.
\end{lem}

\begin{proof}
  An internal natural transformation $\alpha \colon C_0 \to D_1$ from $F$ to $G$ is invertible if and only if the natural transformation $\alpha_* \colon \cC(A,C_0) \to \cC(A,D_1)$ from $F_*$ to $G_*$ is invertible for each $A \in \cC$.

  This is equivalent to $f \alpha$ being an isomorphism in $\cC(A,D)$ for each $f \colon A \to C_0$ in $\cC$, which is in turn equivalent to $\alpha$ factoring through $D_\toy$.
\end{proof}

Now, for a given internal category $D$ in a finitely complete category $\cC$, we may from an internal category $D^\iso$ with objects $D_{\toy\toy},D_\toy,D_0$ by applying the limit-preserving anafunctor $D_{(-)}$ to the corresponding internal category in $\abs{\FinCat}^\op$ with objects $(\toy\toy),(\toy),[0]$.
The latter is equipped with an internal functor to the standard internal category in $\abs{\FinCat}^\op$ with objects $[2],[1],[0]$, and hence we have an internal functor $i \colon D^\iso \to D$ in $\cC$.

\begin{propn}\label{propn:catc-core-exists}
  If $\cC$ is a finitely complete category, then $i \colon D^\iso \to D$ is a core of $D$ for each $D \in \Cat(\cC)$.
\end{propn}

\begin{proof}
  We claim that an internal functor $F \colon C \to D$ in $\cC$ is an arrow-wise isomorphism if and only if $F_1 \colon C_1 \to D_1$ factors through $D_{\toy}$.
  Indeed, if $F_1$ factors through $D_\toy$, then for any internal category $A$ and internal natural transformation $\alpha \colon A_0 \to C_1$, the composite $\alpha F_1$ will factor through $D_\toy$ and hence be a natural isomorphism by Lemma~\ref{lem:catc-core-nat-trans}; and conversely, if $\alpha F_1$ factors through $D_\toy$ for each internal natural transformation $\alpha \colon A \tocell C$, then we can find $A$ with $A_0 = C_1$ such that $\alpha = \id_{C_1} \colon C_1 \to C_1$ is a natural transformation.
  For example, we can take $A$ to be the ``discrete'' internal category on $C_1$, with $A_0 = A_1 = A_2 = C_1$ and all structural morphisms equal to $\id_{C_1}$ (or, once we know that $\Cat(\cC)$ has arrow objects, we can take $A = C^\to$ and $\alpha = \partial$ the universal 2-cell).

  From this it is easy to see that $D^\iso$ satisfies the universal property of the core with respect to morphisms; and the universal property with respect to 2-cells follows from Lemma~\ref{lem:catc-core-nat-trans}.
\end{proof}

Next, we consider the Axioms~\ref{item:ax-N}~and~\ref{item:ax-Q}, and the related Axioms~\ref{item:ax-EX}~and~\ref{item:ax-BO}.

First, let us prove that they are indeed related:
\begin{lem}\label{lem:exact-implies-groupoidant}
  If $\CC$ is corepita and satisfies \ref{item:ax-EX} and \ref{item:ax-BO}, then it also satisfies \ref{item:ax-N} and \ref{item:ax-Q}.
\end{lem}
In the following two proofs, we assume familiarity with the definitions and notation from \cite{bourke-garner-2-reg-ex}.

\begin{proof}
  Assume \ref{item:ax-EX} and \ref{item:ax-BO}.

  The nerve of $X \in \CC$ is the same thing as the $\cF_{\bo}$-kernel of $i \colon X^\iso \to X$ in the sense of \cite[\S5.1]{bourke-garner-2-reg-ex}, and the quotient in our sense of a given congruence is the same as the $\cF_{\bo}$-quotient.
  The gbo-congruences in our sense are precisely the $\cF_{\bo}$-congruences which are level-wise groupoids, by \cite[Proposition~22]{bourke-garner-2-reg-ex}.

  Now, given any $X$, since $i \colon X^\iso \to X$ is by assumption an $\cF_{\bo}$-strong epi map, it follows from \cite[Propositions~3~and~22 and Corollary~20]{bourke-garner-2-reg-ex} that $i$ is effective, which means that the tautological cocone under $\Nv(X)$ is a quotient, so we have \ref{item:ax-Q}.

  As for \ref{item:ax-N}, we have, again by \cite[Proposition~22]{bourke-garner-2-reg-ex}, that any gbo-congruence $C$ is the kernel of its quotient $q \colon C_0 \to X$, so we only need to see that if $C$ is complete, then $C_0$ is a core of $X$ with inclusion morphism $q$.

  But now $q$ factors as $C_0 \tox{\bar q} X^\iso \tox{i} X$, and $q$ is an $\cF_{\bo}$-strong epi by \cite[Corollary~20]{bourke-garner-2-reg-ex}), hence $\bar{q}$ is as well by the second part of Axiom~\ref{item:ax-BO}.
  We now claim that if $C$ is complete, then $\bar{q}$ is also $\cF_{\bo}$-monic (that is to say, fully faithful), and hence an isomorphism, and hence that $q$ is a core inclusion, as desired.

  The completeness of $C$ says, by Proposition~\ref{propn:rezk-equiv-crit}, that the morphism $\bar{e} \colon C_0 \to C_{\cong}$ is an equivalence, where $c \colon C_{\cong} \to C_1$ is the object of isomorphisms, and hence both the projections $c \pi_0,c \pi_1 \colon C_{\cong} \to C_0$ (for which $\bar{e}$ is a section) are equivalences as well.
  Now, by definition of $C$ being an $\cF_{\bo}$-kernel of $q$, we have a lax pullback square as shown below the left.
  It follows from this and the definition of $C_{\cong}$ that the square in the middle is a pseudo-pullback, and hence that the resulting square on the right is also a pseudo-pullback.
  \[
    \begin{tikzcd}
      C_1 {
        \ar[r, "\pi_0"]
        \ar[d, "\pi_1"']
      }&
      C_0 {
        \ar[d, "q"]
        \ar[dl, Rightarrow, "\alpha"']
      }\\
      C_0\ar[r, "q"']&
      X
    \end{tikzcd}
    \qquad\qquad
    \begin{tikzcd}
      C_{\cong} {
        \ar[r, "c \pi_0"]
        \ar[d, "c \pi_1"']
      }&
      C_0 {
        \ar[d, "q"]
        \ar[dl, Rightarrow, "\sim"' sloped, "c \alpha"']
      }\\
      C_0\ar[r, "q"']&
      X
    \end{tikzcd}
    \qquad\qquad
    \begin{tikzcd}
      C_{\cong} {
        \ar[r, "c \pi_0"]
        \ar[d, "c \pi_1"']
      }&
      C_0 {
        \ar[d, "\bar{q}"]
        \ar[dl, Rightarrow, "\sim"' sloped]
      }\\
      C_0\ar[r, "\bar{q}"']&
      X^\iso
    \end{tikzcd}
  \]
  Since $c\pi_0,c\pi_1$ are equivalences, this implies that $\bar{q}$ is \emph{pseudo-monic} (see Remark~\ref{rmk:pseudo-monic}) and hence, since $X^\iso$ is a groupoid, that $\bar{q}$ is fully faithful, as claimed.
\end{proof}

Next, let us show \ref{item:ax-BO}.
\begin{propn}\label{propn:i-is-bo}
  If $\cC$ is a finitely complete category, then $i \colon C^\iso \to C$ is an $\cF_{\bo}$-strong epi for every $C \in \Cat(\cC)$, and for any two morphisms $A \tox{f} B \tox{g} C$ in $\Cat(\cC)$, if $g$ and $fg$ are both $\cF_{\bo}$-strong epis, then so is $f$.
\end{propn}

\begin{proof}
  It is clear from the description of $C^\iso$ (just before Proposition~\ref{propn:catc-core-exists}) that $i_0 \colon (C^\iso)_0 \to C_0$ is an isomorphism.
  Both claims now follow from the fact that a morphism $F \colon A \to B$ in $\Cat(\cC)$ is an $\cF_{\bo}$-strong epi if and only if $F_0 \colon A_0 \to B_0$ is an isomorphism (\cite[Proposition~2.61]{bourke-thesis}).
\end{proof}

We summarize what we have so far:
\begin{propn}\label{propn:internal-cat-axioms}
  If $\cC$ is a finitely complete category, then $\Cat(\cC)$ satisfies Axioms \ref{item:ax-L}, \ref{item:ax-C}, \ref{item:ax-N}, \ref{item:ax-Q}, \ref{item:ax-EX}, and \ref{item:ax-BO} from \S\ref{sec:axioms}.

  If $\cC$ is moreover cartesian closed, then $\Cat(\cC)$ also satisfies Axiom~\ref{item:ax-CC}, i.e., is a cartesian closed 2-category.
\end{propn}

\begin{proof}
  Let $\cC$ be a finitely complete category and let $\CC = \Cat(\cC)$.

  It is well known that $\CC$ has all strict finite limits (and in particular all pita limits); for example, it is asserted in \cite[(4.4)]{street-cosmoi-of-internal-cats}.
  In fact, pullbacks (and the terminal object) in $\CC$ are constructed object-wise, and arrow objects can be constructed using Lemma~\ref{lem:quasi-cotensors} by a similar construction to the one in Proposition~\ref{propn:catc-core-exists}.

  That $\CC$ is cartesian closed when $\cC$ is, is proven in \cite[(4.5)]{street-cosmoi-of-internal-cats}.

  Proposition~\ref{propn:catc-core-exists} shows that $\CC$ has cores.

  It is proven in \cite[Proposition~60]{bourke-garner-2-reg-ex} that $\CC$ is $\cF_\bo$-exact (i.e., satisfies \ref{item:ax-EX}); we have that $\CC$ satisfies \ref{item:ax-BO} by Proposition~\ref{propn:i-is-bo}; and hence by Lemma~\ref{lem:exact-implies-groupoidant} that $\CC$ also satisfies \ref{item:ax-N}~and~\ref{item:ax-Q}.
\end{proof}

\subsection{Colimits and cocores}
For the verification of the remaining axioms, we restrict our attention to the 2-category $\Cat(\Set^{\cC})$ of internal categories in the category of set-valued functors $\Set^{\cC}$.
This is strictly equivalent to the 2-category $\Cat^\cC$ of strict category-valued functors $\cC \to \Cat$, strict natural transformations, and modifications.

\begin{propn}\label{propn:presheaf-bicomplete}
  The 2-category \( \Cat^{\cC} \) has all small bilimits and bicolimits, and in fact all strict limits and colimits, and has cocores.
\end{propn}

\begin{proof}
  Strict limits and colimits are identical with limits and colimits in the sense of enriched category theory.
  Thus, the first claim follows from the general result \cite[Proposition~3.75]{kelly-basic-concepts}.

  The cocores in \( \Cat^{\cC} \) are given pointwise; that is, for \( X \in \Cat^{\cC} \), define \( \ol X \) as the composite of \( X \) with the 2-functor \( \Cat \to \Cat \) taking each category to its cocore.
  The canonical morphisms \( X a \to \ol X a \) for \( a \in \cC \) assemble into a natural transformation \( X \to \ol X \), and this makes \( \ol X \) into a cocore of \( X \).
\end{proof}

\begin{rmk}\label{rmk:probably-bc}
  By Proposition~\ref{propn:internal-cat-axioms}, \( \Cat(\cC) \) has \emph{finite} limits under much weaker assumptions on \( \cC \) than the assumption in Proposition~\ref{propn:presheaf-bicomplete} that \( \cC \) is a presheaf category.

  I was informed by the anonymous referee that it has recently been proven that \( \Cat(\cC) \) also has finite \emph{colimits} under much weaker assumptions (for instance, that \( \cC \) is an elementary topos with natural numbers objects) \cite{hughes-miranda-colimits}, and that part of this had already been proven in \cite[Corollary~6.1]{johnstone-wraith-alg-theories-in-toposes}.
\end{rmk}

\subsection{DOF classifiers}\label{subsec:examples-dof-classifiers}
We can obtain DOF classifiers in $\Cat^\cC$ using \cite[Example~4.7]{weber-2-toposes}.
Here, we consider the 2-functor $\Sp \colon \Cat \to \Cat^\cC$ which is a right 2-adjoint to the ``domain-category-of-the-Grothendieck-construction'' 2-functor $\el \colon \Cat^\cC \to \Cat$.
Explicitly, for $\cD \in \Cat$, $\Sp(\cD) \colon \cC \to \Cat_{\cU_2}$ is the functor $\Fun((-)/\cC,\cD)$ taking $c \in \cC$ to the category of functors $\Fun(c/\cC,\cD)$.

In \loccit, it is shown that for any generic DOF $p \colon E \to B$ in $\Cat$, the morphism $\Sp(p) \colon \Sp(E) \to \Sp(B)$ is again a generic DOF (in $\Cat^{\cC}$).

We will now show that $\Sp(p)$ is plentiful for appropriate $p$.
We begin with several lemmas concerning DOFs and generic DOFs in $\Cat^\cC$.

\begin{lem}\label{lem:dofs-in-fun-c-cat}
  Given a category $\cC$ and a morphism $f \colon X \to Y$ in $\CC = \Cat^\cC$, the following are equivalent:
  \begin{enumerate}[(i)]
  \item\label{item:dofs-in-fun-c-cat-dof} $f$ is a DOF in $\Cat^\cC$.
  \item\label{item:dofs-in-fun-c-cat-ptwise} The functor $f_A \colon X(A) \to Y(A)$ is a DOF for each $A \in \cC$.
  \item\label{item:dofs-in-fun-c-cat-el} The functor $\el(f) \colon \el(X) \to \el(Y)$ is a DOF.
  \end{enumerate}
\end{lem}

\begin{proof}
  We first prove \ref{item:dofs-in-fun-c-cat-dof}~$\To$~\ref{item:dofs-in-fun-c-cat-ptwise}.
  For $A \in \cC$, consider the representable functor $\wh{A} = \cC(A,-) \in \CC$ (note that this functor $\wh{A} \colon \cC \to \Cat$ takes values in discrete categories).
  By the 2-categorical Yoneda lemma, we have $\CC(\wh{A},X) \cong X(A)$ (isomorphism of categories) for any $X$.

  Hence, if $p \colon X \to Y$ is a DOF, then by definition, so is $p_* \colon \CC(\wh{A},X) \to \CC(\wh{A},Y)$, and hence (by the naturality of the Yoneda isomorphism) also $p_A \colon X(A) \to Y(A)$.

  Conversely, suppose that each $p_A$ is a DOF.
  We want to show that $p_* \colon \CC(U,X) \to \CC(U,Y)$ is a DOF for each $U \in \CC$.
  Given morphisms $f,g \colon U \to Y$, a 2-cell $\alpha \colon f \to g$, and a lift $\tilde{f} \colon U \to X$ of $f$, we obtain a unique lift $\tilde{\alpha}_A \colon \tilde{f}_A \to g_A$ of $\alpha_A \colon f_A \to g_A$.
  \[
    \begin{tikzcd}
      &[20pt]X\ar[d]\\
      U{
        \ar[r, bend left=20pt, "f"{name=f}]
        \ar[r, bend right=20pt, "g"'{name=g}]
        \ar[ru, bend left, "\tilde{f}"]
      }
      &Y
      \ar[from=f, to=g, Rightarrow, shorten=2pt, "\alpha"]
    \end{tikzcd}
    \qquad
    \begin{tikzcd}
      &[30pt]X(A)\ar[d]\\
      U(A){
        \ar[r, bend left=10pt, "f_A"{name=f}]
        \ar[r, bend right=15pt, "g_A"'{name=g}]
        \ar[ru, bend left, "\tilde{f}_A"{name=tf}]
        \ar[ru, bend left=5pt, ""{name=tg},"\tilde{g}_A"'{pos=0.6}, dashed]
      }
      &Y(A)
      \ar[from=f, to=g, Rightarrow, shorten=2pt, "\alpha_A"]
      \ar[from=tf, to=tg, Rightarrow, shorten <=2pt, shorten >=-4pt, "\tilde{\alpha}_A"{pos=0.9}, dashed]
    \end{tikzcd}
  \]
  One verifies that the $\tilde{g}_A$ constitute a natural transformation $\tilde{g} \colon U \to X$ and the $\tilde{\alpha}_A$ a modification $\tilde{\alpha} \colon \tilde{f} \to \tilde{g}$, and the uniqueness of the $\alpha_A$ implies that of $\tilde{\alpha}$.

  The equivalence of \ref{item:dofs-in-fun-c-cat-ptwise} and \ref{item:dofs-in-fun-c-cat-el} is a special case of the following general fact: given a morphism of opfibrations
  \begin{equation}\label{eq:dofs-in-fun-c-cat-opfib-mor}
    \begin{tikzcd}[column sep=0pt]
      \cE\ar[dr, "E"']\ar[rr, "F"]&&\cE'\ar[dl, "E'"]\\
      &\cC
    \end{tikzcd}
  \end{equation}
  (i.e., the triangle commutes, $E$ and $E'$ are opfibrations and $F$ preserves cocartesian morphisms), $F$ is a DOF if and only if the fibrewise restrictions $\rstr{F}{E\I(C)} \colon E\I(C) \to (E')\I(C)$ of $F$ are DOFs for all $C \in \cC$.
\end{proof}

\begin{lem}\label{lem:generic-dofs-in-cat}
  Given a DOF $p \colon \s_* \to \s$ in $\Cat$:
  \begin{enumerate}[(i)]
  \item\label{item:generic-dofs-in-cat-gen} $p$ is generic if and only if, for all $s,s' \in \s$, the operation $f \mapsto f_*$ gives a bijection from $\s(s,s')$ to the set of functions $p\I(s) \to p\I(s')$.
  \item\label{item:generic-dofs-in-cat-small} If $p$ is generic (and assuming the axiom of choice), a DOF $f \colon X \to Y$ in $\Cat$ is $p$-small if and only if each fibre of $f$ is isomorphic to some fibre of $p$.
  \end{enumerate}
\end{lem}

\begin{proof}
  For the $(\To)$ direction of \ref{item:generic-dofs-in-cat-gen}, assuming $p$ is generic, for any $s,s'$, consider the functors $f,f' \colon [0] \to \s$ from the trivial category $[0]$ taking the unique object of $[0]$ to $s$ and $s'$ respectively, and form the pullbacks $f^*p,(f')^*p$.
  Our desired conclusion about $p$ with respect to $s$ and $s'$ then amounts precisely to the genericity property of $p$ with respect to these two pullback squares.

  We next consider the $(\oT)$ direction of \ref{item:generic-dofs-in-cat-gen}.

  Assume the hypothesis on $p$.
  Referring to the notation from Definition~\ref{defn:generic-dof}, given pullback squares as in (\ref{eq:pb-squares}), a morphism $g \colon F \to H$ and a 2-cell $\ddot{\gamma} \colon \ddot{f} \to g \ddot{h}$, we need to show that there is a unique 2-cell $\gamma$ as in (\ref{eq:generic-diags}) with $\ddot{\gamma} p = \dot{f} \gamma$.

  Given $a \in A$, we are forced to define $\gamma_a \colon f(a) \to h(a)$ to be the unique morphism such that $(\gamma_a)_* \colon p\I\pbig{f(a)} \to p\I\pbig{h(a)}$ is the function taking $\hat{s} \in p\I\pbig{f(a)}$ to the codomain of $\ddot{\gamma}_{\hat{a}}$, where $\hat{a} \in F$ is the unique object with $\ddot{f}(\hat{a}) = \hat{s}$ and $\dot{f}(\hat{a}) = a$.
  It remains to see that $\gamma_a$, thus defined, is natural.

  Fix $u \colon a \to a'$ in $A$.
  To see that
  \begin{equation}\label{eq:generic-dofs-in-cat-squre}
    \begin{tikzcd}
      f(a)\ar[r, "f(u)"]\ar[d, "\gamma_a"']&f(a')\ar[d, "\gamma_{a'}"]\\
      h(a)\ar[r, "h(u)"]&h(a')
    \end{tikzcd}
  \end{equation}
  commutes, it suffices to show that $\pbig{f(u) \cdot \gamma_{a'}}_* = \pbig{\gamma_a \cdot h(u)}_* \colon p\I\pbig{f(a)} \to p\I\pbig{h(a')}$.
  This follows from the commutativity, for each $\hat{s} \in p\I\pbig{f(a)}$ of the naturality square
  \[
    \begin{tikzcd}
      \hat{s}\ar[r, equals]&[-25pt]
      \ddot{f}(\hat{a}) \ar[r, "\ddot{f}(\hat{u})"] \ar[d, "\ddot{\gamma}_{\hat{a}}"']&
      \ddot{f}(\hat{a'}) \ar[d, "\ddot{\gamma}_{\hat{a}'}"]\\
      &\ddot{h}\pbig{g(\hat{a})} \ar[r, "\ddot{h}\pbig{g(\hat{u})}"]&
      \ddot{h}\pbig{g(\hat{a'})}
    \end{tikzcd}
  \]
  for $\ddot{\gamma}$ at $\hat{u} \colon \hat{a} \to \hat{a'}$, which lies over (\ref{eq:generic-dofs-in-cat-squre}).
  Here, $\hat{a} \in F$ is the unique object with $\ddot{f}(\hat{a}) = \hat{s}$ and $\dot{f}(\hat{a}) = a$, and $\hat{u}$ is the unique lift of $u$ with domain $\hat{a}$.

  As for \ref{item:generic-dofs-in-cat-small}, one implication is obvious, and the other implication is straightforward based on what we just proved.
\end{proof}

\begin{lem}\label{lem:el-preserves-reflects-pbs}
  The 2-functor $\el \colon \Cat^\cC \to \Cat$ preserves and reflects strict pullback squares.
\end{lem}

\begin{proof}
  $\el$ factors as $\Cat^\cC \tox{\Gr} \OFs(\cC) \tox{\dom} \Cat$, where $\OFs$ is the 2-category of split opfibrations over $\cC$, and the Grothendieck-construction 2-functor $\Gr$ is a strict equivalence and hence preserves and reflects all finite limits.

  In more detail, an object $(\cE,p)$ of $\OFs$ is a functor $\cE \tox{P} \cC$ together with a choice, for each $f \colon c \to c'$ in $\cC$ and $e \in \cE$ with $P(e) = c$, of a cocartesian lift $\tilde{f}_e$ of $f$ with $\dom(\tilde{f}_e)=e$, satisfying $(\wt{\id_c})_e = \id_e$ and $(\wt{fg})_e = \tilde{f}_e \cdot (\tilde{g})_{\cod(\tilde{f}_e)}$.
  A morphism $F \colon (\cE,P) \to (\cE',P')$ is a functor $F \colon \cE \to \cE'$ preserving the specified cocartesian lifts, and a 2-cell $\alpha \colon (\cE,P) \tocellud{F}{G} (\cE',P')$ is a natural transformation $\alpha \colon F \to G$ with $\alpha P' = \id_{P}$.

  It remains to see that $\OFs(\cC) \tox{\dom} \Cat$ preserves and reflects strict pullback squares.
  The argument is similar to that given in the proof of Proposition~\ref{propn:gr-preserves-reflects-limits}: one shows that to form a strict pullback in $\OFs(\cC)$, we may first form the pullback of the domains in $\Cat$, and there is then a canonical way to supply the result with the structure of a split opfibration, making the resulting square a strict pullback in $\OFs(\cC)$.
\end{proof}

\begin{lem}\label{lem:p-small-in-fun-c-cat}
  Given a category $\cC$, a generic DOF $p \colon E \to B$ in $\Cat$ and a DOF $f \colon X \to Y$ in $\Cat^\cC$, the following are equivalent:
  \begin{enumerate}[(i)]
  \item\label{item:p-small-in-fun-c-cat-f} $f$ is $\Sp(p)$-small.
  \item\label{item:p-small-in-fun-c-cat-f-a} The DOF $f_A$ in $\Cat$ is $p$-small for all $A \in \cC$.
  \item\label{item:p-small-in-fun-c-cat-el-f} The DOF $\el(f)$ in $\Cat$ is $p$-small.
  \end{enumerate}
\end{lem}

\begin{proof}
  We first prove \ref{item:p-small-in-fun-c-cat-f}~$\ToT$~\ref{item:p-small-in-fun-c-cat-el-f}.
  By the adjunction $\el \dashv \Sp$, commutative squares
  \begin{equation}\label{eq:p-small-in-fun-c-cat-square-1}
    \begin{tikzcd}
      \el(X)\ar[r, ""]\ar[d, "\el(f)"']\pb&E\ar[d, "p"]\\
      \el(Y)\ar[r, ""]&B
    \end{tikzcd}
\end{equation}
  in $\Cat$ are in bijection with commutative squares
  \begin{equation}\label{eq:p-small-in-fun-c-cat-square-2}
    \begin{tikzcd}
      X\ar[r, ""]\ar[d, "f"']&\Sp(E)\ar[d, "\Sp(p)"]\\
      Y\ar[r, ""]&\Sp(B)
    \end{tikzcd}
\end{equation}
  in $\Cat^{\cC}$.

  Now, it is proven in \cite[Example~4.7]{weber-2-toposes} that the naturality square with respect to $p$
  \[
    \begin{tikzcd}
      \el\pbig{\Sp(E)}\ar[r, ""]\ar[d, "\el\pbig{\Sp(p)}"']&E\ar[d, "p"]\\
      \el\pbig{\Sp(B)}\ar[r, ""]&B
    \end{tikzcd}
  \]
  for the counit of the adjunction $\el \dashv \Sp$ is a pullback square.

  By Lemma~\ref{lem:el-preserves-reflects-pbs}, it follows (using the 2-of-3 property of pullback squares) that (\ref{eq:p-small-in-fun-c-cat-square-1}) is a strict pullback if and only if (\ref{eq:p-small-in-fun-c-cat-square-2}) is, which proves the claim.

  The equivalence \ref{item:p-small-in-fun-c-cat-f-a}~$\ToT$~\ref{item:p-small-in-fun-c-cat-el-f} is a special case of the more general claim that given a morphism of opfibrations $F$ as in (\ref{eq:dofs-in-fun-c-cat-opfib-mor}) such that $F$ is a DOF, $F$ is $p$-small if and only if $\rstr{F}{E\I(C)} \colon E\I(C) \to (E')\I(C)$ is $p$-small for all $C \in \cC$.
  This, in turn, is immediate from Lemma~\ref{lem:p-small-in-fun-c-cat}~\ref{item:generic-dofs-in-cat-small}.
\end{proof}

\label{universes}
Before proceeding, let us clarify our conventions concerning universes, about which we will have to be careful in what follows.
A \emph{Grothendieck universe} (or just \emph{universe}) is a non-empty set $\cU$ which is (i) transitive (i.e., $a \in b \in \cU$  implies $a \in \cU$) and which is closed under (ii) power sets, and (iii) unions indexed by elements of $\cU$ \cite[Exposé~1]{sga-4-1}\cite[\S1.1]{jonshon-yau-2-categories}.

For a universe $\cU$, we write $\Set_\cU$ for the category of sets that are elements of $\cU$ (and $\Set_{\cU*}$ for the category of pointed sets) and $\Cat_\cU$ for the 2-category of categories that are elements of $\cU$.

\begin{propn}\label{propn:sp-set-is-plentiful}
  Let $\cU' \in \cU$ be two universes, and let $p=p_{\cU'} \colon \Set_{\cU'*} \to \Set_{\cU'}$ be the forgetful functor (this is a generic DOF in $\Cat_{\cU}$ by Lemma~\ref{lem:generic-dofs-in-cat}).

  Then the generic DOF $\Sp(p_\cU) \colon \Sp(\Set_{\cU'*}) \to \Sp(\Set_{\cU'})$ in $\Cat_{\cU}^\cC$ is plentiful.
\end{propn}
\begin{proof}
  First observe that $p$ itself is plentiful.
  More precisely, referring to the defining properties \ref{item:plentiful-ax-first}-\ref{item:plentiful-ax-last} of a plentiful DOF from \S\ref{sec:axioms}, and appealing to Lemma~\ref{lem:generic-dofs-in-cat}~\ref{item:generic-dofs-in-cat-small}, we have that $p$ satisfies \ref{item:plentiful-ax-monos} as soon as $\cU'$ contains all subsets of some one-element set, $p$ satisfies \ref{item:plentiful-ax-composites} as soon as $\cU'$ is closed under disjoint unions indexed by sets that are elements of $\cU$, and $p$ satisfies \ref{item:plentiful-ax-pow} as soon as $\cU'$ is additionally closed under power sets (this follows upon verifying that, in $\Cat$, the DOF collapse of $\partial_1^\iso \colon \Mon(\Set_\cU)^\iso \to \Set_\cU^\iso$ is given by the forgetful functor $\wt \pow \to (\Set_\cU)^\iso$ from the groupoid $\wt \pow \subset (\Set_\cU^\to)^\iso$ of subset inclusions).

  It follows, using Lemma~\ref{lem:p-small-in-fun-c-cat}, and the fact that (by Lemma~\ref{lem:el-preserves-reflects-pbs}) $\el \colon \Cat^\cC \to \Cat$ preserves monomorphisms, that $\Sp(p)$ is also pre-plentiful, i.e., satisfies \ref{item:plentiful-ax-first}-\ref{item:plentiful-ax-pre-last}.

  To see that $\Sp(p)$ satisfies \ref{item:plentiful-ax-pow}, we first make some observations about the 2-category $\Cat^\cC$, which are straightforward to verify (making use of the representable functors $\wh{A} \colon \cC \to \Cat$ as in the proof of Lemma~\ref{lem:dofs-in-fun-c-cat}):
  \begin{enumerate}[(i)]
  \item The core of an object $F \in \Cat_\cU^\cC$ is the functor $F^\iso$ which is the composite of $F$ with the core functor $(-)^\iso \colon \abs{\Cat_\cU} \to \abs{\Cat_\cU}$.
    (Note: the latter does \emph{not} naturally extend to a 2-functor $\Cat_\cU \to \Cat_\cU$.)
  \item If $F \in \Cat_\cU^\cC$ has finite limits, then $\Mon(F)$ is given by the composite of $F$ with the partial functor $\abs{\Cat_\cU} \dto \abs{\Cat_\cU}$ taking each finite limit category to its category of monomorphisms, and taking each functor to the induced functor on monomorphisms.
  \end{enumerate}
  Now recall that $\Sp(\Set_{\cU'})(A) = \Fun(A/\cC,\Set_{\cU'})$ and $\Sp(\Set_{\cU'})(f)$ for $f \colon A \to B$ is the functor $(f^*)^* \colon \Fun(A/\cC,\Set_{\cU'}) \to \Fun(B/\cC,\Set_{\cU'})$, where $f^* = (-) \cdot f \colon B/\cC \to A/\cC$.

  Now, for each $A \in \cC$, let $\Sub(A)$ be the subcategory of $\Fun(A/\cC,\Set_{\cU'})^\to$ consisting of inclusions of subfunctors and isomorphisms between these.
  It follows from the above observations that this defines a subfunctor $\Sub \colon \cC \to \Cat_\cU$ of $\Mon(\Sp(\Set_{\cU'}))^\iso$, and we have a commutative triangle
  \[
    \begin{tikzcd}[column sep=0pt]
      \Sub\ar[dr, "\partial_2"']\ar[rr, hookrightarrow]&&\Mon(\Sp(\Set_{\cU'}))^\iso\ar[dl, "\partial_2^\iso"]\\
      &\Sp(\Set_{\cU'})^\iso,
    \end{tikzcd}
  \]
  where $\partial_2 \colon \Sub \to \Sp(\Set_{\cU'})^\iso$ a DOF by Lemma~\ref{lem:dofs-in-fun-c-cat}.

  Now, the inclusion $i \colon \Sub \to \Mon(\Sp(\Set_{\cU'}))^\iso$ is such that $i_A$ is an equivalence for each $A$.
  Since $(\partial_2)_A$ is a DOF for each $A$, it follows from Proposition~\ref{propn:equiv-of-mors-props}~\ref{item:equiv-of-more-props-dof-good} that $i$ is itself an equivalence, and hence that $\partial_2$ is a DOF collapse of $\partial_2^\iso$.

  Hence, to check that $\Sp(p)$ satisfies \ref{item:plentiful-ax-pow}, it suffices to check that $\partial_2$ is $p$-small, i.e., that $(\partial_2)_A \colon \Sub(A) \to \Fun(A/\cC,\Set_{\cU'})$ has $\cU'$-small fibres (or rather, fibres isomorphic to $\cU'$-small sets) for each $A \in \cC$.
  But the fibre over a given $F \colon A/\cC \to \Set_{\cU'}$ is the set of subfunctors of $F$, and this is isomorphic to a subset of $\bigsqcup_{B \in \Ob(A/\cC)}\pow(F(B)) \in {\cU'}$.
\end{proof}

\begin{cor}\label{cor:except-UA}
  If $\cU$ is a universe that contains another universe $\cU' \in \cU$, then the 2-category $\Cat_{\cU}$ satisfies all the axioms of \S\ref{sec:axioms}, except possibly Axiom~\ref{item:ax-UA}.
\end{cor}

We note that the existence of a universe containing another universe is a stronger assumption than the mere existence of universe, though it is certainly implied by the commonly assumed \emph{Axiom of Universes}, which says that every set is contained in some universe.

For the purposes of finding a 2-category satisfying Axiom~\ref{item:ax-UA} from \S\ref{sec:axioms}, it would seem that we need a universe $\cU$ such that each $x \in \cU$ is contained in some further universe $x \in \cU' \in \cU$, the existence of which goes beyond even Axiom of Universes.
However, we can get away with slightly less:

Let us say that a set $\cU$ is a \emph{pre-universe} if it satisfies (i),(ii) in the definition of universe above on p.~\pageref{universes}, and in addition satisfies (iii$'$) $\cU$ is closed under \emph{finite} unions.
It is then clear (say, in ZF), that every set is contained in a pre-universe, and that, assuming the Axiom of Universes, there exists a pre-universe $\cU$ such that every $x \in \cU$ is contained in a universe $x \in \cU' \in \cU$.

For a pre-universe $\cU$, we use the notation $\Set_\cU$ and $\Cat_\cU$ as above.

It is not hard to see that the above results, and in particular Proposition~\ref{propn:sp-set-is-plentiful}, still holds if $\cU$ is only assumed to be a pre-universe.
The one subtlety is that the 2-functor $\el$ with domain $\Cat_{\cU}^\cC$ can no longer be considered to have codomain $\Cat_{\cU}$, and we must instead use $\Cat_{\mathcal{V}}$ for some larger pre-universe $\mathcal{V}$.

We conclude:

\begin{cor}\label{cor:also-UA}
  If $\cU$ is a pre-universe such that any $x \in \cU$ is contained in some universe $x \in \cU' \in \cU$, then the 2-category $\Cat_{\cU}$ satisfies all the axioms of \S\ref{sec:axioms}.
\end{cor}

Similarly, in Corollary~\ref{cor:except-UA}, it is enough to assume that \( \cU \) is a pre-universe.

\appendix

\section{Finite limit sketches}\label{sec:pbt-skectches}
In this appendix, we prove that, for an object $X$ with finite limits in a pita 2-category, we can form cotensors of $X$ not only with finite categories, but also finite \emph{finite-limit sketches}.
This will be used in Appendix~\ref{sec:groupoid-stuff}, but is also of obvious independent interest, as it says that we can form the category of models in $X$ of any (finitary) essentially algebraic theory.

A \defword{pbt (``pullbacks and terminal objects'') sketch} $J = (\abs{J},P,T)$ is a category $\abs{J}$ together with a set $P$ of commutative squares in $\abs{J}$ and a set $T$ of objects in $\abs{J}$; it is a \defword{finite} pbt sketch if $\abs{J}$ (hence also $P$ and $T$) is finite.
A \defword{model} $J \to \cD$ of $J$ in a category $\cD$ is a functor $\abs{J} \to \cD$ taking each square in $P$ to a pullback square in $\cD$ and each object in $T$ to a terminal object in $\cD$.
(Pbt sketches are closely related to \emph{finite limit sketches}, and are (nearly) the same as the ``flSk sketches'' of \cite{makkai-generalized-sketches-2}.
These notions are all practically equivalent to one another, and there is considerable flexibility in the definition; for example, dropping the requirement that the squares in \( P \) be commutative would again result in an essentially equivalent notion.)

We will also need the following refinement of the notion of model: if $\CC$ is a 2-category and $A,B \in \CC$, then a \defword{stable model} $J \to \CC(A,B)$ is a model for which moreover each specified square or object is taken to a \emph{stable} pullback square or terminal object in $\CC(A,B)$ in the sense of Remark~\ref{rmk:pointwise-limits}.

Given an object $X \in \CC$ and a pbt sketch $J$, a \defword{strict cotensor of $X$ with $J$} is an object $X^{J}$ together with a stable model $\ev=\ev_{J,X} \colon J \to \CC(X^J,X)$ with the following universal property: for each $A \in \CC$, the functor $\CC(A,X^J) \to \Fun(\abs{J},\CC(A,X))$ defined as in (\ref{eq:cotensor-morphism}) on p.~\pageref{eq:cotensor-morphism} is an isomorphism onto the subcategory consisting of stable models $J \to \CC(A,X)$.

In \S\ref{subsec:limits}, we mentioned the abuse of notation that, given a category $K$ and a cotensor $X^K$, we may sometimes conflate a diagram $K \to \CC(A,X)$ with the corresponding morphism $A \to X^K$.
Similarly, for a pbt sketch $J$, we may conflate a morphism $A \to X^J$ with the corresponding stable model $J \to \CC(A,X)$.

A \defword{morphism of pbt sketches} is a functor of underlying categories taking each specified square or object to a specified square or object; a morphism of sketches $J \to J'$ can be composed with a (stable) model of $J'$ to produce a (stable) model of $J$.
Given a morphism $F \colon J \to J'$ of pbt-sketches and cotensors $X^J$ and $X^{J'}$, there is an induced morphism $X^F \colon X^{J'} \to X^J$ classifying the stable model $J \tox{F} J' \tox{\ev} \CC(X^{J'},X)$.

Note that if there exists a cotensor $X^{\abs J}$ of $X$ with the underlying category $\abs{J}$ of $J$, then there is a morphism (in fact, an isofibration)
\begin{equation}\label{eq:sketch-cotensor-forget}
  X^{J} \tox{\phi_{X,J}} X^{\abs J}
\end{equation}
classifying (the underlying functor of) $\ev_{J,X}$.

We will prove:
\begin{propn}\label{propn:sketch-cotensors-exist}
  If $X \in \CC$ has finite limits and $\CC$ is pita, then there exists a cotensor $X^J$ for any finite pbt sketch $J$.
\end{propn}
See \cite{street-sketches-in-bicats-arxiv} for a related result.

The proof will parallel the construction of cotensors in a category with PIE limits, which is (a special case of) what is done in \cite[\S2]{power-robinson-pie-limits}.
For the proof of the proposition, we will also need to use the simplest special cases of it, which we now prove separately:
\begin{lem}
  Suppose $\CC$ is pita and $X \in \CC$ has finite limits.
  If $J$ is the pbt sketch $\pbsquare$ consisting of a single pullback square or the sketch $\tm_!$ consisting of a single terminal object, then the cotensor $X^J$ exists.
\end{lem}
\begin{proof}
  We do the case of a pullback square; the other case is similar but easier.
  Note that in this case, a stable model of $J = \pbsquare$ in $\CC(A,B)$ is just a stable pullback square in $\CC(A,B)$.

  To begin with, we form a cotensor $X^{\lrcorner}$ of $X$ with a freestanding cospan $\lrcorner$.
  We thus have a universal cospan $\ev_{\lrcorner,X}$ in $\CC(X^\lrcorner,X)$, and we next form a pullback of $\ev_{\lrcorner,X}$ to obtain a (stable) pullback square $p \colon X^\lrcorner \to X^\square$ in $\CC(X^\lrcorner,X)$, where $X^\square$ is the cotensor of $X$ with a freestanding commutative square $\square$.

  Now consider the finite category $\toy * \lrcorner$, the join (see \S\ref{subsec:gpd-stuff-finite-lims}) of a freestanding isomorphism $\toy$ with a cospan.
  The two inclusions $j_0,j_1 \colon \tm\hto \toy$, where $\tm$ is the terminal category, induce morphisms $i_k \colon \square \toi \tm*\lrcorner\htox{j_k}(\toy * \lrcorner)$ (for $k = 0,1$).

  We now form $X^\pbsquare$ as a strict pullback
  \[
    \begin{tikzcd}
      X^\pbsquare\pb\ar[r, "\pi_1"]\ar[d, "\pi_0"']&X^{\toy * \lrcorner}\ar[d, "X^{i_1}"]\ar[r, "X^{i_0}"]&X^{\square}\\
      X^\lrcorner\ar[r, "p"]&X^\square
    \end{tikzcd}
  \]
  (which exists as $X^{i_1}$ is an isofibration).
  The desired universal stable model of $\pbsquare$, we claim, is the stable pullback square $\pi_1 X^{i_0}$ in $\CC(X^\pbsquare,X)$.
  To see that this is indeed a stable pullback square, note that it suffices to show that $\pi_1 X^{i_1}$ is a stable pullback square, since if two cones over a given span are isomorphic, then one is a (stable) pullback square if and only if the other is; but $\pi_1 X^{i_1} = \pi_0 p$, and $p$ is a stable pullback square by definition.

  To see that $X^\pbsquare$ has the desired universal property, it suffices, using the universal property of $X^\square$, to show that for each $A \in \CC$, the functor $(\pi_1X^{i_0})_* \colon \CC(A,X^\pbsquare) \to \CC(A,X^\square)$ is an isomorphism onto the full subcategory on the objects of $\CC(A,X^\square)$ classifying pullback squares in $\CC(A,X)$ (which are automatically stable, since $X$ is assumed to have finite limits).

  Full-faithfulness of $(\pi_1 X^{i_0})_*$ follows from the fact that $X^{i_0}$ and $p$ are fully faithful and the fact that fully faithful morphisms are stable under strict pullbacks and composition.

  To prove that $(\pi_1X^{i_0})_*$ is injective on objects and has the correct image, fix a pullback square $q \colon A \to X^{\square}$ in $\CC(A,X)$, with underlying cospan $r \colon A \to X^\lrcorner$.
  \[
    \begin{tikzcd}
      A
      \ar[rd, dashed, "t"]
      \ar[rrd, dashed, "s", bend left=5pt]
      \ar[rrrd, "q", bend left=13pt]
      \ar[rdd, "r"', bend right=10pt]
      &\\
      &X^\pbsquare\pb\ar[r, "\pi_1"]\ar[d, "\pi_0"]&X^{\toy * \lrcorner}\ar[d, "X^{i_1}"]\ar[r, "X^{i_0}"']&X^{\square}\\
      &X^\lrcorner\ar[r, "p"]&X^\square
    \end{tikzcd}
  \]
  Thus, $r p$ and $q$ are both pullbacks of the same cospan $r$, and hence assemble into a unique diagram $s \colon A \to X^{\toy * \lrcorner}$ in $\CC(A,X)$ satisfying $s X^{i_0} = q$ and $s X^{i_1} = rp$.
  The universal property of the pullback square then gives our desired unique morphism $t \colon A \to X^\pbsquare$ with $t \pi_1 X^{i_0} = q$.
\end{proof}

\begin{proof}[\textrm{\bf Proof of Proposition~\ref{propn:sketch-cotensors-exist}}.]
  Fix a finite pbt sketch $J = (\abs J,P,T)$ and an object $X \in \CC$ having finite limits.
  To begin with, we form a cotensor $X^{\abs{J}}$.
  Each $p \in P$ is a commutative square $\square \tox{p} \abs{J}$, and gives a morphism $X^p \colon X^{\abs{J}} \to X^{\square}$.
  Now form the cotensor $X^\pbsquare$, so that we have a morphism $\phi_{X,\pbsquare} \colon X^\pbsquare \to X^\square$ as in (\ref{eq:sketch-cotensor-forget}), and form a strict pullback
  \[
    \begin{tikzcd}
      X^{J,\mathrm{pb}}\ar[r, "\pi_1"]\ar[d, "\pi_0"']\pb&[20pt]\prod_{p \in P}X^\pbsquare\ar[d, "\prod_{p \in P}\phi_{X,\pbsquare}"]\\
      X^{\abs{J}}\ar[r, "\br{X^p}_{p \in P}"]&\prod_{p \in P}X^\square.
    \end{tikzcd}
  \]
  Note that these are finite products since $P$ is finite, and that the morphism on the right is an isofibration.

  The morphism $\pi_0$ is fully faithful since $\prod_{p \in P}\phi_{X,\pbsquare}$ is, and we see immediately that a morphism $A \to X^{\abs J}$ factors through $\pi_0$ if and only if the functor $J \to \CC(A,X)$ which it classifies takes each pullback square $p \colon \square \to J$ in $P$ to a stable pullback square in $\CC(A,X)$.

  To finish the construction, we now do the same thing but with terminal objects.
  Each $t \in T$ gives a morphism $\tm \tox{t} \abs{J}$ and hence a morphism $X^{J,\mathrm{pb}} \tox{\pi_0} X^{\abs{J}} \tox{t^*} X$.
  We then form a cotensor $X^{\tm_!}$ of $X$ with the sketch $\tm_!$ consisting of a single terminal object, so that we have a morphism $\phi_{X,\tm_!} \colon X^{\tm_!} \to X$, and we form the pullback
  \[
    \begin{tikzcd}
      X^J\ar[r, ""]\ar[d, ""'] \pb &[25pt]
      \prod_{t \in T}X^{\tm_!}\ar[d, "\prod_{t \in T}\phi_{X,\tm!}"]\\
      X^{J,\mathrm{pb}}\ar[r, "\br{\pi_0 t^*}_{t \in T}"] &
      \prod_{t \in T}X
    \end{tikzcd}
  \]
  Then $X^J$ is the desired cotensor.
\end{proof}

\begin{cor}[Proposition~\ref{propn:mon-exists}]\label{cor:pbt-monos}
  If $\CC$ is pita and $X \in \CC$ has finite limits, then the object $\Mon(X)$ of monomorphisms in $X$ exists.
\end{cor}
\begin{proof}
  Take $J = (\abs{J},P,T)$ to be the pbt sketch where $\abs{J} = [1]$ is the free-standing arrow (as in \S\ref{subsec:nerves}), and where $T = \emptyset$ and $P$ consists of the single square
  \[
    \begin{tikzcd}
      0\ar[r]\ar[d]&0\ar[d]\\
      0\ar[r]&1.
    \end{tikzcd}
  \]
  Then the cotensor $X^J$ has the universal property of $\Mon(X)$.
\end{proof}

We end by noting that there is an obvious notion of \defword{tbp-sketch} (``terminal object and binary products'') analogous to that of pbt-sketch.
The proof of Proposition~\ref{propn:sketch-cotensors-exist} works just as well for tbp-sketches and gives:
\begin{propn}\label{propn:tbp-sketch}
  If $X \in \CC$ has finite products and $\CC$ is pita, then $X^J$ exists for any tbp-sketch $J$.
\end{propn}

\section{Power objects and exponentials}\label{sec:groupoid-stuff}
This appendix has three objectives.

The first is to introduce the notion of an object $X$ in a 2-category $\CC$ \emph{having exponential objects} (or being \emph{cartesian closed}).
Since we defined an internal elementary topos $X$ as an object having finite limits and power objects---as opposed to having finite limits, exponentials, and a subobject classifier---we managed to avoid discussing exponentials, and we now fill that gap.
The definition is the expected ``representable'' one, and like the definition of $X$ having power objects (and for the same reason), it makes reference only to the \emph{groupoids} in $\CC$.

Secondly, we give alternative, (non-``representable'') definitions of an object $X$ having finite limits, power objects, or exponentials; in particular, in the latter two cases, the definition avoids making special reference to the groupoids in $\CC$ (though the definition does still make use of the cores of certain objects).
The definition in all three cases follows the same general pattern, reminiscent of the ``operations on diagrams'' introduced in \cite[\S1.1]{makkai-ultraproducts}.
In the latter two cases, the definition also makes use of the cotensors with sketches from Appendix~\ref{sec:pbt-skectches}.

In the case of finite limits, the alternative definition is, in general, not quite equivalent to $X$ having finite limits, but rather $\CC(A,X)$ having finite limits for each \emph{groupoid} $A$, and these being stable along morphisms $A' \to A$ with $A'$ a groupoid.
However, \emph{assuming $\CC$ is groupoidant}, this latter condition is seen, using the results of \S\ref{subsec:2-yoneda}, to be equivalent to $X$ having finite limits.

Thus, altogether, the alternative definitions are meant to deepen the conviction that, in the context of a groupoidant 2-category $\CC$, the definitions of an object $X$ having power objects or exponentials are ``correct''---in spite of the fact that they seemingly somewhat artificially make reference only to the groupoids in $\CC$.

The third objective is also related to this: we show, again assuming $\CC$ is groupoidant---and again using the results of \S\ref{subsec:2-yoneda}---that if a an object $X$ has power objects or exponentials, then the natural covariant and contravariant power functors $X \to X$ and $X^\op \to X$ and exponential functor $X^\op \times X \to X$ can be defined internally.

Related to this, we end with some comments on the question of the cocompleteness of a plentiful generic DOF $\s$ in Remark~\ref{rmk:cocompleteness}.

\subsection{Finite limits}\label{subsec:gpd-stuff-finite-lims}
Recall that given two categories $\cC,\cD$, their \defword{join} $\cC * \cD$ is formed from the disjoint union of $\cC$ and $\cD$ by adjoining a single morphism from each object in $\cC$ to each object in $\cD$.

Given a category $J$, we set $\wh{J} = [0] * J$ (this is the \emph{cone} on $J$), and $\wwh{J} = [1] * J$ (where $[0]$ and $[1]$ are as in \S\ref{subsec:nerves}).

We write $i_0,i_1 \colon \wh{J} \to \wwh{J}$ for $\delta_0*\id_J,\delta_1*\id_J$.

Given a pita category $\CC$ and objects $A,X \in \CC$, we then have that for a diagram $u \colon A \to X^J$ of shape $J$, a limit of $u$ is the same thing as a lift $\bar{u} \colon A \to X^{\wh{J}}$ of $u$, such that for any other lift $v$ of $u$, there is a unique 2-cell $\alpha \colon v \to u$ with $\alpha X^{i} = \id_u$, where $i \colon J \to \wh{J}$ is the canonical inclusion into the join $[0] * J$.
\[
  \begin{tikzcd}
    &[30pt] X^{\wh{J}}\ar[d, "X^{i}"]\\
    A\ar[r, "u"']
    &X^J
    \ar["\forall v", ""{name=0, anchor=center, inner sep=0}, bend left=30pt, from=2-1, to=1-2]
    \ar["\bar{u}"', ""{name=1, anchor=center, inner sep=0}, from=2-1, to=1-2]
    \ar["\exists!\,\alpha", shorten <=3pt, shorten >=3pt, Rightarrow, dashed, from=0, to=1].
  \end{tikzcd}
\]
That $\bar{u}$ is a \emph{stable} limit is expressed by requiring that for each $f \colon A' \to A$, the composite $f\bar{u} \colon A' \to X^{\wh{J}}$ is still a limit (of $f u$).

Note, moreover, that given $v \colon A \to X^{\wh{J}}$, a 2-cell $\alpha \colon v \to u$ with $\alpha X^{i} = \id_u$ is the same thing as a morphism $w \colon A \to X^{\wwh{J}}$ with $v = w X^{i_0} \colon A \to X^{\wh{J}}$ and $u = w X^{i_1} \colon A \to X^{\wh{J}}$.
\begin{equation}\label{eq:diag-limt-cond}
  \begin{tikzcd}
    &[30pt]
    X^{\wwh{J}}
    \ar[d, "X^{i_0}"', shift right]
    \ar[d, "X^{i_1}", shift left]\\
    &X^{\wh{J}}\ar[d, "X^{i}"]\\
    A\ar[r, "u"']
    \ar[ru, "v", shift left]
    \ar[ru, "\bar{u}"', shift right]
    \ar[ruu, "w", bend left=10pt]
    &X^J
  \end{tikzcd}
\end{equation}

\begin{propn}\label{propn:fin-lims-equiv}
  Let $\CC$ be a pita 2-category.
  Then for any $X \in \CC$ and any finite category $J$, the following are equivalent:
  \begin{enumerate}
  \item[(i)] $X$ has finite limits of shape $J$.
  \item[(ii)] There exists a morphism $\ell \colon X^J \to X^{\wh{J}}$ with $\ell X^{i} = \id_{X^J}$ and satisfying the following equivalent conditions:
    \begin{enumerate}
    \item[(ii-a)] For each $A \in X$ and each diagram $u \colon A \to X^J$ of shape $J$ in $\CC(A,X)$, the cone $u \ell \colon A \to X^{\wh{J}}$ over $u$ is a limit cone.
    \item[(ii-b)] If we form a strict pullback of the isofibration $X^{i_1} \colon X^{\wwh{J}} \to X^{\wh{J}}$ along $\ell$
      \[
        \begin{tikzcd}
          Y\ar[r, "\pi_1"]\ar[d, "\pi_0"']\pb&X^{\wwh{J}}\ar[d, "X^{i_1}"]\ar[r, "X^{i_0}"]&X^{\wh{J}}\\
          X^J\ar[r, "\ell"]&X^{\wh{J}}
          &
          X^J
          \ar[from=l, "X^{i}"]
          \ar[from=u, "X^{i}"]
          ,
        \end{tikzcd}
      \]
      then the composite $\pi_1X^{i_0}$ is an isomorphism.
    \end{enumerate}
  \end{enumerate}
\end{propn}
\begin{proof}
  If $X$ has finite limits of shape $J$, we take $\ell \colon X^J \to X^{\wh{J}}$ to be a (necessarily stable) limit cone over $\id_{X^J} \colon X^J \to X^J$; then $\ell X^{i} = \id_{X^J}$, and by stability, $u \ell \colon A \to X^{\wh{J}}$ is a limit cone in $\CC(A,X)$ for each $u \colon A \to X^J$, whence (ii-a).

  Conversely, given $\ell \colon X^J \to X^{\wh{J}}$ as in (ii-a), each diagram $u \colon A \to X^J$ of shape $J$ has a limit $u \ell$, and it is stable because, given $f \colon A' \to A$, the diagram $f (u \ell) = (f u) \ell \colon A' \to X^{\wh{J}}$ is by assumption again a limit cone.

  It remains to see (ii-a)~$\ToT$~(ii-b), for which we use the above characterization of limits, according to which (ii-a) is equivalent to there existing, for each $v \colon A \to X^{\wh{J}}$, a unique $w \colon A \to X^{\wwh{J}}$ with
  \[
    \tag{$\star$}\label{eq:fin-lims-equiv-cond-on-w}
    w X^{i_0} = v \quad \AND \quad w X^{i_1} = v X^{i} \ell.
  \]

  But now note that $(\pi_1)_* \colon \CC(A,Y) \to \CC(A,X^{\wwh{J}})$ gives a bijection between those $y \colon A \to Y$ with $y \pi_1 X^{i_0} = v$ and those $w \colon A \to X^{\wwh{J}}$ satisfying (\ref{eq:fin-lims-equiv-cond-on-w}).

  Hence, we see that (ii-a) is equivalent to there being, for each $v \colon A \to X^{\wh{J}}$, a unique $y \colon A \to Y$ with $y \pi_1 X^{i_0} = v$, i.e., to $\pi_1 X^{i_0}$ being an isomorphism.
  \qedhere
\end{proof}
\begin{rmk}
  A third equivalent condition to (ii-a) and (ii-b) in the above proposition is that $\ell \colon X^J \to X^{\wh{J}}$ is right adjoint to $X^{i} \colon X^{\wh{J}} \to X^J$.
\end{rmk}

\begin{propn}\label{propn:gpd-fin-lims-equiv}
  Let $\CC$ be a corepita 2-category.
  Then for any $X \in \CC$ and finite category $J$, the following are equivalent:
  \begin{enumerate}
  \item[(i)] $\CC(A,X)$ has finite limits of shape $J$ for all groupoids $A \in \CC$ and these are preserved by $f^*$ for all $f \colon A' \to A$ with $A'$ a groupoid.
  \item[(ii)] There exists a morphism $\ell \colon (X^J)^\iso \to (X^{\wh{J}})^\iso$ with $\ell (X^{i})^{\iso} = \id_{(X^J)^\iso}$ and satisfying the following equivalent conditions:
    \begin{enumerate}
    \item[(ii-a)] For each groupoid $A \in X$ and each diagram $u \colon A \to X^J$ of shape $J$ in $\CC(A,X)$, the cone $u^\iso \ell i \colon A \to X^{\wh{J}}$ over $u$ is a limit cone.
    \item[(ii-b)] If we form a strict pullback of the isofibration $(X^{i_1})^{\iso} \colon (X^{\wwh{J}})^\iso \to (X^{\wh{J}})^\iso$ along $\ell$
      \[
        \begin{tikzcd}
          Y\ar[r, "\pi_1"]\ar[d, "\pi_0"']\pb&(X^{\wwh{J}})^{\iso}\ar[d, "(X^{i_1})^{\iso}"]\ar[r, "(X^{i_0})^{\iso}"]&(X^{\wh{J}})^{\iso}\\
          (X^J)^{\iso}\ar[r, "\ell"]&(X^{\wh{J}})^{\iso}
          &
          (X^J)^\iso
          \ar[from=l, "X^{i}"]
          \ar[from=u, "X^{i}"]
          ,
        \end{tikzcd}
      \]
      then the composite $\pi_1(X^{i_0})^{\iso}$ is an isomorphism.
    \end{enumerate}
  \end{enumerate}
\end{propn}

\begin{proof}
  The proof is essentially the same as that of the previous proposition, where we use that in the characterization (\ref{eq:diag-limt-cond}) of stable limits, if $A$ is a groupoid, we may replace $X^J$, $X^{\wh{J}}$, $X^{i}$, and so on, by $(X^J)^{\iso}$, $(X^{\wh{J}})^{\iso}$, $(X^{i})^{\iso}$, and so on; and also that a morphism $f \colon X \to Y$ of groupoids in $\CC$ is an isomorphism if and only if $f_* \colon \CC(A,X) \to \CC(A,Y)$ is a bijection on objects for all groupoids $A \in \CC$.
\end{proof}

Now we come to the main point:
\begin{propn}\label{propn:lim-gpd-nongpd-equiv}
  If $\CC$ is a groupoidant 2-category, then for any $X \in \CC$ and finite category $J$, the conditions of Propositions~\ref{propn:fin-lims-equiv}~and~\ref{propn:gpd-fin-lims-equiv} are equivalent.
\end{propn}

\begin{proof}
  Clearly, the condition Proposition~\ref{propn:fin-lims-equiv}.(i) is stronger than Proposition~\ref{propn:gpd-fin-lims-equiv}.(i), so we only need to prove the reverse implication.

  Thus, fix a morphism $\ell \colon (X^J)^\iso \to X^{\wh{J}}$ as in Proposition~\ref{propn:fin-lims-equiv}.(ii), so that, for each diagram $u \colon A \to X^J$ with $A$ a groupoid, we have a stable limit diagram $u^\iso \ell i \colon A \to X^{\wh{J}}$ over $u$.

  For each groupoid $A$, we then have a functor $F_A \colon \CC(A,X^J) \to \CC(A,X^{\wh{J}})$ taking each diagram $u$ to its chosen stable limit diagram $u^\iso \ell i$, and whose action on morphisms is determined by the additional condition $F_A (X^{i})_* = \id_{\CC(A,X^J)}$.

  Given $u \colon A' \to A$, we have, for each $v \colon A \to X^J$, that $u \cdot F_A(v) = u v^\iso \ell i = F_{A'}(uv)$, and likewise with morphisms in $\CC(A,X^J)$.
  Hence we see that the $F_A$ constitute a natural transformation $F \colon (X^J)_\gpd \to (X^{\wh{J}})_\gpd$.

  Thus, by Theorem~\ref{thm:gpd-yoneda-rstr}~(iii), there is a (unique) morphism $\ell' \colon X^J \to X^{\wh{J}}$ with $F = \ell'_*$.
  It follows that $\ell'$ satisfies the condition of Proposition~\ref{propn:fin-lims-equiv}~(ii-a) with respect to \emph{groupoids} $A$ and hence, by the proof of that proposition, that $(\pi_1X^{i_0})_* \colon \CC(A,Y) \to \CC(A,X^{\wh{J}})$ is a bijection on objects for all groupoids $A$.
  By a similar argument, one can also show that $(\pi_1X^{i_0})_*$ is in fact an isomorphism.
  (Here, one uses that there is a unique morphism of limit diagrams of shape $\wh{J}$ extending any given morphism of their restrictions to $J$---and similarly, there is a unique morphism of diagrams of shape $\wwh{J}$ extending a given morphism between their restrictions along $\wh{J} \htox{i_0} \wwh{J}$, assuming the restrictions along $\wh{J} \htox{i_1} \wwh{J}$ are limit diagrams.)

  Hence, using Theorem~\ref{thm:gpd-yoneda-rstr}~(i)~and~(iii), we conclude that $\pi_1X^{i_0}$ is an isomorphism, so $\ell'$ satisfies the condition Proposition~\ref{propn:fin-lims-equiv}~(ii-b).
\end{proof}

\subsection{Power objects}\label{subsec:gpd-stuff-powers}
We now discuss power objects.
We begin by observing that they can be characterized in a similar way to the characterization of finite limits at the beginning of \S\ref{subsec:gpd-stuff-finite-lims} above.

We define the following finite categories.
\begin{equation}\label{eq:pow-pbt-sketches}
  \abs{K} = \pbig{0}
  \quad
  \abs{\wc{K}} = \left(
    \begin{tikzcd}[sep=13pt,font=\small]
      &3\ar[d, >->]\\
      &2\ar[ld]\ar[rd]\ar[ld]\ar[d, phantom, "\lrcorner"{rotate=-45, pos=0}]\\
      0&{}&1
    \end{tikzcd}
  \right)
  \quad
  \abs{\wwc{K}} = \left(
    \begin{tikzcd}[sep=13pt,font=\small]
      &3'\ar[rr]\ar[d, >->]\pb&{}&3\ar[d, >->]\\
      &2'\ar[rr]\ar[rd]\ar[ld]\ar[d, phantom, "\lrcorner"{rotate=-45, pos=0}]&{}&2\ar[ld]\ar[rd]\ar[d, phantom, "\lrcorner"{rotate=-45, pos=0}]\\
      1'\ar[rrrr, bend right=8pt]&{}&0&{}&1
    \end{tikzcd}
  \right)
\end{equation}
Each category is a preorder generated by the displayed graph.

Note that we have an inclusion $i \colon \abs{K} \to \abs{\wc{K}}$, as well as two embeddings $i', i'' \colon \abs{\wc{K}} \to \abs{\wwc{K}}$, the first being an inclusion, and the second taking $1,2,3$ to $1',2',3'$.

Moreover, we make $\abs{K}$ into a pbt sketch (see Appendix~\ref{sec:pbt-skectches}) $K = (\abs{K},P_K,T_K)$, and likewise with $\abs{\wc{K}}$ and $\abs{\wwc{K}}$, in the way indicated in the diagram; and we see that the functors $i,i',i''$ then become morphisms of sketches.
Here, the monomorphisms $\tto$ are encoded in a pbt sketch as in the proof of Corollary~\ref{cor:pbt-monos}, and the ``\rotatebox[origin=r]{-45}{$\lrcorner$}'' symbols indicate product diagrams.
(Note that in order to represent the desired products as pullbacks, we must add an additional object to the diagram and declare it to be terminal---or more directly, we can just extend the notion of pbt sketch to also allow for the specification of binary products.
We assume one of these two things has been done.)

Now, given a pita 2-category $\CC$, observe that for an object $X \in \CC$ with finite limits, $X$ has power objects if and only if for each groupoid $A \in \CC$ and each $u \colon A \to X^{K} = X$, there is a morphism $\bar{u} \colon A \to X^{\wc{K}}$ with $\bar{u} X^i = u$ and such that, for any $v \colon A \to X^{\wc{K}}$ with $v X^{i} = u$, there is a unique $w \colon A \to X^{\wwc{K}}$ with $w X^{i} = \bar{u}$ and $w X^{i'} = v$ (compare (\ref{eq:diag-limt-cond})).
\[
  \begin{tikzcd}
    &[30pt]
    X^{\wwc{K}}
    \ar[d, "X^{i'}"', shift right, pos=0.6]
    \ar[d, "X^{i}", shift left]\\
    &X^{\wc{K}}\ar[d, "X^{i}"]\\
    A\ar[r, "u"']
    \ar[ru, "v", shift left]
    \ar[ru, "\bar{u}"', shift right]
    \ar[ruu, "w", bend left=10pt]
    &X
  \end{tikzcd}
\]

The proof of Proposition~\ref{propn:gpd-fin-lims-equiv} now carries over immediately to the present context to give:
\begin{propn}\label{propn:gpd-pow-objs-equiv}
  Let $\CC$ be a corepita 2-category.
  Then for any $X \in \CC$ with finite limits, the following are equivalent:
  \begin{enumerate}
  \item[(i)] $X$ has power objects (i.e., is an internal elementary topos).
  \item[(ii)] There exists a morphism $P \colon X^\iso = (X^K)^\iso \to (X^{\wc{K}})^\iso$ with $P (X^{i})^{\iso} = \id_{X^\iso}$ and satisfying the following equivalent conditions:
    \begin{enumerate}
    \item[(ii-a)] For each groupoid $A \in X$ and each $u \colon A \to X$, the model $u^\iso P i \colon A \to X^{\wc{K}}$ of $\wc{K}$ in $\CC(A,X)$ is a power object diagram.
    \item[(ii-b)] If we form a strict pullback of the isofibration $(X^{i'})^{\iso} \colon (X^{\wwc{K}})^\iso \to (X^{\wc{K}})^\iso$ along $P$
      \[
        \begin{tikzcd}
          Y\ar[r, "\pi_1"]\ar[d, "\pi_0"']\pb&(X^{\wwc{K}})^{\iso}\ar[d, "(X^{i'})^{\iso}"]\ar[r, "(X^{i''})^{\iso}"]&(X^{\wc{K}})^{\iso}\\
          X^{\iso}\ar[r, "P"]&(X^{\wc{K}})^{\iso}
          &
          X^\iso
          \ar[from=l, "(X^{i})^\iso"]
          \ar[from=u, "(X^{i})^\iso"]
          ,
        \end{tikzcd}
      \]
      then the composite $\pi_1(X^{i''})^{\iso}$ is an isomorphism.
      \hfill
      \qed
    \end{enumerate}
  \end{enumerate}
\end{propn}
Now let us see that we can construct the covariant and contravariant power object morphisms $X \to X$ and $X^\op \to X$ for any internal elementary topos $X$.

Given objects $A,B \in \cC$ in an elementary topos $\cC$, and power object diagrams on $A$ and $B$ with associated power objects $PA$ and $PB$, any morphism $f \colon A \to B$ induces, in a well-known manner, morphisms $f_* \colon PA \to PB$ and $f^* \colon PB \to PA$ (see e.g., \cite[IV.1,IV.3]{mac-lane-moerdijk}).
Thus, a choice of power object diagram for each $A \in \cC$ gives rise to covariant and contravariant power object functors $\pow \colon \cC \to \cC$ and $\pow \colon \cC^\op \to \cC$.

A choice of power object diagram for each object corresponds to a section $s \colon \cC^\iso \to (\cC^{\wc{K}})^\iso$ of the projection $(\cC^i)^\iso \colon (\cC^{\wc{K}})^\iso \to \cC^\iso$ (with $\wc{K}$ as in (\ref{eq:pow-pbt-sketches}) above) taking each object $A \in \cC$ to a power object diagram $s(A)$.

\begin{defn}\label{defn:power-morphism}
  Given a corepita 2-category $\CC$ and an internal elementary topos $X \in \CC$, we say that a morphism $\pow \colon X \to X$ is a \defword{covariant power morphism} if there exists a morphism $P \colon X^\iso \to (X^{\wc{K}})^\iso$ as in Proposition~\ref{propn:gpd-pow-objs-equiv} such that:
  \begin{enumerate}[(i)]
  \item For each groupoid $A \in \CC$ and $f \colon A \to X$, the object $f \pow \in \CC(A,X)$ is the power object underlying the power object diagram $f^\iso Pi \colon A \to X^{\wc{K}}$.
  \item For each groupoid $A$ and 2-cell $\alpha \colon A \tocellud{f}{g} X$, the morphism $\alpha \pow \colon f \pow \to g \pow$ in $\CC(A,X)$ is the morphism between power objects induced by $\alpha$ and the power object diagrams $f^\iso P i$ and $g^\iso P i$.
  \end{enumerate}
  A \defword{contravariant power morphism} $\pow \colon X^\op \to X$ is defined similarly, except that now for $f \colon A \to X$, it is $A \tox{f^\op} X^\op \tox{\pow} X$ which should be the power object underlying $f^\iso P i$, and for a 2-cell $\alpha \colon A \tocellud{f}{g} X$, the morphism $\alpha^\op \pow \colon g^\op \pow \to g^\op \pow$ should be the contravariant morphism induced by $\alpha$, $f^\iso P i$, and $g^\iso P i$.
\end{defn}

\begin{propn}\label{propn:power-morphisms-exist}
  Let $\CC$ be a groupoidant 2-category and suppose $X \in \CC$ is an internal elementary topos.
  Then there exist covariant and contravariant power morphisms for $X$, and they are each unique up to isomorphism.
\end{propn}

\begin{proof}
  We treat the case of the covariant power morphism; the proof for the contravariant power morphism is similar, using (the first part of) Proposition~\ref{propn:rep-op-defn}.

  Fix a morphism $P \colon X^\iso \to (X^{\wc{K}})^\iso$ as in Proposition~\ref{propn:gpd-pow-objs-equiv}.

  For each groupoid $A \in \CC$, consider the functor $F_A \colon \CC(A,X) \to \CC(A,X)$ taking each $f$ to the power object underlying the power object diagram $f^\iso Pi$, and having the evident action on morphisms, as in (ii) above.

  An argument similar to that in the proof of Proposition~\ref{propn:lim-gpd-nongpd-equiv} shows that the $F_A$ can be made into a natural transformation $F \colon (\wh{X})_{\iso} \to (\wh{X})_{\iso}$.
  Here, we make use the fact that any functor preserving (finite limits and) power objects also preserves the induced morphisms between power objects.

  Hence by Theorem~\ref{thm:gpd-yoneda-rstr}, there exists a unique morphism $\pow \colon X \to X$ with $\pow_* = F$.

  To prove the uniqueness of $\pow$ up to isomorphism, it suffices by Corollary~\ref{cor:yoneda-iso} to show that given a second morphism $P' \colon X^\iso \to (X^{\wc{K}})^\iso$ as above, the resulting natural transformation $F' \colon (\wh{X})_{\iso} \to (\wh{X})_{\iso}$ is isomorphic to $F$.

  For each groupoid $A \in \CC$, we define a natural isomorphism $\phi_A \colon F_A \toi F_A'$ by letting $(\phi_A)_f$ for $f \colon A \to X$ be the morphism $F_A(f) \to F_A'(f)$ induced by $\id_f$ with respect to the chosen power object diagrams $f^\iso Pi$ and $f^\iso P'i$.
  One then checks that the $\phi_A$ constitute a modification, again using the aforementioned fact that functors preserving power objects also preserve the induced morphisms.
\end{proof}

\subsection{Exponentials}\label{subsec:exponentials}
We now turn to the case of cartesian closed objects $X \in \CC$.
We have not yet introduced this concept, but there is an obvious definition analogous to that of $X$ having power objects: $X$ is \defword{cartesian closed} if for each groupoid $A \in \CC$, the category $\CC(A,X)$ is cartesian closed, and for each $f \colon A' \to A$ with $A'$ a groupoid, the functor $f^ * \colon \CC(A,X) \to \CC(A',X)$ preserves products and exponentials.

As with power objects, there is no obvious way to extend this condition to non-groupoidal $A$, because exponentials have the opposite covariance in their two arguments---and moreover, it is simply not the case in $\Cat$ for a general non-groupoidal $A$ that if $X$ is cartesian closed, then so is $\Cat(A,X)$.

Now, with the same conventions as in (\ref{eq:pow-pbt-sketches}) above, define the following finite tbp sketches (see Proposition~\ref{propn:tbp-sketch}).
\[
  L = \left(
    \begin{tikzcd}[sep=13pt,font=\small]
      0&1
    \end{tikzcd}
  \right)
  \quad
  \wc{L} = \left(
    \begin{tikzcd}[sep=13pt,font=\small]
      2&3\ar[l]\ar[rd]\ar[rrd, bend left]\ar[d, phantom, "\lrcorner"{rotate=-45, pos=0}]\\
      &{}&0&1
    \end{tikzcd}
  \right)
  \quad
  \wwc{L} = \left(
    \begin{tikzcd}[sep=13pt,font=\small]
      2&3\ar[l]\ar[rd]\ar[rrd, bend left]\ar[d, phantom, "\lrcorner"{rotate=-45, pos=0}]\\
      &{}&0&1\\
      2'\ar[uu]&3'\ar[l]\ar[ru]\ar[uu, shorten={10pt}]\ar[rru, bend right]
      \ar[u, phantom, "\urcorner"{rotate=45, pos=0}]
    \end{tikzcd}
  \right)
\]
Again, we have an inclusion $i \colon L \to \wc{L}$ and embeddings $i',i'' \colon \wc{L} \to \wwc{L}$

The same proof as that of Propositions~\ref{propn:gpd-fin-lims-equiv}~and~\ref{propn:gpd-pow-objs-equiv} gives:
\begin{propn}\label{propn:gpd-exp-equiv}
  Let $\CC$ be a corepita 2-category.
  Then for any $X \in \CC$ with finite products, the following are equivalent:
  \begin{enumerate}
  \item[(i)] $X$ is cartesian closed.
  \item[(ii)] There exists a morphism $E \colon (X \times X)^\iso = (X^L)^\iso \to (X^{\wc{L}})^\iso$ with $E (X^{i})^{\iso} = \id_{(X \times X)^\iso}$ and satisfying the following equivalent conditions:
    \begin{enumerate}
    \item[(ii-a)] For each groupoid $A \in X$ and each $f,g \colon A \to X$, the model $\br{f,g}^\iso E i \colon A \to X^{\wc{L}}$ of $\wc{L}$ in $\CC(A,X)$ is an exponential object diagram.
    \item[(ii-b)] If we form a strict pullback along the isofibration $(X^{i'})^{\iso} \colon (X^{\wwc{L}})^\iso \to (X^{\wc{L}})^\iso$ along $E$
      \[
        \begin{tikzcd}
          P\ar[r, "\pi_1"]\ar[d, "\pi_0"']\pb&(X^{\wwc{L}})^{\iso}\ar[d, "(X^{i'})^{\iso}"]\ar[r, "(X^{i''})^{\iso}"]&(X^{\wc{L}})^{\iso}\\
          (X \times X)^{\iso}\ar[r, "E"]&(X^{\wc{L}})^{\iso}
          &
          (X \times X)^\iso
          \ar[from=l, "X^{i}"]
          \ar[from=u, "X^{i}"]
          ,
        \end{tikzcd}
      \]
      then the composite $\pi_1(X^{i''})^{\iso}$ is an isomorphism.
      \hfill
      \qed
    \end{enumerate}
  \end{enumerate}
\end{propn}

Next, let us see that we can construct the exponential functor for a cartesian closed object $X \in \CC$ in a groupoidant 2-category $\CC$.

For such an $X$, we define the notion of \defword{exponential morphism} $\exp \colon X^\op \times X \to X$ in an analogous manner to that of power morphism (Definition~\ref{defn:power-morphism}): there should be a morphism $E \colon X^\iso \to (X^{\wc{L}})^\iso$ as in Proposition~\ref{propn:gpd-exp-equiv} such that, for each $f,g \colon A \to X$ with $A$ a groupoid, the object $\br{f^\op,g}\exp \in \CC(A,X)$ is the exponential object underlying the diagram $\br{f^\op,g}^\iso Ei$, and given $f',g' \colon A \to X$ and 2-cells $\alpha \colon f' \to f$ and $\beta \colon g \to g'$, the 2-cell $\br{\alpha^\op,\beta} \colon \br{f^\op,g} E \to \br{(f')^\op,g'} E$ is the corresponding induced morphism.

The proof of Proposition~\ref{propn:power-morphisms-exist} carries over to the present context to give:
\begin{propn}\label{propn:exp-morphisms-exist}
  If $\CC$ is a groupoidant 2-category, then for any cartesian closed object $X \in \CC$, there is an exponential morphism $\exp \colon X^\op \times X \to X$, and it is unique up to isomorphism.
\end{propn}

Finally, we point out that the notion of $X$ being cartesian closed discussed here is (assuming $\CC$ is groupoidant!) stronger than the one given in \cite[Definition~8.1]{weber-2-toposes}:
\begin{propn}\label{propn:weak-cart-clos}
  Let $\CC$ be a groupoidant 2-category.
  If $X \in \CC$ is cartesian closed in the above sense, then for any morphism from the terminal object $x \colon \tm \to X$, the morphism $(- \times x) \colon X \to X$ defined as the composite
  \[
    X \toi X \times \tm \tox{\id_X \times x} X \times X \tox{m} X,
  \]
  has a right adjoint, where $m$ is a right adjoint to $\Delta \colon X \to X \times X$.
\end{propn}
\begin{proof}
  We fix an exponential morphism $\exp \colon X^\op \times X \to X$ and define our putative right adjoint $(-)^x \colon X \to X$ as the composite
  \[
    X
    \tox{\br{!_Xx^\op,\id}}
    X^\op \times X
    \tox{\exp}
    X,
  \]
  where $! = !_X \colon X \to \tm$ is the morphism to the terminal object.
  We want to see that this is indeed a right adjoint.

  First, a few preliminary remarks.
  By definition, associated with the exponential morphism $\exp$ is the morphism
  $E \colon (X \times X)^\iso \to (X^{\wc{L}})^\iso$.

  Let us write $S$ for the tbp-sketch with a single product diagram $S = (2 \ot 3 \to 0)$, so that we have an inclusion $j \colon S \to \wc{L}$.
  Note that the composite $E i X^{j} \colon (X \times X)^\iso \to X^S$ is a product diagram over the objects
  \[
    \br{\pi_0^\iso i^\op,\pi_1^\iso i}\exp
    ,\ \pi_0^\iso i^\op
    \in \CC((X \times X)^\iso,X).
  \]
  On the other hand, the morphism $m$ (or rather the counit of the corresponding adjunction) also determines a product diagram over these objects.

  We can assume that $E$ has been so chosen so that these agree; i.e., that the product appearing in the universal exponential object $E$ is the one determined by $m$.

  Now, fix a groupoid $A \in \CC$.
  Since $\CC(A,-) \colon \CC \to \Cat$ preserves adjunctions, we have that $\CC(A,X) \times \CC(A,X) \toi \CC(A,X \times X) \tox{m_*} \CC(A,X)$ is a product functor, and in particular that the functor $\pbig{(- \times x)_*}_A \colon \CC(A,X) \to \CC(A,X)$ takes each object $f \in \CC(A, X)$ to a product $f \times (!x) \in \CC(A,X)$ in the category \( \CC(A,X) \), and takes each morphism $f \to f'$ in \( \CC(A,X) \) to the induced morphism $f \times (!x) \to f' \times (!x)$.

  On the other hand, by definition of the exponential morphism $\exp$, the functor $\pbig{(-)^x_*}_A \colon \CC(A,X) \to \CC(A,X)$ is just the functor forming an exponential object with $!x \in \CC(A,X)$.

  It follows that we have an adjunction $\pbig{(- \times x)_*}_A\dashv\pbig{(-)^x_*}_A$ for each groupoid $A$.
  We claim that the unit and counit of these adjunctions constitute modifications $\id_{(\wh{X})^\iso} \to (- \times x)_* \pbig{(-)^x}_*$ and $\pbig{(-)^x}_* (- \times x)_* \to \id_{(\wh{X})^\iso}$, respectively .

  Indeed, for fixed $f \colon A \to X$ ($A$ a groupoid), the counit at $A$ is just the evaluation morphism of the exponential diagram for $f^{!x}$.
  This exponential diagram is
  \[
    A \tox{\br{!x,f}^\iso} (X \times X)^\iso \tox{Ei} X^{\wh{L}}
  \]
  and hence the evaluation morphism is the image of the arrow $3 \to 1$ of $\wc{L}$ in this model of $\wc{L}$.
  This defines a modification because, for $u \colon A' \to A$, we have $u \br{!x,f}^\iso E i = \br{!x,uf}^\iso E i$.

  Similarly, for the unit, we have the canonical morphism $\id_X \to \br{!x^\op,\br{!x,\id_X} m} \exp$ in $\CC(X,X)$ and the unit at $f \colon A \to X$ is just given by composing with this morphism, which again clearly determines a modification.

  We conclude from Theorem~\ref{thm:gpd-yoneda-rstr}~(iv) that we have 2-cells $\eta \colon \id_X \to (- \times x)(-)^x$ and $\varepsilon \colon (-)^x(- \times x) \to \id_X$ giving rise to these modifications, and it then follows from Theorem~\ref{thm:gpd-yoneda-rstr}~(ii) that these are a unit and counit giving the desired adjunction.
\end{proof}

\begin{rmk}\label{rmk:cocompleteness}
  In a 2-topos in our sense (more specifically, in a groupoidant 2-category), we have that any plentiful DOF classifier is (an internal elementary topos and hence) cartesian closed.
  On the other hand, in \cite[Theorem~8.6]{weber-2-toposes}, it is proven that a DOF classifier $\s$ is cartesian closed (in the sense of Proposition~\ref{propn:weak-cart-clos}) under the assumption that $\s$ is \emph{cocomplete} (i.e., has ``$\s$-small colimits'') in the sense of \cite[Definition~3.8~and~Corollary~5.4]{weber-2-toposes} (a notion originating in \cite{street-walters-yoneda-structure}).

  This thus raises the question of whether a plentiful generic DOF $\s$ is automatically cocomplete, so that Weber's proof of cartesian closedness applies in our setting.
  We do not know the answer to this question, but we will make some general remarks on the question of cocompleteness.

  For one, we \emph{can} show that $\s$ is cocomplete if we take this to mean ``having finite colimits and arbitrary coproducts''.
  Indeed, since $\s$ is an elementary topos, it does have (stable) finite colimits.
  And it is not hard to see that $\s$ has arbitrary (weighted) colimits of diagrams $M \to \s$ in the sense of \loccit whenever $M$ is \emph{$\s$-small-discrete} in the sense that $!_M \colon M \to \tm$ is an $\s$-small DOF.
  This is closely related to the fact that any locally cartesian closed category $\cC$ is cocomplete when regarded as an indexed category over itself \cite[\S B1.4,~Lemma~1.4.7]{johnstone-elephant-vol-1}.

  By contrast, the notion of cocompleteness from \cite[Definition~3.8~and~Corollary~5.4]{weber-2-toposes} requires that such colimits exist whenever $M$ is \emph{$\s$-essentially-small} in the sense that both $M$ and the presheaf object $\wh{M}$ are $\s$-locally-small (see \S\ref{subsec:historical}).
  Though $\s$-small-discrete objects are clearly $\s$-locally-small, we do not know if they are $\s$-essentially-small.

  As a final remark, we should mention that the motivation for this notion of $\s$-essentially-small---namely, the proof in \cite{freyd-street-size-of-categories} that a category $\cC$ is essentially small as soon as $\cC$ and $\Set^{\cC^\op}$ are locally small---is, at least as it is given, not intuitionistically valid.
  Thus, from a foundational perspective, this may not be the appropriate internalization of the notion of smallness to a 2-topos.

  Here is an alternative characterization of essentially small objects which can be expressed in a general 2-topos and which \emph{is} intuitionistically valid:
  a category $\cC$ is essentially small if and only if it is locally small and there exists an essentially injective (i.e., injective on isomorphism classes of objects) functor $\cC^\iso \to X$ with $X$ small discrete---or, equivalently, a functor $\cC^\iso \to X$ which is full and essentially surjective.
  The latter notion can be defined for a morphism in an arbitrary 2-category in terms of the (full essentially surjective, faithful) factorization system, see \cite[Example~7.9]{dupont-vitale-factorization-systems}.
  (Note that there is another, simpler characterization of essential smallness of $\cC$, which is that it admits an essential surjection from a small discrete category---but, surprisingly, this equivalence also does not seem to be intuitionistically valid, even if ``discrete category'' is replaced by ``setoid''.)

  This notion of smallness, and the resulting notion of cocompleteness, seems worthy of further investigation.
  For example, it is an interesting question whether a plentiful generic DOF $\s$ is always cocomplete in this sense.
\end{rmk}

\printbibliography

@article {bird-kelly-power-street-flexible,
  AUTHOR       = {Bird, G. J. and Kelly, G. M. and Power, A. J. and Street, R.  H.},
  TITLE        = {Flexible limits for {$2$}-categories},
  JOURNAL      = {J. Pure Appl. Algebra},
  FJOURNAL     = {Journal of Pure and Applied Algebra},
  VOLUME       = {61},
  YEAR         = {1989},
  NUMBER       = {1},
  PAGES        = {1--27},
  ISSN         = {0022-4049},
  MRCLASS      = {18D05},
  MRNUMBER     = {1023741},
  MRREVIEWER   = {Timothy Porter},
  DOI          = {10.1016/0022-4049(89)90065-0},
  URL          = {https://doi.org/10.1016/0022-4049(89)90065-0},
}

@article {bourke-accessible-aspects,
  AUTHOR       = {Bourke, John},
  TITLE        = {Accessible aspects of 2-category theory},
  JOURNAL      = {J. Pure Appl. Algebra},
  FJOURNAL     = {Journal of Pure and Applied Algebra},
  VOLUME       = {225},
  YEAR         = {2021},
  NUMBER       = {3},
  PAGES        = {Paper No. 106519, 43},
  ISSN         = {0022-4049},
  MRCLASS      = {18C35 (18N10)},
  MRNUMBER     = {4137713},
  MRREVIEWER   = {Hugo Luiz Mariano},
  DOI          = {10.1016/j.jpaa.2020.106519},
  URL          = {https://doi.org/10.1016/j.jpaa.2020.106519},
}

@article {bourke-garner-2-reg-ex,
  AUTHOR       = {Bourke, John and Garner, Richard},
  TITLE        = {Two-dimensional regularity and exactness},
  JOURNAL      = {J. Pure Appl. Algebra},
  FJOURNAL     = {Journal of Pure and Applied Algebra},
  VOLUME       = {218},
  YEAR         = {2014},
  NUMBER       = {7},
  PAGES        = {1346--1371},
  ISSN         = {0022-4049},
  MRCLASS      = {18F10 (18A32 18D05)},
  MRNUMBER     = {3168499},
  MRREVIEWER   = {R. H. Street},
  DOI          = {10.1016/j.jpaa.2013.11.021},
  URL          = {https://doi.org/10.1016/j.jpaa.2013.11.021},
}

@article {bourke-thesis,
  TITLE        = {Codescent objects in 2-dimensional universal algebra},
  AUTHOR       = {Bourke, John},
  YEAR         = {2010},
  NOTE         = {Thesis (Ph.D.)--University of Sydney},
  URL          = {https://web.archive.org/web/20240416043222/https://www.math.muni.cz/~bourkej/papers/JohnBThesis.pdf}
}

@article {carboni-et-al-modulated-bicats,
  AUTHOR       = {Carboni, Aurelio and Johnson, Scott and Street, Ross and Verity, Dominic},
  TITLE        = {Modulated bicategories},
  JOURNAL      = {J. Pure Appl. Algebra},
  FJOURNAL     = {Journal of Pure and Applied Algebra},
  VOLUME       = {94},
  YEAR         = {1994},
  NUMBER       = {3},
  PAGES        = {229--282},
  ISSN         = {0022-4049},
  MRCLASS      = {18D05},
  MRNUMBER     = {1285544},
  DOI          = {10.1016/0022-4049(94)90009-4},
  URL          = {https://doi.org/10.1016/0022-4049(94)90009-4},
}

@article {dupont-vitale-factorization-systems,
    AUTHOR = {Dupont, M. and Vitale, E. M.},
     TITLE = {Proper factorization systems in 2-categories},
   JOURNAL = {J. Pure Appl. Algebra},
  FJOURNAL = {Journal of Pure and Applied Algebra},
    VOLUME = {179},
      YEAR = {2003},
    NUMBER = {1-2},
     PAGES = {65--86},
      ISSN = {0022-4049},
   MRCLASS = {18A20 (18A32 18D05)},
  MRNUMBER = {1957815},
       DOI = {10.1016/S0022-4049(02)00244-X},
       URL = {https://doi.org/10.1016/S0022-4049(02)00244-X},
}

@article {freyd-street-size-of-categories,
  AUTHOR       = {Freyd, Peter and Street, Ross},
  TITLE        = {On the size of categories},
  JOURNAL      = {Theory Appl. Categ.},
  FJOURNAL     = {Theory and Applications of Categories},
  VOLUME       = {1},
  YEAR         = {1995},
  PAGES        = {No. 9, 174--178},
  MRCLASS      = {18A25},
  MRNUMBER     = {1363007},
  MRREVIEWER   = {S. B. Niefield},
}

@book {gray-formal-category-theory,
  AUTHOR       = {Gray, John W.},
  TITLE        = {Formal category theory: adjointness for 2-categories},
  SERIES       = {Lecture Notes in Mathematics, Vol. 391},
  PUBLISHER    = {Springer-Verlag, Berlin-New York},
  YEAR         = {1974},
  PAGES        = {xii+282},
  MRCLASS      = {18DXX},
  MRNUMBER     = {371990},
  MRREVIEWER   = {R. H. Street},
}

@incollection {gray-the-categorical-comprehension-scheme,
  AUTHOR       = {Gray, John W.},
  TITLE        = {The categorical comprehension scheme},
  BOOKTITLE    = {Category {T}heory, {H}omology {T}heory and their {A}pplications, {III} ({B}attelle {I}nstitute {C}onference, {S}eattle, {W}ash., 1968, {V}ol. {T}hree)},
  SERIES       = {Lecture Notes in Math., No. 99},
  PAGES        = {242--312},
  PUBLISHER    = {Springer, Berlin-New York},
  YEAR         = {1969},
  MRCLASS      = {18.10},
  MRNUMBER     = {249483},
  MRREVIEWER   = {D. Pumpl\"{u}n},
}

@misc {helfer-paradoxes,
  AUTHOR       = {Helfer, Joseph},
  TITLE        = {Set-theoretic universes and paradoxes in 2-topoi},
  NOTE         = {In preparation},
}

@article {helfer-sentai,
  AUTHOR       = {Helfer, Joseph},
  TITLE        = {Homotopies in {G}rothendieck fibrations},
  JOURNAL      = {Theory Appl. Categ.},
  FJOURNAL     = {Theory and Applications of Categories},
  VOLUME       = {35},
  YEAR         = {2020},
  PAGES        = {Paper No. 35, 1312--1378},
  MRCLASS      = {18D30 (18N10 55U35)},
  MRNUMBER     = {4130744},
}

@article{hughes-miranda-colimits,
  TITLE        = {Colimits of internal categories},
  AUTHOR       = {Calum Hughes and Adrian Miranda},
  YEAR         = {2025},
  EPRINT       = {2501.17769},
  ARCHIVEPREFIX= {arXiv},
  PRIMARYCLASS = {math.CT},
  URL          = {https://arxiv.org/abs/2501.17769},
}

@article {hughes-miranda-et2cc,
  TITLE        = {The elementary theory of the 2-category of small categories},
  AUTHOR       = {Calum Hughes and Adrian Miranda},
  YEAR         = {2024},
  EPRINT       = {2403.03647},
  ARCHIVEPREFIX= {arXiv},
  PRIMARYCLASS = {math.CT},
  URL          = {https://arxiv.org/abs/2403.03647},
}

@book {johnstone-elephant-vol-1,
  AUTHOR       = {Johnstone, Peter T.},
  TITLE        = {Sketches of an elephant: a topos theory compendium. {V}ol. 1},
  SERIES       = {Oxford Logic Guides},
  VOLUME       = {43},
  PUBLISHER    = {The Clarendon Press, Oxford University Press, New York},
  YEAR         = {2002},
  PAGES        = {xxii+468+71},
  ISBN         = {0-19-853425-6},
  MRCLASS      = {18B25 (18-02)},
  MRNUMBER     = {1953060},
  MRREVIEWER   = {Colin McLarty},
}

@article {johnstone-fibrations,
  AUTHOR       = {Johnstone, P. T.},
  TITLE        = {Fibrations and partial products in a {$2$}-category},
  JOURNAL      = {Appl. Categ. Structures},
  FJOURNAL     = {Applied Categorical Structures. A Journal Devoted to Applications of Categorical Methods in Algebra, Analysis, Order, Topology and Computer Science},
  VOLUME       = {1},
  YEAR         = {1993},
  NUMBER       = {2},
  PAGES        = {141--179},
  ISSN         = {0927-2852},
  MRCLASS      = {18D05 (18A30 18B25 18D30)},
  MRNUMBER     = {1245798},
  MRREVIEWER   = {Edmund Robinson},
  DOI          = {10.1007/BF00880041},
  URL          = {https://doi.org/10.1007/BF00880041},
}

@incollection {johnstone-wraith-alg-theories-in-toposes,
  AUTHOR       = {Johnstone, Peter T. and Wraith, Gavin C.},
  TITLE        = {Algebraic theories in toposes},
  BOOKTITLE    = {Indexed categories and their applications},
  SERIES       = {Lecture Notes in Math.},
  VOLUME       = {661},
  PAGES        = {141--242},
  PUBLISHER    = {Springer, Berlin-New York},
  YEAR         = {1978},
  MRCLASS      = {18B25 (18C10)},
  MRNUMBER     = {514195},
}

@book {jonshon-yau-2-categories,
  AUTHOR       = {Johnson, Niles and Yau, Donald},
  TITLE        = {2-dimensional categories},
  PUBLISHER    = {Oxford University Press, Oxford},
  YEAR         = {2021},
  PAGES        = {xix+615},
  ISBN         = {978-0-19-887137-8},
  MRCLASS      = {18-02 (18N10 18N15 18N20)},
  MRNUMBER     = {4261588},
  MRREVIEWER   = {Robert Laugwitz},
  DOI          = {10.1093/oso/9780198871378.001.0001},
  URL          = {https://doi.org/10.1093/oso/9780198871378.001.0001},
}

@book {joyal-moerdijk-ast,
  AUTHOR       = {Joyal, A. and Moerdijk, I.},
  TITLE        = {Algebraic set theory},
  SERIES       = {London Mathematical Society Lecture Note Series},
  VOLUME       = {220},
  PUBLISHER    = {Cambridge University Press, Cambridge},
  YEAR         = {1995},
  PAGES        = {viii+123},
  ISBN         = {0-521-55830-1},
  MRCLASS      = {03E70 (03-02 03C90 03F55 03G25 04A10)},
  MRNUMBER     = {1368403},
  MRREVIEWER   = {Ioan Tofan},
  DOI          = {10.1017/CBO9780511752483},
  URL          = {https://doi.org/10.1017/CBO9780511752483},
}

@article {joyal-street-pullbacks,
  AUTHOR       = {Joyal, Andr\'{e} and Street, Ross},
  TITLE        = {Pullbacks equivalent to pseudopullbacks},
  JOURNAL      = {Cahiers Topologie G\'{e}om. Diff\'{e}rentielle Cat\'{e}g.},
  FJOURNAL     = {Cahiers de Topologie et G\'{e}om\'{e}trie Diff\'{e}rentielle Cat\'{e}goriques},
  VOLUME       = {34},
  YEAR         = {1993},
  NUMBER       = {2},
  PAGES        = {153--156},
  ISSN         = {0008-0004},
  MRCLASS      = {18A40},
  MRNUMBER     = {1223657},
}

@article {kelly-2-cat-limits,
  AUTHOR       = {Kelly, G. M.},
  TITLE        = {Elementary observations on {$2$}-categorical limits},
  JOURNAL      = {Bull. Austral. Math. Soc.},
  FJOURNAL     = {Bulletin of the Australian Mathematical Society},
  VOLUME       = {39},
  YEAR         = {1989},
  NUMBER       = {2},
  PAGES        = {301--317},
  ISSN         = {0004-9727},
  MRCLASS      = {18D05},
  MRNUMBER     = {998024},
  MRREVIEWER   = {Kimmo I. Rosenthal},
  DOI          = {10.1017/S0004972700002781},
  URL          = {https://doi.org/10.1017/S0004972700002781},
}

@article {kelly-basic-concepts,
  AUTHOR       = {Kelly, G. M.},
  TITLE        = {Basic concepts of enriched category theory},
  NOTE         = {Reprint of the 1982 original [Cambridge Univ. Press, Cambridge; MR0651714]},
  JOURNAL      = {Repr. Theory Appl. Categ.},
  FJOURNAL     = {Reprints in Theory and Applications of Categories},
  NUMBER       = {10},
  YEAR         = {2005},
  PAGES        = {vi+137},
  MRCLASS      = {18-02 (00B60 18D10 18D20)},
  MRNUMBER     = {2177301},
}

@article {kelly-street-elements,
  AUTHOR       = {Kelly, G. M. and Street, Ross},
  TITLE        = {Review of the elements of 2-categories},
  BOOKTITLE    = {Category {S}eminar ({P}roc. {S}em., {S}ydney, 1972/1973)},
  SERIES       = {Lecture Notes in Math., Vol. 420},
  PAGES        = {75--103},
  PUBLISHER    = {Springer, Berlin-New York},
  YEAR         = {1974},
  MRCLASS      = {18D05},
  MRNUMBER     = {357542},
  MRREVIEWER   = {E. G. Manes},
}

@article {lawvere-cat-of-cats,
  AUTHOR       = {Lawvere, F. William},
  TITLE        = {The category of categories as a foundation for mathematics},
  BOOKTITLE    = {Proc. {C}onf. {C}ategorical {A}lgebra ({L}a {J}olla, {C}alif., 1965)},
  PAGES        = {1--20},
  PUBLISHER    = {Springer, New York},
  YEAR         = {1966},
  MRCLASS      = {02.00},
  MRNUMBER     = {0207517},
  MRREVIEWER   = {J. R. Isbell},
}

@article {lawvere-etcs,
  AUTHOR       = {Lawvere, F. William},
  TITLE        = {An elementary theory of the category of sets (long version) with commentary (2005 reprint)},
  NOTE         = {Reprinted and expanded from Proc. Nat. Acad. Sci. U.S.A.  {{\bf{5}}2} (1964) [MR0172807], With comments by the author and Colin McLarty},
  JOURNAL      = {Repr. Theory Appl. Categ.},
  FJOURNAL     = {Reprints in Theory and Applications of Categories},
  NUMBER       = {11},
  YEAR         = {1964},
  PAGES        = {1--35},
  MRCLASS      = {03A05 (00B60 03E99 18B05)},
  MRNUMBER     = {2177727},
}

@book {lurie-htt,
  AUTHOR       = {Lurie, Jacob},
  TITLE        = {Higher topos theory},
  SERIES       = {Annals of Mathematics Studies},
  VOLUME       = {170},
  PUBLISHER    = {Princeton University Press, Princeton, NJ},
  YEAR         = {2009},
  PAGES        = {xviii+925},
  ISBN         = {978-0-691-14049-0},
  MRCLASS      = {18-02 (18B25 18E35 18G30 18G55 55U40)},
  MRNUMBER     = {2522659},
  MRREVIEWER   = {Mark\ Hovey},
  DOI          = {10.1515/9781400830558},
  URL          = {https://doi.org/10.1515/9781400830558},
}

@book {mac-lane-moerdijk,
  AUTHOR       = {Mac Lane, Saunders and Moerdijk, Ieke},
  TITLE        = {Sheaves in geometry and logic},
  SERIES       = {Universitext},
  NOTE         = {A first introduction to topos theory, Corrected reprint of the 1992 edition},
  PUBLISHER    = {Springer-Verlag, New York},
  YEAR         = {1994},
  PAGES        = {xii+629},
  ISBN         = {0-387-97710-4},
  MRCLASS      = {03G30 (18B25 54B40)},
  MRNUMBER     = {1300636},
  MRREVIEWER   = {M. Makkai},
}

@book {maclane-categories,
  AUTHOR       = {Mac Lane, Saunders},
  TITLE        = {Categories for the working mathematician},
  SERIES       = {Graduate Texts in Mathematics},
  VOLUME       = {5},
  EDITION      = {Second},
  PUBLISHER    = {Springer-Verlag, New York},
  YEAR         = {1998},
  PAGES        = {xii+314},
  ISBN         = {0-387-98403-8},
  MRCLASS      = {18-02},
  MRNUMBER     = {1712872},
}

@article {makkai-avoiding-choice,
  AUTHOR       = {Makkai, M.},
  TITLE        = {Avoiding the axiom of choice in general category theory},
  JOURNAL      = {J. Pure Appl. Algebra},
  FJOURNAL     = {Journal of Pure and Applied Algebra},
  VOLUME       = {108},
  YEAR         = {1996},
  NUMBER       = {2},
  PAGES        = {109--173},
  ISSN         = {0022-4049},
  MRCLASS      = {18A15 (04A25)},
  MRNUMBER     = {1382246},
  MRREVIEWER   = {Andreas Blass},
  DOI          = {10.1016/0022-4049(95)00029-1},
  URL          = {https://doi.org/10.1016/0022-4049(95)00029-1},
}

@incollection {makkai-categorical-foundation,
  AUTHOR       = {Makkai, M.},
  TITLE        = {Towards a categorical foundation of mathematics},
  BOOKTITLE    = {Logic {C}olloquium '95 ({H}aifa)},
  SERIES       = {Lecture Notes Logic},
  VOLUME       = {11},
  PAGES        = {153--190},
  PUBLISHER    = {Springer, Berlin},
  YEAR         = {1998},
  MRCLASS      = {03B15 (00A30 03A05 03G30 18A15 18D05)},
  MRNUMBER     = {1678360},
  MRREVIEWER   = {Peter Johnstone},
  DOI          = {10.1007/978-3-662-22108-2\_11},
  URL          = {https://doi.org/10.1007/978-3-662-22108-2_11},
}

@article {makkai-duality-and-definability,
  AUTHOR       = {Makkai, Michael},
  TITLE        = {Duality and definability in first order logic},
  JOURNAL      = {Mem. Amer. Math. Soc.},
  FJOURNAL     = {Memoirs of the American Mathematical Society},
  VOLUME       = {105},
  YEAR         = {1993},
  NUMBER       = {503},
  PAGES        = {x+106},
  ISSN         = {0065-9266},
  MRCLASS      = {03G30 (03-02 03C20 03C40 18B25)},
  MRNUMBER     = {1164835},
  MRREVIEWER   = {P. \v{S}t\v{e}p\'{a}nek},
  DOI          = {10.1090/memo/0503},
  URL          = {https://doi.org/10.1090/memo/0503},
}

@article {makkai-generalized-sketches-2,
  AUTHOR       = {Makkai, M.},
  TITLE        = {Generalized sketches as a framework for completeness theorems.  {II}},
  JOURNAL      = {J. Pure Appl. Algebra},
  FJOURNAL     = {Journal of Pure and Applied Algebra},
  VOLUME       = {115},
  YEAR         = {1997},
  NUMBER       = {2},
  PAGES        = {179--212},
  ISSN         = {0022-4049},
  MRCLASS      = {03G30 (03C95 03F03 18C10)},
  MRNUMBER     = {1431160},
  MRREVIEWER   = {Anna Labella},
  DOI          = {10.1016/S0022-4049(96)00008-4},
  URL          = {https://doi.org/10.1016/S0022-4049(96)00008-4},
}

@book {makkai-pare-accessible,
  AUTHOR       = {Makkai, Michael and Par\'{e}, Robert},
  TITLE        = {Accessible categories: the foundations of categorical model theory},
  SERIES       = {Contemporary Mathematics},
  VOLUME       = {104},
  PUBLISHER    = {American Mathematical Society, Providence, RI},
  YEAR         = {1989},
  PAGES        = {viii+176},
  ISBN         = {0-8218-5111-X},
  MRCLASS      = {03G30 (03C75 03C95 18B25 18C10 18D05 18F10)},
  MRNUMBER     = {1031717},
  MRREVIEWER   = {Horst Reichel},
  DOI          = {10.1090/conm/104},
  URL          = {https://doi.org/10.1090/conm/104},
}

@incollection {makkai-ultraproducts,
  AUTHOR       = {Makkai, M.},
  TITLE        = {Ultraproducts and categorical logic},
  BOOKTITLE    = {Methods in mathematical logic ({C}aracas, 1983)},
  SERIES       = {Lecture Notes in Math.},
  VOLUME       = {1130},
  PAGES        = {222--309},
  PUBLISHER    = {Springer, Berlin},
  YEAR         = {1985},
  MRCLASS      = {03G30 (03C20)},
  MRNUMBER     = {799044},
  DOI          = {10.1007/BFb0075314},
  URL          = {https://doi.org/10.1007/BFb0075314},
}

@article {power-robinson-pie-limits,
  AUTHOR       = {Power, John and Robinson, Edmund},
  TITLE        = {A characterization of pie limits},
  JOURNAL      = {Math. Proc. Cambridge Philos. Soc.},
  FJOURNAL     = {Mathematical Proceedings of the Cambridge Philosophical Society},
  VOLUME       = {110},
  YEAR         = {1991},
  NUMBER       = {1},
  PAGES        = {33--47},
  ISSN         = {0305-0041},
  MRCLASS      = {18D05},
  MRNUMBER     = {1104599},
  MRREVIEWER   = {R. H. Street},
  DOI          = {10.1017/S0305004100070092},
  URL          = {https://doi.org/10.1017/S0305004100070092},
}

@article {rezk-homotopy-of-homotopies,
  AUTHOR       = {Rezk, Charles},
  TITLE        = {A model for the homotopy theory of homotopy theory},
  JOURNAL      = {Trans. Amer. Math. Soc.},
  FJOURNAL     = {Transactions of the American Mathematical Society},
  VOLUME       = {353},
  YEAR         = {2001},
  NUMBER       = {3},
  PAGES        = {973--1007},
  ISSN         = {0002-9947},
  MRCLASS      = {55U35 (18G30)},
  MRNUMBER     = {1804411},
  MRREVIEWER   = {Brooke E. Shipley},
  DOI          = {10.1090/S0002-9947-00-02653-2},
  URL          = {https://doi.org/10.1090/S0002-9947-00-02653-2},
}

@book {sga-4-1,
  EDITOR       = {Artin, M. AND Grothendieck, A. AND Verdier, J. L.},
  TITLE        = {Th\'{e}orie des topos et cohomologie \'{e}tale des sch\'{e}mas. {T}ome 1: {T}h\'{e}orie des topos},
  SERIES       = {Lecture Notes in Mathematics, Vol. 269},
  NOTE         = {S\'{e}minaire de G\'{e}om\'{e}trie Alg\'{e}brique du Bois-Marie 1963--1964 (SGA 4), Dirig\'{e} par M. Artin, A. Grothendieck, et J. L. Verdier. Avec la collaboration de N. Bourbaki, P. Deligne et B. Saint-Donat},
  PUBLISHER    = {Springer-Verlag, Berlin-New York},
  YEAR         = {1972},
  PAGES        = {xix+525},
  MRCLASS      = {14-06},
  MRNUMBER     = {354652},
}

@misc {shulman-nlab-page,
  AUTHOR       = {Shulman, Mike},
  YEAR         = {2012},
  TITLE        = {Personal page on the nLab},
  HOWPUBLISHED = {\url{https://web.archive.org/web/20231226195819/https://ncatlab.org/michaelshulman/show/HomePage}}
}

@article {street-bicategories-of-stacks,
  AUTHOR       = {Street, Ross},
  TITLE        = {Characterizations of bicategories of stacks},
  BOOKTITLE    = {Category theory ({G}ummersbach, 1981)},
  SERIES       = {Lecture Notes in Math.},
  VOLUME       = {962},
  PAGES        = {282--291},
  PUBLISHER    = {Springer, Berlin-New York},
  YEAR         = {1982},
  MRCLASS      = {18D05 (18F20)},
  MRNUMBER     = {682967},
}

@article {street-cosmoi-of-internal-cats,
  AUTHOR       = {Street, Ross},
  TITLE        = {Cosmoi of internal categories},
  JOURNAL      = {Trans. Amer. Math. Soc.},
  FJOURNAL     = {Transactions of the American Mathematical Society},
  VOLUME       = {258},
  YEAR         = {1980},
  NUMBER       = {2},
  PAGES        = {271--318},
  ISSN         = {0002-9947},
  MRCLASS      = {18D35 (18C10 18F20)},
  MRNUMBER     = {558176},
  MRREVIEWER   = {G. M. Kelly},
  DOI          = {10.2307/1998059},
  URL          = {https://doi.org/10.2307/1998059},
}

@inproceedings {street-elementary-cosmoi-i,
  AUTHOR       = {Street, Ross},
  TITLE        = {Elementary cosmoi. {I}},
  BOOKTITLE    = {Category {S}eminar ({P}roc. {S}em., {S}ydney, 1972/1973)},
  SERIES       = {Lecture Notes in Math., Vol. 420},
  PAGES        = {134--180},
  PUBLISHER    = {Springer, Berlin-New York},
  YEAR         = {1974},
  MRCLASS      = {18D99},
  MRNUMBER     = {354813},
  MRREVIEWER   = {H. Gonshor},
}

@article {street-fibrations-in-bicategories,
  AUTHOR       = {Street, Ross},
  TITLE        = {Fibrations in bicategories},
  JOURNAL      = {Cahiers Topologie G\'{e}om. Diff\'{e}rentielle},
  FJOURNAL     = {Cahiers de Topologie et G\'{e}om\'{e}trie Diff\'{e}rentielle},
  VOLUME       = {21},
  YEAR         = {1980},
  NUMBER       = {2},
  PAGES        = {111--160},
  ISSN         = {0008-0004},
  MRCLASS      = {18D05 (18D30)},
  MRNUMBER     = {574662},
  MRREVIEWER   = {Timothy Porter},
}

@article {street-sketches-in-bicats-arxiv,
  TITLE        = {Pointwise extensions and sketches in bicategories},
  AUTHOR       = {Ross Street},
  YEAR         = {2014},
  EPRINT       = {1409.6427},
  ARCHIVEPREFIX= {arXiv},
  PRIMARYCLASS = {math.CT},
  URL          = {https://arxiv.org/abs/1409.6427},
}

@article {street-two-sheaf,
  AUTHOR       = {Street, Ross},
  TITLE        = {Two-dimensional sheaf theory},
  JOURNAL      = {J. Pure Appl. Algebra},
  FJOURNAL     = {Journal of Pure and Applied Algebra},
  VOLUME       = {23},
  YEAR         = {1982},
  NUMBER       = {3},
  PAGES        = {251--270},
  ISSN         = {0022-4049},
  MRCLASS      = {18F20 (18D05)},
  MRNUMBER     = {644277},
  MRREVIEWER   = {M. Barr},
  DOI          = {10.1016/0022-4049(82)90101-3},
  URL          = {https://doi.org/10.1016/0022-4049(82)90101-3},
}

@article {street-walters-yoneda-structure,
  AUTHOR       = {Street, Ross and Walters, Robert},
  TITLE        = {Yoneda structures on 2-categories},
  JOURNAL      = {J. Algebra},
  FJOURNAL     = {Journal of Algebra},
  VOLUME       = {50},
  YEAR         = {1978},
  NUMBER       = {2},
  PAGES        = {350--379},
  ISSN         = {0021-8693},
  MRCLASS      = {18D05},
  MRNUMBER     = {463261},
  MRREVIEWER   = {Sym\'{e}on Bozapalid\`es},
  DOI          = {10.1016/0021-8693(78)90160-6},
  URL          = {https://doi.org/10.1016/0021-8693(78)90160-6},
}

@article {weber-2-toposes,
  AUTHOR       = {Weber, Mark},
  TITLE        = {Yoneda structures from 2-toposes},
  JOURNAL      = {Appl. Categ. Structures},
  FJOURNAL     = {Applied Categorical Structures. A Journal Devoted to Applications of Categorical Methods in Algebra, Analysis, Order, Topology and Computer Science},
  VOLUME       = {15},
  YEAR         = {2007},
  NUMBER       = {3},
  PAGES        = {259--323},
  ISSN         = {0927-2852},
  MRCLASS      = {18A05 (18A15 18B25 18D05)},
  MRNUMBER     = {2320763},
  MRREVIEWER   = {Stephen Lack},
  DOI          = {10.1007/s10485-007-9079-2},
  URL          = {https://doi.org/10.1007/s10485-007-9079-2},
}

@misc {weber-ncafe-comment,
  AUTHOR       = {Weber, Mark},
  YEAR         = {2008},
  TITLE        = {Discussion on the n-Category Café},
  HOWPUBLISHED = {\url{https://web.archive.org/web/20221022042140/https://golem.ph.utexas.edu/category/2008/01/2toposes.html\#c014444}}
}
\end{document}